\documentclass[a4paper,fleqn,leqno]{article}
\usepackage[utf8]{inputenc}
\usepackage[english]{babel}
\usepackage{latexsym,amssymb,amsthm,xcolor}
\usepackage[centertags]{amsmath}
\numberwithin{equation}{section}
\usepackage{mathtools}
\usepackage{scrtime}
\usepackage{graphicx}
\usepackage{bbm}
\usepackage{framed}
\usepackage{fancyhdr}
\usepackage{appendix}                    

\usepackage{authblk}

\usepackage{hyperref}
\usepackage{accents}           

\usepackage[
headheight=15pt,top=1.5in,bottom=1in,right=1in,left=1in]{geometry}
\usepackage{scrtime}          
 \newcommand{\R}{\mathbb{R}}
 \newcommand{\N}{\mathbb{N}}
 \newcommand{\M}{\mathbb{M}}
 \newcommand{\Z}{\mathbb{Z}}

 \newcommand{\E}{\mathbb{E}}
 \renewcommand{\P}{\mathbb{P}}
\newcommand{\1}{\mathbbm{1}}

\newcommand{\dr}{\underline{\underline{r}}}
\newcommand{\uuxi}{\underline{\underline{\xi}}}
\newcommand{\uur}{\underline{\underline{r}}}
\newcommand{\uxi}{\underline{\xi}}
\newcommand{\uF}{\underline{F}}

\def\CB{\mathcal{B}}

\def\CE{\mathcal{E}}
\def\CF{\mathcal{F}}

\def\CM{\mathcal{M}}

\def\CL{\mathcal{L}}

\def\D{\mathbb{D}}
\def\E{\mathbb{E}}

\def\N{\mathbb{N}}
\def\P{\mathbb{P}}

\def\R{\mathbb{R}}

\def\U{\mathbb{U}}
\def\Z{\mathbb{Z}}

\DeclareMathSymbol{\varNu}{\mathord}{letters}{78}

\newcommand{\wh}{\widehat}
\newcommand{\wt}{\widetilde}
\newcommand{\be}{\begin{equation}}
\newcommand{\ee}{\end{equation}}
\newcommand{\bea}{\begin{eqnarray}}
\newcommand{\eea}{\end{eqnarray}}
\newcommand{\bean}{\begin{eqnarray*}}
\newcommand{\eean}{\end{eqnarray*}}
\newcommand{\noi}{\noindent}
\newcommand{\intl}{\int\limits}

\newcommand{\suml}{\sum\limits}

\newtheorem{theorem}{Theorem}[section]
\newtheorem{proposition}[theorem]{Proposition}
\newtheorem{corollary}[theorem]{Corollary}
\newtheorem{lemma}[theorem]{Lemma}
\newtheorem{definition}[theorem]{Definition}

\newtheorem{assumption}[theorem]{Assumption}

\newtheorem{xx}{\bf xxx}

\newtheorem{yy}{\bf yyy}

\newtheorem{zz}{\bf zzz}

\makeatletter                                           
\def\th@newremark{\th@remark\thm@headfont{\bfseries}}   
\makeatletter                                           
\theoremstyle{newremark}
\newtheorem{remark}[theorem]{Remark}

\newtheorem{example}[theorem]{Example}

\usepackage{enumerate}

\renewcommand{\epsilon}{\varepsilon}
\newcommand{\dx}{\mathrm{d}}                                   
\newcommand{\independent}{\perp\!\!\!\perp}                    
\newcommand{\eqd}{\overset{\textup{d}}{=}}                     

\allowdisplaybreaks[4]

\newcommand{\mcB}{\mathcal{B}}

\newcommand{\mcD}{\mathcal{D}}

\newcommand{\mcF}{\mathcal{F}}
\newcommand{\mcG}{\mathcal{G}}
\newcommand{\mcH}{\mathcal{H}}

\newcommand{\mcM}{\mathcal{M}}

\newcommand{\mcP}{\mathcal{P}}

\newcommand{\mfC}{\mathfrak{C}}
\newcommand{\mfk}{\mathfrak{k}}
\newcommand{\mfK}{\mathfrak{K}}

\newcommand{\mfR}{\mathfrak{R}}

\newcommand{\mfu}{\mathfrak{u}}
\newcommand{\mfU}{\mathfrak{U}}
\newcommand{\mfv}{\mathfrak{v}}

\newcommand{\mfw}{\mathfrak{w}}

\newcommand{\bbD}{\mathbb{D}}
\newcommand{\bbE}{\mathbb{E}}

\newcommand{\bbK}{\mathbb{K}}

\newcommand{\bbN}{\mathbb{N}}

\newcommand{\bbR}{\mathbb{R}}
\newcommand{\bbS}{\mathbb{S}}

\newcommand{\bbU}{\mathbb{U}}

\newcommand{\bbZ}{\mathbb{Z}}

\newcommand{\mbfP}{\mathbf{P}}

\newcommand{\Levy}{L{\'e}vy }

\newcommand{\cadlag}{c\`adl\`ag{ }}
\newcommand{\Ito}{It\^o}
\renewcommand{\S}{S}

\usepackage{color}
\definecolor{darkgreen}{rgb}{0.0,0.6,0.1}  

\usepackage{xr} 
 \usepackage[pdf]{pstricks}

\begin{document}

\title{Branching Processes --- A General Concept}
\author{Andreas Greven$^1$, Thomas Rippl$^2$, Patric Karl Gl\"{o}de$^3$}

\date{{\today}\\
\begin{center}
\end{center}}

\maketitle

\begin{abstract}
The paper has four goals.
First, we want to generalize the classical concept of the branching property so that it becomes true for historical and genealogical processes, where the classical concept fails (here we use the description of genealogies by ($V$-marked) ultrametric measure spaces leading to state spaces $\U$ resp. $\U^V$).
In particular we want to complement the corresponding concept of infinite divisibility developed in \cite{infdiv} for this context.
The processes we consider are always defined by well-posed martingale problems.
The point of the generalized branching property is that the state at times $t+s$ can at any time $t$ be decomposed in a measurable function of the state at time $t$ and an independent part which itself then decomposes in independent copies of the process evolving for time $s$.
Secondly we want to find a corresponding characterization of the generators of branching processes both easy to apply and general enough to cover a wide range of mechanisms and state spaces, this is our first main result.

As a third goal we want to obtain the branching property for some important examples  as the $\U$-valued Feller diffusion respectively $\U^V$-valued super random walk and the historical process on countable geographic spaces, the latter as two examples of a whole zoo of spatial processes we could treat.
The fourth goal is to show the robustness of the method and to get the generalized  branching property for genealogies marked with ancestral path, giving the line of descent moving through the ancestors and space, leading to path-marked ultrametric measure spaces.
These new processes are constructed here and proved to have the generalized branching property, both together our second major result.

We develop an abstract framework covering above situations and questions, leading to a new generator criterion.
The state spaces suitable for historical and genealogical processes are consistent collections of topological  semigroups each enriched with a compatible collection of maps, the truncation maps.
All objects are defined  on the state space of the process.
The method allows to treat every type of population model formulated as solution to a well-posed martingale problem. 
This framework in particular includes processes taking values in the space of marked ultrametric measure spaces and hence allows to treat historical information and genealogies of spatial population models both at once, if genealogies are described this way.
Another example is a multitype population, more specific with genetic types under mutation.
\bigskip

\bigskip

\noi {\bf AMS Subject Classification: 60J80}

\end{abstract}

\medskip

\footnoterule
\noi
\hspace*{0.3cm} {\footnotesize \mbox{}\\
$^{1)}$ Department Mathematik, Universit\"at
Erlangen-N\"urnberg, Cauerstr. 11,
D-91058 Erlangen, Germany, e-Mail: greven@math.fau.de\\
$^{2)}$ Inst.~f.~Math.~Stochastik, Universit\"{a}t G\"ottingen, Goldschmidtstra\ss{}e 7, 37077 G\"ottingen, Germany, trippl@uni-goettingen.de \\
$^{3)}$ Technion-Israel Institute of Technology, Faculty of Industrial Engineering and Management, 32000 Haifa, Israel, gloede@technion.ac.il }

\newpage

\tableofcontents

\section{Introduction}

\paragraph{The problem and goals}
Branching processes are one of the most important and best-studied class of stochastic processes covered in many books, starting with Harris \cite{Harris2002} and Athreya and Ney \cite{AN11} and later including \textit{spatial} models \cite{D93} and \cite{Eth00}.
However new questions and aspects keep arising, 
concerning in particular the evolution of genealogies and histories in such branching models, as we shall see below in detail.
The key point of this work is to define and get criteria for the \emph{"generalized branching property"} in the face of new types of processes describing \emph{genealogies or histories in spatial population models}, where the classical branching property does \textit{not} hold.
Some of these processes are introduced here, others are well known in the literature. 
The reader not familiar with the terms genealogy or history is invited to read Remark~\ref{r.519}.

The focus in the theory of branching processes has now shifted to \emph{historical} and \emph{genealogical information} on the population as objects of interest in its own right.
It is also a tool to understand the behaviour of population sizes or  type frequencies in spatial models, see for example \cite{Neveu86},\cite{Ald1991a}, \cite{NP89}, \cite{LG89}, \cite{AldlGall}, \cite{LG99}, \cite{DLG02}, or in another perspective \cite{BLG00,BLG03,BLG05} to name a few  examples for the use of  genealogies modeled as \emph{labeled trees}.
Here also \emph{multitype} models fit in as marked labeled trees, where individuals have genetic types which undergo mutation.

We focus here on the description of genealogical information more in the spirit of the description via historical processes \cite{DP91} or $\R$-trees \cite{EPW06}, by using here equivalence classes of \emph{(marked) ultrametric measure spaces}.
The latter approach is developed in \cite{GPW09}, \cite{DGP11}, applied to Fleming-Viot type models in \cite{DGP12} and \cite{GPWmp13} and it has been extended to branching in \cite{Gl12,ggr_tvF14,infdiv}.
This particular description of genealogies is chosen in this paper to be able to work with the framework of \emph{well-posed martingale problems} to characterize the models and second to be able to use the concept of infinite divisibility from~\cite{infdiv}.

For this approach using (equivalence classes of) ultrametric measure spaces there is a self-contained and detailed survey, see \cite{DG18evolution}, where in particular one can find examples and references showing more generally the \emph{usefulness of this abstract approach} to \emph{evolving} (in time) \emph{genealogies} of populations, which is based on the construction of the genealogy valued processes via \textit{well-posed martingale problems}.
The latter allows to use the tools which have been developed in \cite{EK86},\cite{D93} in particular for processes relevant in the theory of spatial population models for example see \cite{DG14}, but where genealogical questions arise, which can in the framework of $\U$-valued processes now be tackled.

If we record \emph{genealogies} and \emph{histories} for the time-$t$-population there are two forms in which information about the past may sit in the \emph{present state}, for once in the ancestral relationships of individuals (genealogies) or in the marks as locations or types (histories).
Examples are paths in geographic space or paths in type space, the latter storing information on the genome. 
Both effects may require some care in order to still distill a generalized branching property, since the classical branching property does not hold in that situation. Our examples we treat show the typical obstacles to the classical form of independence in the further evolution of subpopulations with states storing information about the past.

The main goals of this paper are:
\begin{itemize}
\item first to develop for a \emph{generalized} branching property a suitable \emph{abstract framework} and formulation, 
\item second to find and prove an \emph{operator criterion} for a process specified by a well-posed martingale problem to check this property, 
\item thirdly to apply it to \emph{historical processes} and \emph{evolving genealogies} the latter described in a specific way namely as \emph{equivalence classes of ultrametric measure spaces}, respectively their \emph{marked} versions for \emph{spatial} and for \emph{multitype} population models and 
\item fourth but not least, construct \emph{ancestral path marked genealogies}, generalizing and combining both \emph{historical} processes as well as \textit{genealogical} processes and  which we introduce here as new processes and where we use the criterion to verify the generalized branching property.
\end{itemize}
\begin{remark}[The word genealogy]\label{r.519}
As the word genealogy or history gives occasionally rise to misunderstanding we clarify here some points concerning the precise meaning of these words \textit{genealogy - history}.

To clarify terminology consider a binary splitting Galton-Watson process in continuous time.
We can draw the \emph{labeled tree} of all individuals ever alive up to time $t$ (labeled means individuals are identified and distinguished by a name).
Then we observe how this object evolves further in time $t$.
The vertices of the tree are the individuals at their birth time and edges representing the life time are attached, ending or giving rise to two descendants splitting the edge at a new vertex.
This allows to define ancestors and descendants and we define the distance between two points in this tree as the sum of the two cumulative edge lengths back to the most recent common ancestor.
This induces a tree-like labeled metric space representing the labeled genealogy.
If we pass to the equivalence class under isometries we obtain an equivalence class of tree-like metric spaces (we forget the names of individuals).

Embedded in the labeled tree is the set of leaves at height $t$ forming a \emph{labeled ultrametric space}, describing the \emph{labeled genealogy} of the \emph{individuals alive at time $t$} referred to as \emph{current population}.
The ultrametric space is always in one-to-one correspondence with a \emph{unique minimal $\R$-tree} of which it is the set of \textit{leaves} (compare Example~\ref{ex.umspaces}).
Again we might pass to the equivalence class under isometries and forget the names of individuals getting what we refer to as genealogy of the current population.
The genealogy including the \emph{fossils} is then the whole genealogical tree including everybody alive at some times before time $t$. The individuals of our population might have types or locations in some space, where they move according to some stochastic process.
Then we get \textit{marked} genealogies.

In such a genealogy we can recover historical information via the ancestral path, obtained by tracing backward type or location of an individual alive at time $t$ till its birth time, continue tracing the type/location of the father etc.
This way we get a collection of backward path in type or geographic space, the \textit{ancestral path}.
Give each path weight $1$ and obtain a measure on paths, this is the state of the \emph{historical process}  a measure-valued process on the space of paths.

We are interested in the time \emph{evolution} of \emph{genealogies/histories} as the current time $t$ evolves in particular in the genealogy of the population currently alive represented by the equivalence class of ultrametric spaces.

Since we are interested in \emph{large populations} even the infinite population limit, such a genealogy can only be observed by taking \textit{samples} from the current population.
Therefore we add a probability measure on the individuals alive at time $t$ to the genealogy.
This results in a labeled genealogy, a triple $(X,r,\mu)$ with $X$ set of individuals, $r$ the genealogical distance (an ultrametric) and $\mu$ the sampling measure.
The equivalence class under isometries which are measure preserving denoted $[X,r,\mu]$ represents for us the relevant information on the genealogy of the population alive at some time $t$.
The genealogy of all individuals alive ever we call the \textit{fossil process}. This is however not the object of this paper.
To incorporate population sizes fluctuating we consider  also $\mu$ which are finite measures.
We return to this point.
\end{remark}

\paragraph*{Branching property}
Important examples for the typical structure of branching processes are \emph{continuous state branching processes} \cite{LG99}, \cite{Lamb07} and \emph{measure valued branching processes} \cite{D93}, \cite{LJ91} or the \textit{historical Dawson-Watanabe} process~\cite{DP91}. 
The state spaces then are {\em linear} spaces, here $\bbR^d$ and $\mcM_f(E)$, respectively.
Here $E$ is a Polish space \emph{and} $\mcM_f(E)$ is the space of finite measures on $(E,\mcB(E))$. 
All branching processes share the following characteristic {\em branching property}: given the state of the present population, the states of sub-populations descending from different "ancestors" in today's population evolve independently from each other and with identical laws.
This has been formalized so far as follows.

The underlying state 
space $\S$ is a {\em semigroup} meaning it is endowed with the binary operation of addition (of reals or measures, respectively).
Suppose that   $X=(X_t)_{t\geq0}$ is a Markov process with values in $\S$.

The transition probabilities denoted
by $P_t$, that is, $(P_t f)(x) = P_t(x,\cdot)[f]=\bbE[f(X_t)|X_0=x]$ for $f \in \textrm{b}\mcB(S)$, the space of bounded and measurable functions on $\S$.
The semigroup $(P_t)_{t\geq0}$ has the \emph{branching property} if
\begin{equation}\label{e1301141845}
 P_t(x+y,\cdot)[f]= (P_t(x,\cdot)\ast P_t(y,\cdot))[f]\,, \text{ for all } x,y \in \S\text{ and } f\in \textrm{b}\mcB(S),
\end{equation}
with $\ast$ the convolution.
Then $X$ is defined to have the {\em branching property} if $(P_t)_{t\geq0}$ has the branching property.

If $\textrm{b}\mcB(S)$ contains a subset $D$ of functions which are \emph{multiplicative} w.r.t.~the semigroup operation $+$ and \emph{$D$ separates} laws on $S$, 
then \eqref{e1301141845} holds if $P_t(x+y,\cdot)[f] = P_t(x,\cdot)[f] \, P_t(y,\cdot)[f]$ for all $x,y \in \S$ and $f \in D$.
Indeed, in some of the {\em classical cases} the branching property is related to the form of the {\em Laplace transform}.
In the second example above we have $X_t \in \S = \mcM_f(E)$.
With the help of a {\em time evolution operator} 
  $V_{s,t}:\textrm{bp}\mcB(E)\to \textrm{bp}\mcB(E)$, $0\leq s\leq t$, $\textrm{bp}\mcB(E)$ the 
  non-negative, bounded, measurable functions on $E$, 
  the Laplace transform can be represented as (set $\langle \mu,f\rangle=\int f \,\dx\mu$ for every 
  $\mu\in\mcM_f(E)$ and $f\in \textrm{b}\mcB(E)$):
  \begin{equation}\label{e1301}
  \bbE[\exp(-\langle X_t,\phi\rangle)|X_s]=\exp( - \langle X_s,V_{s,t}\phi\rangle)\,, \quad\text{a.s.}
  \end{equation}
Then (\ref{e1301141845}) can be read of from this rhs. 
   Clearly, for $E=\{1,2,\dotsc,d\}$ this includes $\bbR^d$-valued processes.

Then first the question arises whether the processes describing features of the {\em history} or {\em genealogies} of individuals of a ``branching'' population have an analogue property and structure.
The problem being that the \emph{present state} contains information about the state at \emph{past times}.
In other words, can we define an abstract \emph{generalized branching property} dealing with this problem.
This complements the investigation of a concept of  \emph{infinite divisibility} for \emph{genealogical} structures modeled as ultrametric measure spaces which is introduced in \cite{infdiv} which presents a generalized infinite divisibility since the classical one does not hold.

However even if we can define a branching property for the objects, in general it may be difficult to obtain analytical expressions for $P_t(x,\cdot)[f]$ to \emph{check}  in practice whether \eqref{e1301141845} holds.
In particular often an analogue of \eqref{e1301} cannot be found.
Instead of working with the Markov process' transition probabilities $(P_t)_{t \geq 0}$ or the martingale problem itself we may use its \textit{operator $A$}.
Here we hope to derive a statement in the following spirit:
{\em Suppose the generator of a Markov process (or operator of a martingale problem) has a particular form, then the Markov process is a branching process} and similarly for the generalized branching property. 
Here one has the Kurtz criterion in \cite{EK86}.
The point however is that this particular form of the generator criterion needs to be such that it is {\em easily} verifiable.
This is important for \emph{general} state spaces and more complex evolution mechanisms for example for evolving genealogies.
This means in the new cases we want to cover, we have to go \emph{beyond} the \emph{Kurtz criterion} in \cite{EK86} by using an \emph{additional structure}. 
To find such a structure and criterion is our second goal.

Therefore we first {\em extend the concept} of branching processes to more general state spaces and evolution mechanisms with certain \emph{algebraic} and fitting \emph{topological} properties.
Second we find the {\em characterization of operators} for processes which have the generalized branching property.
In particular we will present simple criteria which the generator has to 
satisfy in order for the processes to have the (generalized) branching property. 
One of the criteria is such that it can be checked by explicit calculation for complicated processes {\em if we have checked that our setup applies}, which needs of course some work, since it requires that our process is given by a well-posed martingale problem.  

We briefly explain below first using classical examples how the criterion works in a simpler context where also the classical branching property holds and this is verified by this criterion here.
Then we come to our third goal and check the generalized branching property for the genealogy of \emph{Fellers diffusion model} or the \emph{super random walk}.

We finally come to our \emph{fourth} goal where we use the criterion to get the generalized branching property in a spatial model where every individual alive today is equipped with its \emph{ancestral path}.
This is the line of descent moving geographic or type through space till the own birth time, that of the father etc.
This process we define and characterize here via a \textit{well-posed martingale problem}.
We will argue why this new class of processes with values in marked ultrametric measure spaces,
is in fact a {\em very general class} of population processes in this context of describing genealogies and histories.
This process has some relation to the "snake" of LeGall \cite{LG99} formulated via labeled trees.
We construct our object in Section~\ref{ss.histbranpro} and ~\ref{ss.outquest} combining the approach based on historical processes from \cite{DP91} with the approach based on the ultrametric measure space valued description of genealogies \cite{GPWmp13}, \cite{DGP13} and \cite{GSW}.
The new processes constitute the class of {\em (ancestral path)-marked $\U^V$-valued Feller resp. super random walk processes} and if we model the population \textit{ever} alive we obtain processes also \emph{including fossils}.

The term very general is here meant in the sense that a multitude of processes typically studied in population models can be embedded into this class.
\emph{Therefore we prove here the criterion works in this case} which is then one of the deeper reasons, that we have the branching property for all the embedded Markovian processes often studied in their own right.
As there are functionals typically not one-to-one the proof of the branching property for them requires work (for example getting the Markov property).
We come to this at the end of Section~\ref{s1501141699}.

We cannot cover here everything of interest however.
It would be nice to cover the case of the genealogy of the Dawson-Watanabe process.
Here however at least in the $d \geq 2$ case the \emph{martingale problem formulation} of the process has a different form on a continuum geographic space and preparing this needs some effort, which would take too much space in this paper and will be treated in forthcoming work.
However we can define the $\U^{\R^d}$-valued process as a functional of the historical Dawson-Watanabe process to which we can apply our criterion so that we obtain the branching property of that functional as well.
Similar issues arise also for continuum state branching processes not of the diffusive type.

A corresponding interesting {\em open problem} is to find other examples of processes fitting the abstract algebraic and topological framework, but are \emph{not} arising  from genealogies or histories.

\paragraph{The classical framework and the criterion of Kurtz} \; In classical situations the state space $ S $ has the following features, where we will always use $\S$ as a generic notation for a {\em Polish state space}.
We will for 
now assume that {\em $\S$ is a topological semigroup with operation $+$}. 

Let us recall the following 
observations from \cite{EK86}[Section 4.10]. 
Assume the following properties, first

\begin{equation}\label{1371}
D\subset D' = \{f\in\textrm{b}\mcB(\S):f(w+z)=f(w)f(z), \text{ for all } w,z \in \S \}
\end{equation} 
is a set of multiplicative functions and $A: D \to \textrm{b}\mcB(\S)$ is a linear operator and $X^x$ solves the martingale problem for $(A,D,\delta_x)$ where $x\in \S$.
Furthermore assume that $X^x$ has the (classical) branching property meaning for every $x$ (\ref{e1301141845}) holds. 
Finally let $X^y$ be a solution of the $(A,D,\delta_y)$ martingale problem, $y \in \S$, which is independent of $X^x$. 

Then for any $f \in D$, the process
\begin{equation}\label{e521}
  f(X^y_t+X^y_t)-\int_0^t Af(X^x_s)f(X^y_s)+f(X^x_s)Af(X^y_s)\,\dx s,\, t \geq0
 \end{equation}
defines a \emph{martingale}. This in turn implies for the \emph{operator} $A$ that
\begin{equation}\label{e1301141729}
  Af(x+y)=f(y)Af(x)+f(x)Af(y)\,.
\end{equation}

Conversely, by uniqueness, if (\ref{e1301141729}) holds this implies that $X^{x+y}$ has the 
same distribution as $X^x+X^y$ where $X^x$ and $X^y$ are independent. This is the \emph{branching 
property}. 

Thus, one can indeed prove that                                                
whenever the martingale 
problem for $(A,D,\delta_x)$ is well-posed for all $x$ and $D\subseteq D'$ then the solution 
process has the \emph{branching property iff (\ref{e1301141729}) holds}.

Obviously, for complicated operators $A$ it will not be obvious whether or not (\ref{e1301141729})
is satisfied. Hence we would like to have a criterion for a generator which is easy to check and
which guarantees (\ref{e1301141729}). 
That here the {\em linearity} of the rate may play an important role was discovered and used used for $\R^+$-valued processes in \cite{caballero2009}. 
This approach can be formulated more abstractly as follows.

\paragraph{Towards a new operator criterion}
It is easy to check that (\ref{e1301141729}) is satisfied
if for any $f\in D\subseteq D'$ there exists a {\em semigroup homomorphism} $g_f:\S\to (\bbR,+)$ (typically \emph{depending on $f$})  such that for all $x\in 
\S$,
\begin{equation}\label{e1301141742}
 Af(x)=g_f(x)f(x)\,.
\end{equation}

This works very well for the classical branching processes with values in \textit{linear} spaces (where we get the classical branching porperty), it is however not working in the case of genealogical or historical processes as we shall explain latter on.
Therefore~\eqref{e1301141742} will be the \emph{starting point} for our criterion we develop below, which needs some \emph{new} elements if the state space is \emph{not} anymore a \emph{linear space} as in the classical situation and also the mechanism of the process is of a different nature so that we get only the \textit{generalized} branching property.

\paragraph{The new criterion in some classical cases (Examples)}
Some well-known examples for branching processes show that indeed \eqref{e1301141742} is a good criterion for the classical branching property as well and that $g_f$ can be specified \emph{explicitly}.

\begin{example}\label{ex.572}
First consider the simplest example, the classical continuous time 
{\em critical binary Galton-Watson} process with branching rate $b>0$, meaning we have $exp(b)$-distributed splitting or death times. That is, $E=\bbZ_+$, 
$D=\{\bbN_0\to\bbR, x\mapsto e^{-\lambda 
x}:\lambda>0\}$, $Af(x)=\frac12 bx(f(x+1)+f(x-1)-2f(x))$. The martingale problem for $(A,D)$ is well-posed, see Section 8.3 in \cite{EK86} and for $f \in D$:
\begin{equation}\label{e1302}
 Af(x)=\frac12 b x (f(1)+f(-1)-2f(0)) f(x)\,, \, x \in \N .
\end{equation}
So, we can choose the homomorphism as:
\begin{equation}\label{e605}
g_\lambda(x)=\frac12 bx (f(1)+f(-1)-2f(0))=\frac12 b(e^{-\lambda} + e^\lambda -2)x, x \in \N
\end{equation}
and (\ref{e1301141742}) is satisfied. \qed
\end{example}

\begin{example}\label{ex.583}
Now, consider the class of {\em measure-valued continuous state branching processes} (CSBP). 
(We \textit{specialize} later to the case where the description via a well-posed martingale problem has already been established.)
They are 
spatial analogues of the real-valued CSBP, see \cite{LG99}. More precisely, let 
$E$ be a locally compact metric space and let  $\S = \mcM_f(E)$, the space of finite measures on $E$. We 
denote the process by $X=(X_t)_{t\geq0}$. Intuitively, $X$ locally behaves like a CSBP plus there 
is migration/mutation in $E$ governed by a Feller process with generator $(B,\mcD(B))$, say. The 
branching dynamics is locally determined by the branching mechanism, 
where for each $x\in E$, $\alpha(x)\in\bbR$, $\beta(x)>0$ for all $x\in E$
\begin{equation}\label{e1303}
\psi(x,u)=\alpha(x) u+\beta(x) u^2+\int_{(0,\infty)} (e^{-ru}-1+ru)\,\pi(x, \dx r)\,
\end{equation}
and the measure $\pi(x,\dx y)$ satisfies $\sup_{x\in E} (\int_0^1 r^2 \pi(x,\dx r)+
\int_1^\infty r \pi(x,\dx r))<\infty$. The generator of $X$ then takes the following form,
see \cite{D93}[page 106]:
\begin{equation}\label{ag100} \mbox{with\;} D:=\{F(\mu)=F_\phi(\mu)=f(\langle 
\mu,\phi\rangle):\phi\in \mcD(B)\} \mbox{, choose  } f(x)=e^{-x}
\end{equation} 
and set for arbitrary $\phi \in \mcD(B)$ 
\begin{align}\label{e1304}
 \Omega F_\phi(\mu)
 &=f'(\langle\mu,\phi\rangle)\langle\mu,B\phi\rangle
 +f'(\langle\mu,\phi\rangle)\langle\mu,\alpha\phi\rangle
 +\frac12 f''(\langle\mu,\phi\rangle)\langle\mu,\beta\phi^2\rangle\\
 &\quad\,+\int_E\mu(\dx x)\int_0^\infty n(x,\dx u)[f(\langle\mu,\phi\rangle)+u \phi(x))
 -f(\langle\mu,\phi\rangle)-f'(\langle\mu,\phi\rangle)u\phi(x)]\,.\nonumber
\end{align}
The martingale problem for $(\Omega,D,\delta_\mu)$ is 
well-posed for all $\mu\in\mcM_f(E)$. One 
can easily check that
\begin{equation}\label{e604}
  \Omega F_\phi(\mu)= g_\phi(\mu) F_\phi(\mu),
\end{equation}
with the linear function 
\begin{align}\label{e608}
 g_\phi(\mu)&=-\langle\mu,B\phi\rangle -\langle\mu,\alpha \phi\rangle
 +\langle\mu,\frac12\beta\phi^2\rangle +\int_E\mu(\dx x)\int_0^\infty n(x,\dx u)[f(u\phi(x))-1+u(\phi(x))]\,.
\end{align}
Hence criterion (\ref{e1301141742}) is satisfied. Note that if $E$ is a point this covers the well-known $\bbR$-valued CSBP. \qed
\end{example}

\begin{example}[Shortlist of classical examples]\label{ex.707}
We could add here \emph{spatial} models like \emph{super random walk} or \emph{multitype branching with mutation} and many more as long as we have a characterization of the process as \emph{well-posed martingale problem} of the form given by an operator $A$.
The latter is a restriction, as we will see in the context of the Dawson-Watanabe process which we can only treat with our method so far in $d=1$. 
Given such a characterization the criterion is easily checked again, by explicit calculation left to the reader.
\end{example}

These examples show that the criterion given in (\ref{e1301141742}) is indeed 
{\em easy to verify}. For the classical processes mentioned above the branching property is 
typically proven by verifying (\ref{e1301}), see for example \cite{D93}[Chapter 4] for the 
measure-valued setting and \cite{LG99}[Section II.1] for the real-valued setting. 

However, the Laplace approach turns out to be difficult to apply to {\em more general classes} of 
processes like {\em $\U$-valued} branching processes while the \emph{generator approach introduced above still works} to at least give a generalized branching property.
One reason is that the classical approach of showing the branching property via Laplace transforms (evolution equations) depends on the underlying state space being {\em linear}, or at least a convex cone, 
which is not true for $\U$-valued processes.
It turns out however that the generator approach can be extended and the branching property \emph{generalized} to cover {\em genealogy-valued branching processes} (in ultrametric measure space description) and its spatial versions as well as {\em historical processes}, where the classical property does not hold.

\paragraph{A generalized branching property.} The classical concept of branching processes can be modified and then extended to much more \textit{general mechanism} and with it state spaces, where the branching property in the classical sense \emph{is false} as for example Markov processes of evolving {\em histories} and {\em genealogies} in particular also in the description with marked ultrametric measure spaces. 
The idea of the {\em generalized branching property} in that context is as follows. 

We have to 
incorporate the possibility that sub-populations originating from the population at time $s$ still at a later time $t$ \emph{share some common information about the population} and hence, {\em conditional on the 
information up to time} $s$ {\em fail to be independent} and hence the classical  branching property is false. 

For \textit{example}, consider a Galton-Watson branching 
dynamics on $\N_0$. 
Conditional on the information up to time $s$, the sizes of the subfamilies  at a later time $t$ originating from the individuals alive at time $s$ are independent. However, if we 
incorporate the genealogical relationships into the state space {\em independence fails}. The reason 
is that unless the population at time $s$ consisted of only one individual all individuals at 
time $t$ are connected by ancestral lines going back beyond time $s$.
Another example for the same problem occurs for instance if in addition to branching the  individuals migrate independently from each other in geographic space or mutate in type space (genome). 
Then if we include information about their path in geographic or type space up to the present time $s$, then independence \textit{fails}.

However we would like to consider both these processes as \emph{generalized}  branching processes. 
Since we have no branching property as a process with values in $ S $ in the classical sense we have to introduce some additional structure on $ S $ and to define a (new) \emph{general branching property}.

The branching property is intimitely connected with the concept of \emph{infinite divisibility}.
In \cite{infdiv} we developed a new concept of infinite divisibility for genealogy valued processes, which used some abstract algebraic and fitting topological structures which allowed to prove basic facts on infinitely divisible random variables with states in genealogies.
This allowed to prove the (generalized form of) \Levy-Khintchine representation of genealogy valued random variables.
This is based on algebraic arguments which had been developed by Evans and Molchanov in \cite{EM14}.
We also formulated a branching property for stochastic processes on $\U$ and $\U^V$. 
We therefore make here again use of abstract concepts in this spirit and use the tools developed in that paper.

However we move on here to a \emph{more abstract approach} of the branching property which allows us to handle more examples as the historical process and path-marked genealogical processes as we shall see later on.
This allows us then in particular also to obtain new examples for processes whose marginal distributions have \Levy-Khintchine representations and inhomogeneous Poison point process representations.

This we next first motivate and then introduce formally in Section ~\ref{s1501141603}.

\paragraph{Idea of formalization for a general branching property and a generator criterion:} Intuitively, the solution of the problem from above is as follows. 
Here since we have to distinguish past and the future the time parameter $t$ will appear, the past is taken relative to $t$.
We have to replace the $+$-operation we have in the measure-valued process by an abstract algebraic operation $\sqcup^t$, in fact depending on a time parameter $t$ and which forms a semigroup for each $t \geq 0$.
Furthermore we need maps $T_t$, which extracts the part of the state at time $t$, which contains information on the past and the part which arises newly in the future.  
For that we then need the set $S_t$ of these pieces of the state and the map $T_t$ and the "$\sqcup^t$"-operation has to be compatible.

    The state space $S$ contains subsets $S_t$ for $t\geq 0$.
    Elements in $S_u$ contain genealogical information (or historical one) which has been generated at times at most $u$ back from the current time $t$.
    Formally we require $S_u \subseteq S_s$ for $u< s$.

    Furthermore we need three assumptions. Consider times $u<s<t$. 
{\em First}, we assume that there are {\em ``truncation maps''} $ T_s $ on $ S $ which allow us to remove at any time $t$ the information about the states of the process before any time $s<t$.
This is mathematically a function $T_{t-s}:\S\to \S_{t-s}$. 
{\em Second}, we assume that on each $\S_t$ there is a binary operation $\sqcup^t: \S_t\times \S_t\to \S_t$ we call {\em concatenation}.
This operations generate for each $ t > 0 $ a {\em topological semigroup structure}. This is the analogue of the addition in the classical settings with $\S$ a linear topological space.
\emph{Thirdly we will relate} $ T_t $ and $ \sqcup^t $ allowing to extend the latter naturally from $ S_t $ to $ \S \times \S $ using $T_t$.

The idea is now that a process which at time $t$ takes values in $\S_t$ is a {\em generalized branching process}, if 
for each $s<t$ conditional on time $s$, the $(t-s)$-truncation of a state in $S_t$ has the same distribution 
as the concatenation of independent subfamilies originating from the individuals alive at time $s$. 
Then we \emph{generalize} the generator criterion \ref{e1301141742} to this structure working with the whole collection of semigroups and truncations.

\paragraph{Outline}
The results on our four goals we present in two sections~\ref{s1501141603} and~\ref{s1501141699}.
We make these ideas described above rigorous in Section~\ref{s1501141603} and give the precise result, which characterizes the generalized branching property via the generator criterion.
In Section~\ref{s1501141699} we give an application to processes of evolving genealogies.
Namely to the $\U$-valued Feller diffusion in Section~\ref{ss.trvalFD} and in Section~\ref{ss.histbranpro} to the $\U^V$-valued super random walk and to historical processes or more general evolving populations with individuals carrying ancestral path.
Altogether we have the four main results Theorems~\ref{T:BRANCHING},~\ref{T.TVF.BRAN.PROP},~\ref{TH.SUPWALK},~\ref{T.ACBRANCH}, the first and the last the highlights.

The proof of the criterion is given in Section~ \ref{ss.T:BRANCHING}.
The proofs of all other statements, those concerning applications, are in Sections~\ref{s1501141604b},~\ref{S.VERCRIT} and then in ~\ref{ss.extmc} we formulate the additional arguments needed for the extension of the claims to \textit{spatial} models.

\section{Results 1: Generator Characterization of Generalized Branching}{\sectionmark{Results 1}}\label{s1501141603}
We now carry out the last two points of the last sections rigorously with the key result Theorem~\ref{T:BRANCHING}.
\paragraph{Formal framework} \; The whole section works for processes with Polish  \emph{state space} $S$ arising as solution of a \emph{well-posed martingale problem}.
We need the following assumptions on the state space $S$ which formalizes the maps $T_t$, {\em truncation} and the binary operation $\sqcup^t$, {\em concatenation} as well as their relation, which is a form of \emph{consistency} property, together with the further technical Assumption~\ref{a.7} below.

\begin{assumption}[Collection of semigroups with consistent  truncation maps]\label{a.6}
Let $\S$ be a Polish space. We use $\mcB(\S)$ for the Borel $\sigma$-field on $S$ and $\textrm{b}\mcB(S)$ to denote the bounded measurable functions on $\S$.

Assume there are $S_t \subseteq S$, $t \geq0$ with the following properties.

\begin{itemize}
\item $\S_s \subseteq \S_t$ for $0 \leq s \leq t$.  
\item $S_t$ is closed in $S$ for all $t \geq 0$.
\item There is for every $t>0$ a continuous mapping $T_t : \S \mapsto \S_t$, which is the identity if restricted to $S_t$: $T_t (x) = x$ for any $x \in \S_t$.
\item There is for all $t\geq0$ a binary operation $\sqcup^t:S_t\times S_t \to S_t$ such that $(\S_t, \sqcup^t)$ is a commutative topological semigroup with neutral element $0 \in S_0$. 
\item The extension of $\sqcup^t$ to all elements of $\S$ is defined for all $t \geq 0$ via
 \begin{equation}\label{e1306}
  \sqcup^t: \S \times \S \to \S_t,\, (x,y)\mapsto (T_t(x))\sqcup^t (T_t(y)) \, .
 \end{equation}
\end{itemize}
 For simplicity, we drop the index $t$ at $\sqcup^t$ if  it's clear from the context. \qed
\end{assumption}
There are many natural examples for this structure without the dependence on $t$, i.e.~$S=S_t$ for all $t$, namely all classical case and the classical branching property, see Example~\ref{ex.572}, ~\ref{ex.583} and the list~\ref{ex.707}, where we saw they include $S = [0,\infty)$ and $S = \mcM_f(E)$, the space of  finite Borel measures which both are semigroups when equipped with addition (of reals or measures, respectively).

However including {\em genealogical} or {\em historical} information requires $ t- $dependence and the generalized branching property, here are examples for both these  effects.

\begin{example}[Ultrametric measure space $\U$-valued processes]\label{ex.umspaces}
 Recall the setting of \cite{infdiv} and \cite{ggr_tvF14} using equivalence classes of ultrametric measure spaces and the semigroup of $t$-forests.
In that paper the evolving {\em genealogy of the population} alive at time $t$ was described via a set of individuals $ U_t $, the genealogical distance between individuals $ r_t(\cdot , \cdot) $ on $U_t \times U_t$,  a population size $ \overline{\mu}_t $ and a sampling (probability) measure $ \wh \mu_t $ on $ U_t $, altogether giving an \emph{ultrametric measure space} $(U_t,r_t,\bar \mu_t \wh  \mu_t)$ and finally with its isomorphy class $[U_t,r_t,\bar \mu_t \wh \mu_t]$ we get the elements of the \emph{state space} $ \U $ describing for us genealogies.
(These objects are called often trees in some literature even though they describe the subset of the \emph{leaves} of weighted $\R$-trees, since nevertheless for every ultrametric space there is an $\R$-tree such that the space can be mapped $1-1$ isometric in the set of leaves, see \cite{GPWmp13}, Remark 2.2 for details.)

Elements of $\U$ are {\em $ t- $truncated} by truncating the metric at $ 2t $ and two such objects are {\em concatenated}, denoted $\sqcup^t$, by taking the disjoint union of the sets of individuals, keeping the metric in each population and setting the distance between individuals from different sub-populations equal to $ 2t $ and by adding the (extended) measures.
 Then $\S = \U$, $\S_t = \U(t)^\sqcup$, the latter the equivalence classes of  ultrametric measure spaces of diameter at most $ 2t $ and $T_t \mfu = \lfloor \mfu \rfloor(t) = [U,r\wedge 2t,\mu]$ the $t$-truncation map, $t \geq0$, $\mfu = [U,r,\mu] \in \U$.
 Details are in Section~\ref{ss.trvalFD}. \qed
\end{example}
\begin{example}[Historical process: measures on ancestral paths]\label{ex.mesanc}
Consider $S_t = \mcM_f(D^t(\R,E)),t\geq0$ where $D^t(\R,E) = \{ f \in D(\R,E) | f(u) = f(0) \; \forall u \leq 0, f(u) = f(t) \;  \forall u \geq t\}$, 
as subsets of measures on \cadlag functions on $\R$ with values in $ E $, a locally compact metric space, for example $E$ would be a geographic space as $ \Z^d $ or $ \R^d $ or $E$ is a type space as finite sets (genome) embedded in $[0,1]$.
These paths could either represent the \emph{geographic location} in the past along the ancestral line through \emph{migration}, or of say a \emph{genetic type} evolving under \emph{mutation}.

Now there are two possibilities to proceed.
This is due to the fact that of interest for us is the ancestral path \emph{seen from the present individual}, i.e. all positions are relative to this present location or alternatively we can thin of the class of path differing only by time-shift as the key object.

(1) The set $ S_t $ consists then of the measures on the path constant except for an interval of at most length $t$.
The equivalence classes are taken w.r.t. \emph{translation in time} of the path.
The truncation of the state is the push forward under the map: replace the path after time $ t $  by its time $ t $ value (this means after the $ t-s $ truncation the equivalence class is a path which has evolved for time $ t-s $ and is otherwise constant, the bracket indicates taking the equivalence classes.
The binary operation is the one induced (on the equivalence classes $ [\cdot] $) by the addition of measures.

(2) Alternatively we could at time $ t $ \emph{shift all path by $ -t $} in the time index to get path which are not constant only for times in $ [-t,0] $.
The truncation and concatenation are defined analog to (1).

Indeed we shall later exploit mainly this second possibility.
See Section~\ref{ss.histbranpro}. \qed
\end{example}

Certain functions on the semigroup will play an important role.
\begin{definition}[$t$-multiplicativity and $t$-additivity]\label{d.tmultiadd}\mbox{}\\
Let $f:\S\to \bbR$ measurable and $t\geq0$. We say that $f$ is $t$-\emph{multiplicative on $\S$} if
 \begin{equation}\label{e1327}
  f(x_1\sqcup^t\dotsm \sqcup^t x_n)=f(x_1)\dotsm f(x_n)
  \,,\quad n\in\bbN\,,\,x_1,\dotsc,x_n\in \S\,.
 \end{equation}
 We say that $f$ is \emph{$t$-additive on $\S$} if
 \begin{equation}\label{e1328}
  f(x_1\sqcup^t\dotsm\sqcup^t x_n)=f(x_1)+\dotsm+f(x_n)
  \,,\quad n\in\bbN\,,\,x_1,\dotsc,x_n\in \S\,.
 \end{equation} \qed
\end{definition}
\begin{remark}\label{r.740}
 \begin{enumerate}
\item In fact the previous definition means that multiplicative functions $f$ are semigroup-homomorphisms $f: (\S_t,\sqcup^t) \to (\R, \cdot)$ {\em except} for the usual continuity assumption for topological semigroup homomorphisms.
  Likewise for $t$-additive functions and the semigroup $(\R,+)$.
\item The previous definition implicitly implies that $f(x)=f(T_t(x))$, $x \in \S$ for a $t$-multiplicative (or  $t$-additive) function. This can be seen in the case of a $t$-multiplicative function via
  \begin{equation}\label{e1329}
   f(x) = f(x) f(0) = f(x\sqcup^t 0) = f(T_t(x) \sqcup^t 0) = f(T_t(x)) \, .
  \end{equation}
 \end{enumerate} \qed
\end{remark}
In order for the above functions being a rich enough set to work with later on we complement the Assumption~\ref{a.6} and require:
\begin{assumption}\label{a.7}
 For any $t\geq0, b\CB(S_t)$ contains a set $D_t \subset \{f \in b\mcB(S) | f(x+y) = f(x)f(y),\, x,\, y \in S_t\}$ of functions which are all strictly positive and the set separates points in $S_t$, i.e.\ for all $x \neq y \in S_t$ we can find $f \in D_t$ with $f(x) \neq f(y)$. \qed
\end{assumption}
The examples~\ref{ex.umspaces} and~\ref{ex.mesanc} above will be shown to satisfy this further condition.

\paragraph{Key result}
Next we define the new concept of the {\em generalized} branching property which holds for a richer class of Markov processes then the Markov processes satisfying the classical one. 
This is formalizing the property that a process which at time $t$ takes values in $\S_t$ is a {\em generalized branching process}, if 
for each $s<t$ conditional on time $s$, the $(t-s)$-truncation has the same distribution 
as the concatenation of independent subfamilies originating from the individuals alive at time $s$. 
\begin{definition}[Generalized branching property]\label{d.genbranpr}\mbox{}\\
Suppose Assumptions~\ref{a.6} and~\ref{a.7} hold.
We are given furthermore a Markov process $(X_t)_{t \geq 0}$ with values in $S$ defined by a family $(P_t)_{t\geq0}$ of probability kernels $P_t:\S \times \mcB(S) \to [0,1]$, $t\geq0$, given via
\begin{equation}\label{e773}
P_t(x,f) = \E[f(X_t)| X_0 = x],\ x \in \S, f\in \textrm{b}\mcB(S), t\geq0.
\end{equation}

The process has the \emph{generalized branching property} if
\begin{equation}\label{e1330}
 P_t(x_1\sqcup^s x_2, h_t) =   P_t(x_1,h_t)  P_t(x_2,h_t) , \quad x_1,x_2 
\in \S_s \,, \end{equation}
 for any $s, t\geq0$ 
and $h_t \in b\mcB(\S)$ $t$-multiplicative on $\S_t$. 

This defines as well the generalized branching property for a solution of a well-posed martingale problem on $D([0,\infty),S)$. \qed
\end{definition}
We write $ h $ for $ h(t,x)=h_t(x), (t,x) \in \R \times S $.
Of course in the previous definitions it suffices to consider a {\em separating subset} of $D_t$ functions in \eqref{e1330}.
\begin{remark}[Terminology]\label{r.948}
Note that we really have here a "$((S_t,\sqcup^t)_{t \geq 0},(T_t)_{t \geq 0})$-branching property", but since we will specify these ingredients in examples we suppress this in the notation throughout.
\end{remark}
\begin{remark}[Time-homogenity of $(X_t)_{t \geq 0}$]\label{r.951}
If we want to treat the time-inhomogeneous case we would have to consider $((S_{u,t},\sqcup^{(u,t)})_{t \geq u}, (T_t)_{t\geq u})_{u \geq 0}$ here setting of a cascade of indices which we want to avoid to focus on the key point.
\end{remark}
\begin{remark}[Measurable path]\label{r.936}
Recall at this point that solutions of a martingale problem must have a version with \textit{measurable} path, since otherwise the defining relation for the martingale problem is not welldefined, since the compensator would not be properly defined.
This property is implied by the stochastic continuity of the process, but is in infinite dimensional state spaces not equivalent to it, therefore requiring stochastic continuity is a restriction.
\end{remark}

From our heuristic reasoning it should be clear that the time an evolution has run is of importance.
Therefore we will even in the time-homogeneous case i.e. a time independent operator for our martingale problem generating the process $X$ consider the time-space process $(t,X_t)_{t \geq 0}$ to have the time the evolution has run explicitly coded in the state.

We can now formulate the main result which gives the {\em characterization of operators} $ A $ of processes $(X)_{t \geq 0}$
which satisfy the {\em generalized branching property}.
We consider here the time-space process $ (t,X_t)_{t \geq0}$ with state space $[0,\infty) \times S$ and operator $A + \frac{\partial}{\partial t}$.
By \cite{EK86}[Lemma 4.3.2] the {\em wellposedness of the $(A,D)$-martingale problem implies wellposedness of the $(\tilde{A},\tilde{D})$ martingale problem.}

\begin{theorem}[Criterion for operators of generalized branching  processes]\label{T:BRANCHING}\mbox{}\\ 
Suppose Assumptions~\ref{a.6} and~\ref{a.7} hold and let $ D_t $ denote the set introduced there.
Let
 \begin{equation}\label{e.tr11}
  \tilde{D}\subset\{(x,t)\mapsto\psi(t) h_t(x):
  \,\psi \in C^1_b (\R,\R),\, h_t \in D_t,\, (t\mapsto h_t ) \in C^1(\R,C_b(\S))\}
 \end{equation}
and let furthermore be given $\tilde{A} = A + \partial_t :\, \tilde{D} \to B(\R \times \S)$.
Finally assume that for any $(x,0) \in \S\times 
 \R$ the following holds.

 For any two solutions $(t,X_t)_{t\geq0}$ and 
$(t,X'_t)_{t\geq0}$ of the martingale problem  for $(\tilde{A},\tilde{D},\delta_{(x,0)})$ one has 
$T_t X_t\overset{\textup{d}}{=} T_t X'_t$ for every $t>0$,
and a solution $(t,X_t)_{t\geq0}$ has a stochastically continuous version.

Under those conditions the family $(P_t)_{t\geq0}$ as in \eqref{e773}
 has the generalized \emph{branching property} if and only if either of the following conditions is satisfied:
 \begin{enumerate}
 \item\label{i:t:1}  For $x_1, x_2 \in \S_t, \, \psi h \in \tilde{D}  , \, t \geq0$: 
\begin{equation}\label{eq:4.10.1}
  \tilde{A} \psi (t) h_t(x_1 \sqcup^t x_2) = \psi'(t) h_t(x_1 \sqcup^t x_2) 
  + \psi(t) [h_t(x_2) \tilde{A}h_t(x_1) + h_t(x_1) \tilde{A}h_t(x_2)]. 
 \end{equation}
 \item\label{i:t:2} For each $\psi h\in \tilde{D}$ there exists a function $g_{\psi,h}:\R_+ \times \S\to\bbR$ 
such that $g _{\psi,h}(t,\cdot)$ is $t$-additive for each $t\geq0$, the $h_t$ are all strictly positive and, for all $(t,x)\in \bbR_+\times \S_t$,
\begin{equation}\label{eq:lin:gen}
 \tilde{A}\psi(t) h_t(x) = \psi'(t) h_t(x) + \psi(t) g_{\psi,h}(t,x) h_t(x)\,.
 \end{equation}
\end{enumerate} \qed
\end{theorem}
Note that $(a)$ is the generalized Kurtz criterion and it is $(b)$ which we will apply in examples.

In the next section we apply part (b) of this result to {\em $\U$-valued} and {\em historical} processes to get \emph{new} examples for the generalized branching property, where the "classical" one does \textit{not} hold. 
In this section we conclude by explaining how the previous result simplifies if $\S_t=\S$ for all $t\geq0$ as in our classical examples~\ref{ex.572}-~\ref{ex.707}, where then the classical branching property holds.
This applies in particular to all classical branching processes mentioned in the introduction, hence it gives an alternative proof of the branching property of measure-valued CSBP in particular of all real-valued CSBP. 

Indeed if $\S= \S_t$ for all $t \geq0$ then multiplicativity of $h_t$ on $\S_t$ just means that $h_t$ is 
multiplicative on $\S$. Assume $D\subseteq \textrm{b}\mcB(S)$ is a set of multiplicative functions and
\begin{equation}\label{e800}
  \tilde{D} \subset \{ (x,t) \mapsto \psi(t) h(x): 
  \, \psi(t) = e^{-\lambda t}, \, \lambda > 0,\, h \in D\}.
\end{equation}
We get, recall from above $(A,D),(\wt A,\wt D)$-well-posedness are equivalent, the following corollary to Theorem~\ref{T:BRANCHING}  giving the {\em classical examples} in \eqref{e1302}-\eqref{e608} by choosing for $ D $ the functions specified there and for $ \sqcup $ just $ + $.

It is then immediate from Theorem~\ref{T:BRANCHING}, that we have the following.

\begin{corollary}[Branching generator: classical case]\label{c.brangen}
Assume $\S$ is a Polish space.
The test functions $D$ satisfy, 
(i) $D \subset \textrm{b}\mcB(S)$ is a set of multiplicative functions  on $\S$, (ii) contains only strictly positive functions and is (iii) separating. 
Finally require that the 
$(A,D)$-martingale problem is well-posed and has a stochastically continuous solution $(X_t)_{t\geq0}$. 

Then the semigroup associated to $(X_t)_{t\geq0}$ has the branching property 
if and only if either of the following conditions is satisfied.
\begin{enumerate}
 \item\label{i:t:1+}  For all $x_1, x_2 \in \S, \, h \in D$: 
\begin{equation}\label{eq:4.10.1+}
  A h(x_1 \sqcup x_2) = h(x_2) Ah(x_1)  + h(x_1) Ah(x_2). 
 \end{equation}
 \item\label{i:t:2+} For each $ h \in D $ there exists a $\sqcup$-additive function
$g_h:\S\to\bbR$, i.e.~$g_h(x_1\sqcup x_2) = g_h(x_1) + g_h(x_2)$ for any $x_1,x_2 \in \S$, with
\begin{equation}\label{eq:lin:gen+}
 A h (x) = g_h(x) h(x)\,,\quad x\in \S, \, h \in D.
\end{equation}
 \end{enumerate} \qed
\end{corollary}

\section[Results 2: New Applications for some evolving genealogies or histories]{Results 2: New Applications for some evolving genealogies  \newline or histories}{\sectionmark{Results 2}}\label{s1501141699}
This work was motivated by trying to generalize and prove the branching property for \emph{evolving genealogies} modeled as marked ultrametric measure spaces, where the classical branching property does not hold.
Of particular interest are \emph{Feller's diffusion model} and the \emph{super random walk}, see \cite{ggr_tvF14}.
Note that we need in order to apply our theory that the martingale problem is well-posed. 
Hence we need to know this either from the literature or we have to prove this here.
For reasons of keeping page numbers under control we give three examples and treat only one example, where we construct a new process and the others are chosen such that we can refer for the constructions of the process to existing literature.
The corresponding results we have here are Theorems~\ref{T.TVF.BRAN.PROP}
,~\ref{TH.SUPWALK},~\ref{T.ACBRANCH}.  

Indeed we now discuss two classes of examples, each in a subsection which are processes of {\em evolving genealogies} respectively {\em histories}, arising as diffusion limits of classical Galton-Watson models. 
First we treat the prototype case which motivated the present paper, the {\em $\U$-valued Feller diffusion}.

Then next we adapt this to the world of \emph{spatial} $\U^V$-valued branching processes covering in particular the {\em genealogical super random walk} or the corresponding {\em historical super random walk process}.
This shows the potential of the criterion for the study of population models. 
There is a whole zoo of further examples.
With the same method we could treat branching random walks, branching Brownian motions or multitype branching processes where individuals undergo mutation or more general continuous state branching processes.

Finally we formulate and construct a \emph{new type} of example the \emph{(ancestral-path)-marked} genealogical super random walk combining histories and genealogies.
This shows the robustness of our concepts and methods for a \emph{new} process of great interest in its own right.
\begin{remark}\label{r.1067}
The last model contains the other ones mentioned in the sense that they are functionals.
However that does not mean that we get the branching property that easy from the general result, since first of all we have to still check the functionals are Markovian, which requires wellposedness of their martingale problems to obtain the Markov property of the functional.
Furthermore we need to show that the ingredients $((S_t,\sqcup^t)_{t \geq 0}(T_t)_{t \geq 0})$ arise as projections.
Therefore we do not loose much by building our examples from bottom to top.
\end{remark}

 Then in the last subsection we describe further examples where the criterion is probably applicable if one carries out some extensions of the theory of $\U$-valued processes to construct them rigorously via well-posed martingale problems.
An example is the \textit{continuum space} spatial models as the Dawson-Watanabe  super process.

\subsection{$\U$-valued Feller Diffusion}\label{ss.trvalFD}

\paragraph{Description of current population and its genealogy}
We consider now the evolution of the genealogy of the \textit{population currently alive} (i.e. at a time $t$) described by an equivalence class of ultrametric measure spaces, in a continuum mass version of a critical binary  branching process. This is the \emph{$\U$-valued Feller diffusion}.
The \emph{main result} of this subsubsection is \emph{Theorem~\ref{T.TVF.BRAN.PROP}}.
However first we have to explain the state space and the process. 
To explain the method of describing genealogical information consider first a critical Galton-Watson process.

The idea is to give two individuals in the population alive at time $t$ a genealogical distance, which is twice the time one has to go back to the most recent common ancestor and to equip the population with the  uniform distribution.
Then we get an ultrametric measure space which describes the genealogy and whose evolution we follow. 
The exchangeability of names and with it the exchangeability of individuals with equal sampling weight then suggests to pass to \emph{equivalence classes} of these objects under \emph{weight preserving isometries}.
More precisely we proceed as follows and begin by shortly reviewing the state space of this process.

The genealogy is modeled as equivalence class of a ultrametric measure spaces $[U,r,\mu] $ where the set $ U $ describes the set of {\em individuals alive} at the current time, the ultrametric $ r $ on $ U $ the {\em genealogical distance} between individuals and $ \mu = {\bar \mu} {\wh \mu}, {\bar \mu} \in \R^+  $ the {\em population size} and $ {\wh \mu} \in M_1 (U,\mcB(U)) $ the {\em sampling measure} specifying how to draw samples of typical individuals from the population alive at time $t$ as for example the uniform distribution.
The \emph{equivalence class} is denoted, $ \mfU = [U,r,\mu] $, equivalence of representations is defined w.r.t. the {\em isometries} of the $ supp(\mu) \subseteq U $ which are {\em measure preserving}. 
For $ \bar \mu =0 $ the measure $ \wh \mu $ is not defined.
The corresponding element we call the zero tree.

The space of all equivalence classes is denoted
\begin{equation}\label{e881}
\U.
\end{equation}
Note that a finite set of points is as metric space characterized by the pairwise distances, i.e. the distance matrix.
The set $\U$ is equipped with the \emph{Gromov weak topology} and is with this topology a \emph{Polish} space.
The topology can be defined, see \cite{GPW09},\cite{Gl12} introducing a metric. 
It is well known however that this topology arises, if we require that a sequence of elements $\mfu_n$ converges to an element $\mfu$ iff the sequence of distance matrix  distributions $\nu_n$ converges to the distance distribution $\nu$ of all orders $m \in \{2,3,\ldots\}$.
Here the distance distribution of order $m$ arises (denoting with the $\Pi^\ast$ the push forward of measures under a map $\Pi$) as $R^\ast(\mu^{\otimes m})$ with $R: U^m \to (\R_+)^{n \choose 2}$ given by
\begin{equation}\label{e940}
R \left((x_i,x_j)_{1 \leq i<j\leq n}\right) = \left(r(x_i,x_j)\right)_{1 \leq i<j\leq m}.
\end{equation}

The description of genealogies by ultrametric \emph{probability}  measure spaces was introduced in \cite{GPW09} extending ideas appearing in \cite{Ev00} and \cite{EPW06}.
This was generalized to the case of \emph{finite} measures in \cite{Gl12} and to locally bounded measures in \cite{GSW}.

The objects of Assumption~\ref{a.6} now arise as follows.
Recall here that the ultrametric spaces can be represented by the set of leaves in an $\R$-tree.
This motivates the following definitions.
The state spaces are the equivalence classes $\mfu=[U,r,\mu]$ of ultrametric measure spaces $(U,r,\mu)$ called $\U,\U(h)^\sqcup $ are the elements of $\U$ with diameter at most $ h,\U(h) $ with diameter strictly less than $ h $ and one considers the {\em truncation of "trees"} of diameter $ t $ at height $ t-h $ for $ h \in (0,t] $ and the {\em $ h- $concatenation} of "trees" as the binary operation on $\U(h)^\sqcup$.

Formally define
\begin{eqnarray}\label{eq:theta:a}
\U(h)^\sqcup = \lbrace \mfu \in \U \mid \mu^{\otimes 2} (\lbrace(x,y)\in U^2 \mid r(x,y) > 2h\rbrace)=0
\\
\nonumber \\
\U(h) = \lbrace \mfu \in \U \mid \mu^{ \otimes 2} (\lbrace(x,y)\in U^2 \mid r(x,y) \geq 2h\rbrace)=0.
\end{eqnarray}

Then define for $ u,v\in \U(h)^\sqcup $ the {\em concatenation} (using $\uplus$ for disjoint union):
\begin{equation}\label{e1310}
\mfu \sqcup \mfv = [U \uplus V, r_U \sqcup^h r_V, \wt \mu + \wt v]\mbox{, with} 
\end{equation}
\begin{equation}\label{e1312}
r_U \sqcup^h r_V \mid_{U \times U} = r_U, \quad r_U \sqcup^h r_V \mid_{V \times V} = r_V,
\end{equation}
\begin{equation}\label{e1313}
r_U \sqcup r_V (x,y) = 2h, \; x \in U, y \in V
\end{equation}
and $ \wt \nu, \wt \mu $ the extension of $ \nu $ and $ \mu $ to the disjoint union which is zero on the respective other component.

The {\em $ h $-top $ \lfloor \mfu \rfloor (h) $} of $ \mfu \in \U $ is defined: 

\begin{equation}\label{e1314}
\lfloor \mfu \rfloor (h) = [U,r \wedge 2h, \mu] \in  \U(h)^\sqcup.
\end{equation}
Then we define the $h$-truncation:
\begin{equation}\label{e1315}
T_h(\mfu) = \lfloor \mfu \rfloor (h), \; S_h= \U(h)^\sqcup.
\end{equation}

\paragraph{The process: $\U$-valued Feller diffusion} 
We consider the $\U$-valued Feller diffusion $ (\mfU_t)_{t \geq 0} $ the process which corresponds to the large population-rapid reproduction-small mass limit of the genealogy of the population alive at time $ t $ of the critical Galton-Watson process. 
The process is the $ \U$-valued process related with its population size process $(X_t)_{t\geq0}$ to the solution of the SDE $\dx X_t = aX_t \, \dx t  + \sqrt{2bX_t} \, \dx B_t$, where $b>0$ and $a \in \R$. 
For the construction and uniqueness of $ \U$-valued processes already treated in the literature we refer to \cite{GPWmp13,Gl12} and for the present model and its spatial versions to~\cite{ggr_tvF14} where the process is constructed via a well-posed martingale problem and many properties of its longtime behaviour are studied.

We recall its operator for the \emph{martingale problem}. 
We need the concept of a polynomial to get the domain of the operator. 
Fix $ n \in \N $ and $ \phi \in C^1_b(\R^{(\substack{n\\2})}, \R)$. Then define for an equivalence class of an ultrametric measure space $ [U,r,\mu] $ the function

\begin{align}\label{e1309}
\Phi^{n,\phi}\Big(\left[U,r,\mu \right]\Big)= \int_{U^n} \phi \Big((r(x_i,x_j),\, 1 \leq i<j \leq n)\Big) \; \mu(dx_1) \ldots \mu(dx_n)
\end{align}
and the action of the operator is given as sum of two operators:
\begin{align}\label{e1305}
 \Omega^{\uparrow} \Phi^{n,\phi}(\mfu) = \Omega^{\uparrow,\mathrm{grow}} \Phi^{n,\phi}(\mfu) + 
\Omega^{\uparrow,\mathrm{bran}} \Phi^{n,\phi}(\mfu) 
\end{align}
and $\Omega^{\uparrow} \Phi^{n,\phi}(0) = 0.$
We need the notation
\begin{equation}\label{e.1143}
\bar \mfu=\mu(U)
\end{equation}
for the total mass of $(X,r,\mu)$, which is an invariant of $[X,r,\mu]$.
 The operators on the r.h.s. are given by
\begin{align}\label{e1307}
 \Omega^{\uparrow,\textrm{grow}}\Phi^{n,\phi}(\mfu)  &=  \Phi^{n,2 \overline{\nabla} \phi} (\mfu) , 
\quad \overline{\nabla} \phi =  \suml_{1\leq i<j \leq n} \frac{\partial \phi}{\partial r_{i,j}},
\end{align}
\begin{align}\label{e1308}
 \Omega^{\uparrow,\textrm{bran}}\Phi^{n,\phi}(\mfu) & = an \Phi^{n,\phi}(\mfu) + \frac{b}{\bar{\mfu}} \sum_{1\leq k < l 
\leq n} \Phi^{n,\phi\circ \theta_{k,l}} (\mfu) ,
\end{align}
 where
\begin{equation}\label{eq:theta}
  \left( \theta_{k,l} (\dr) \right)_{i,j}  := r_{i,j}\1_{\{i\neq l, j\neq l\}} + r_{k,j} 
\1_{\{i=l\}} + r_{i,k} \1_{\{j=l\}} ,\quad 1\leq i < j\,.
\end{equation}

Note that the {\em martingale problem} for $(\Omega^{\uparrow},\Pi(C_b^1))$, where $ \Pi(A) $ are the polynomials defined by $ \phi $  chosen from the set $ A $, has a \emph{unique} solution, see 
\cite{ggr_tvF14}.
This result is presented in Proposition~\ref{p.dual.ext}.

\paragraph*{Interpretation of generator}
This generator arises in this form  in a scaling limit of an individual based model (binary critical Galton-Watson process) as the number of individuals tends to infinity.
The first part correspond to the growth of the genealogical distance with time arising from the growth of the genealogical tree at the tree top only.
The second part describes the effect of a splitting in two branches of the genealogical tree upon a surviving birth.
The term $\bar \mfu$ in \eqref{e1308} arises as rate of the operator which in fact is the resampling operator defining the $\U$-valued Fleming-Viot model.
The background here is the fact that, if we have two groups of ancestors and follow the relative proportion of one among the complete population, we obtain a Fisher-Wright diffusion at a rate which is the inverse of the total mass.
 
See~\cite{DG18evolution} for more explanation and~\cite{ggr_tvF14} for the derivation from an individual based model and detailed information on the structure of the generator and its relation to the Fleming-Viot operator.

\begin{definition}[$\U$-valued Feller diffusion]\label{d.treeval}\mbox{}\\
We refer to the unique solution of the $ (\Omega^\uparrow, \Pi(C^1_b))-$ martingale problem as \emph{$\U$-valued Feller diffusion} and denote it by $\mfU = (\mfU_t)_{t\geq0}$. \qed
\end{definition}

\paragraph{The result}
We consider here as binary operation the {\em concatenation} $ \sqcup^s $ of trees in $ \U^\sqcup(s) $ and as operation $ T_t $ the {\em truncation} operation, which associates if we consider the $\R$-tree representation of the ultrametric space with such a tree its {\em $ t $-tree top} of depth $ t \leqq s$, with $ S_t $ the set of such objects.

We state the main result of this subsection obtained with our criterion.
\begin{theorem}[Generalized branching property: $\U$-valued Feller]\label{T.TVF.BRAN.PROP}\mbox{}\\ 
 The $\U$-valued Feller diffusion $ \mfU $ has the generalized  branching property. \qed
\end{theorem}

\paragraph{A concrete consequence: expected sum of squares of subfamily sizes.}
At first sight,Theorem~\ref{T.TVF.BRAN.PROP} seems to be an abstract statement about an ``extended'' martingale problem on a complicated space.
However if we use Lemma~\ref{l.D_lin:tildeD} from the proof section we can obtain interesting statements, since we can obtain differential equations for expectations of moments where the rhs. involves lower order moments.
This allows in some case to obtain an {\em explicit solution} of such closed systems of equations. The following is an example.
There are much more explicit representations coming out of the martingale problem, in particular via duality relations, see here \cite{ggr_tvF14} for details. 

We consider the expected sum of the squares of the subfamily sizes.

\begin{theorem}[Moment recursion for Feller diffusion]\label{t.recurfeller}
\mbox{}\\ 
 For the $\U$-valued Feller diffusion and $t>0$, if $a \neq 0$:
 \begin{equation}\label{e1360}
 \E [ \Phi^{2,\1(r_{12}< 2t)} (\mfU_t) ] = b \bar{\mfu}_0 \frac{1}{a}\left( e^{2at} - e^{at} \right). 
\end{equation}
Note the rhs. is nothing else than the variance of the Feller diffusion on $\R^+$.
It is known that if $\mfU_t=[U_t,r_t,\mu_t]$ then $(\mu_t(U_t))_{t \geq 0}$ is the classical $\R^+$-valued Feller diffusion with parameter $b$, see \cite{ggr_tvF14}.
In the case $a=0$ the right hand side of \eqref{e1360} is replaced by $2b\bar{\mfu}_0$. \qed
\end{theorem}
\label{1075} 
So far we are not able to derive from the martingale problem other equations which form a closed system and allow for an {\em explicit} solution.
However the compensator of the martingale problem involves lower order expressions which gives some hope.

\subsection{Historical Branching Processes, evolving genealogies of super random walk and path-marked genealogies}\label{ss.histbranpro}

We consider now \emph{stochastic evolutions} of the \emph{genealogy} of populations distributed  in a {\em geographic space} (or type space) $ E $, more precisely a Polish space $ (E,r_E) $.
In addition we have an evolution of the marks, think of migration if marks are locations and of mutation if they are types.
Since the mechanisms are linear in a proper sense we can still hope for a generalized branching property.
This modification of the process requires the generalization of the setup of the previous subsection from the state $\U$ we have to pass to the state space of marked genealogies $\U^V$.

Here $ E $ is separable, complete, metric space which is locally compact as $ E= \Z^d $ or any other at most countably infinite abelian group with a metric or a  continuum for example $ E=\R^d $.
We also consider the case where we have further \emph{historical information} on the whole ancestral paths of the population as they evolve in time.

The \emph{key results} on the generalized branching property are on versions of the model initially motivating the present paper, namely enrichments of the classical \emph{super random walk} (sometimes called interacting Feller diffusions (see \cite{DG03} for example)).
In particular the \textit{genealogical super random walk} process, the \textit{ancestral path marked genealogical super random walk}  process and the \textit{historical process of super random walk}, which are \emph{Theorem~\ref{TH.SUPWALK}, Theorem~\ref{T.ACBRANCH}} and its \emph{Corollary}.

However we have first to define these processes and develop the ingredients for our framework which needs some effort.
We give here the concepts to handle $ E $ of the form mentioned even though we prove the main theorem assuming only that $ E $ is a countable abelian group, because the theory of such processes is for continuum space not yet well developed on the \emph{genealogical level} i.e. as $\U^V$-valued process with $V=E$ to be introduced below.
However as soon as we have settled the \emph{existence and uniqueness problem} of our process we can readily verify the generalized branching property via our criterion.

We discuss first of all three important and related processes which are spatial versions of the case discussed in Subsection~\ref{ss.trvalFD}, among which is the {\em historical process} or the {\em location-marked genealogies} of the \emph{super random walk} modeled as marked ultrametric measure spaces and show that their generators allow to read off the generalized branching property via our criterion and therefore get more and \textit{new} examples, in particular ones not satisfying the classical branching property. 

At the same time we introduce a third model (containing the above as Markovian functionals).
Namely we take up and generalize from \emph{historical processes}, see chapter 12 in \cite{D93}, the concept of \emph{ancestral paths} associated with individuals currently alive and combine this with the genealogy described by elements of $\U$.
Then we can define here a {\em (ancestral paths)- marked genealogy-valued}, precisely $\U^V$-valued, class of \emph{super  processes}.
This class contains as functionals {\em many} processes.
In particular all processes describing \emph{historical} and \emph{genealogical information} of an evolving population:
 location-marked $\U$-valued processes, $\U$-valued processes, historical processes, measure but also $ \R^d $ or $ \R $-valued processes for which the generator criterion for the branching property works very well.
However the case of ancestral path marked genealogies is conceptually and technically \emph{quite demanding}. 

\paragraph{Outline} 
We proceed now in four steps, first we develop the formal description of \emph{genealogical} and \emph{historical information}, second we rigorously define the involved stochastic  processes via well-posed martingale problems and thirdly we identify the ingredients of the formal framework fitting Theorem~\ref{T:BRANCHING} and finally we state in the fourth steps the main results on the generalized \emph{branching property}.
(For facts on relevant spaces of measures and Laplace methods see Section 3 and 4, 6, 8 of \cite{D93}).

\paragraph{Step 1: Genealogical and historical information in spatial population models.}\hfil\\
We begin by introducing the \emph{state space} of the involved processes systematically.

\paragraph{\em (1) Historical process}
The {\em historical process}, a process with values in {\em measures on paths}, was invented in \cite{DP91} to describe the \emph{ancestral paths} of the population alive at time $ t $ in a spatial critical branching process.
Here we assume that the population is observed from time $0$ on and no further information on the past does exist.

Namely every individual alive has a path associated with the migration in space from his birth on, before this birth times the birth time of his direct ancestor, etc. so that the ancestral path of a currently alive individuals gives the \emph{path of descent} (or ancestral path) and the motion through geographic space (or in type space under mutation).

\emph{(i) State space} \; 
That is if $ E $ denotes the geographic space, which has to be at least a Polish space (most frequently with a specified metric, this the setup in chapter 12 of \cite{D93}) one has an element of $ D(\R,E) $ by continuing the path as constant beyond time $ t $ and before time $ 0 $.
These paths merge of course and are in particular equal for all times before $ T $ if they belong to individuals descending from the same parent and $ T $ is the parents death time and the birth time of the two new descendants. 
All individuals are considered exchangeable if they have the same ancestral path.
Hence for time $t$ consider for the population currently alive simply the \emph{"counting measure" on the space of the ancestral paths} to describe the state of the population.

\emph{(ii) Initial states} \;
As initial state in this description one uses typically initially \emph{constant paths} so that the initial state corresponds to the locations of the individuals at the time $ 0 $ in $E$ in a unique way.
Note that in fact we have values in the closed subspace $ D_{0,t} $ where $ D_{0,t} \subseteq D $ are the \cadlag paths with values in $ E $ which are constant before time $ 0 $ and beyond $ t $.
This is a \emph{closed} subset of $ D $.
To get a dynamically closed set of states we need to use $ {\mathop{\cup}\limits_{t > 0}} D_{0,t} $.
Similarly call $ D_{0,\infty} $ the (topologically closed) set of paths constant for times $ \leq 0$.
In order to then discuss infinitely old populations and equilibria we need $ D_{-\infty,t} $ resp $ D_{\infty,\infty}=D $. 
In particular we do have that 
\begin{equation}\label{e1071}
D_{s,t} \mbox{  with } s,t \in \bar \R \mbox{  is a \emph{Polish state space for ancestral paths}.}
\end{equation}

Traditionally one describes the state of these ancestral lines by a {\em locally finite measure}, i.e. a $ \nu_t^\ast \in \mcM(D(\R,E))$ for a locally compact geographic space $ E $ (like $ \R^d $ or some countable set $ E $ like $ \Z^d $).
In the general setup we have measures which are bounded on bounded sets.
These measures generalize the \emph{empirical measures} of a finite population of ancestral paths at varying times and contains in particular the \emph{occupation measure} of the population by projecting the measure $ \nu $ on $ E $ by regarding from the path only the time-$t $ position.
This measure $ \nu_t^\ast $ {\em "counts"} the ancestral path having some specified features, in particular projecting on the current position gives the \emph{occupation measure} $ \nu_t^{\ast, \downarrow} \mbox{ on } E $ of the population at time $ t $.
Namely we define  
\begin{equation}\label{1368}
 \nu_t^{\ast, \downarrow} (A) = \nu_t^\ast (\lbrace v \in D(\R,E):v(t) \in A \rbrace),\quad  \forall A \in \mcB(E).
\end{equation}

Now we say $ \nu^\ast_t $ has the {\em "locally finite} property" if all its time-$t$ projections satisfy:
\begin{equation}\label{e1369}
 \nu_t^{\ast, \downarrow} (A) < \infty, \mbox{ for all A bounded respectively finite in the discrete case}.
\end{equation}

So we have altogether as \emph{historical process} the following  measure valued process:
\begin{equation}\label{e1331}
(\nu^\ast_t)_{t\geq0} ,\;  \nu^\ast_t \in \mcM(D_{0,t}(\R,E)).
\end{equation}

\begin{remark}\label{r.timein}
The process $X$, as for example the process in \eqref{e1331} has typically a time \emph{inhomogeneous} dynamics and therefore it is better to work with the \emph{time-space process} (i.e. $(t,X(t))_{t \geq0}  $ instead of $ X $) of the path process and on top of that also with the time-space process of the measure valued process i.e. we consider $ (\R \times D(\R,E)) $ respectively in (\ref{e1331}) $ \R \times \mcM (\R \times D_{0,t} (\R,E)) $ as state space and consider measures on that space together with another \emph{explicit time coordinate}.
This way one obtains a {\em time-homogeneous} Markovian dynamics. \qed
\end{remark}

The historical process, better its law in its general form, see \cite{DP91} can be described as solution to a Log-Laplace equation or for us more relevant as solution of a \emph{martingale problem} (see~\cite{D93}, chapter 12), in the form which is  specifying martingales for evaluations of the measure and giving their increasing processes.
We will use here for our purposes the more traditional form of the  (local) martingale problem for a given operator, which however is known to be  \emph{equivalent} to the above descriptions.

\paragraph{\em (2) Genealogical processes}
The object in (\ref{e1331}) describes the genealogy at least  implicitly, only if independent copies of the migration paths do {\em not agree on any positive interval}, so that the time point before which the ancestral path agree must be the exact birth point.
This problem can be avoided as follows, using the concept of {\em marked genealogy-valued processes} describing the genealogy of the population \emph{currently alive}, which also allows to formulate and prove the generalized branching property with our criterion.
In that concept it is possible to attach marks to the individuals as for example \emph{types} or \emph{locations}.
This concept has been introduced in \cite{DGP11} and has been successfully applied in the context of Fleming-Viot processes and their genealogies in \cite{GPWmp13},\cite{DGP12} and \cite{GSW} and we use it here for the branching world, see here for $\U$ or $\U^V$-valued processes of this form  also\cite{infdiv,ggr_tvF14}.

\emph{(i) State space: generalities} \; 
We consider more precisely a random variable $ \mfU $ with values a {\em equivalence class} of $ V-${\em marked ultrametric measure spaces}
\begin{equation}\label{e1332}
\mfU = [(U \times V,r \otimes r_V,\nu)],
\end{equation}
with $ (U,r) $ an ultrametric space, $ (V,r_V) $ is a metric space both complete and  separable and $ \nu $ is a Borel measure on $ \mcB ((U \times V, r \otimes r_V)) $ which is {\em boundedly finite} meaning $ \nu(U \times \cdot) $ is a boundedly finite (on bounded sets finite) measure on $ V $.
We denote by $\mu$ the measure on $(U,\CB(U))$ given by $\mu(\cdot)=\nu(\cdot \times V)$, which is increasing limit of finite measures on that space.
Here we have to allow measures $\mu$ now, which are infinite, since on an infinite geographic space we typically want to allow a population with infinite total mass which is only \textit{locally finite}.

Consider $(U \times V,r \otimes r_V,\nu)$ and $(U^\prime \times V,r^\prime \otimes r_{V^\prime},\nu^\prime)$.
Call a map $\phi:U \times V \to U^\prime \times V$ on \emph{mark and measure preserving isometry} if $\phi(u,v)=(\wt \phi(u),v)$ for all $u \in supp(\nu)$ and $\wt \phi$ is isometric between $supp(\mu)$ and $ supp(\mu^\prime)$ and measure preserving, i.e. $\phi_\ast \nu= \; \nu^\prime$.
The space of all equivalence classes, the latter denoted $[U \times V,r \otimes r_v,\nu]$, of $ V-$marked ultrametric measure spaces w.r.t. \emph{measure} and \emph{mark  and measure preserving isometries} of the support of $\mu$ of all \emph{restrictions} to points with marks in bounded sets, equipped with the {\em marked Gromov weak topology}, we denote 
\begin{equation}\label{e1039}
\U^V, \mbox{   which is a {\em Polish space}.}
\end{equation}
We may take equivalently any sequence of bounded sets exhausting the full space $ V $. 

See \cite{DGP11} and \cite{GSW} for the concept of {\em $ V-$marked} ultrametric measure spaces and basic topological facts. 
Roughly convergence amounts to  convergence of all equivalence classes of marked finite subspaces spanned by $ n $ points sampled according to $ \nu $.
(Alternatively we could say: convergence of all polynomials a concept we shall discuss in detail in \eqref{a.101}.)
Then define for $ \mu \in \CM(U) $ with $ \mu(A)=\nu (A \times V) $ the kernel $ \kappa $ by
\begin{equation}\label{e1334}
\nu = \mu \otimes \kappa,
\end{equation}
called the {\em mark kernel}.
This $\kappa$ arises in the special case where a mark {\em function} exists as $ \kappa (u,dv) = \delta_{\kappa(u)} (dv) $ for a measurable function $ \kappa $. 
The second marginal of $ \nu $ corresponding to $ V $ is denoted
\begin{equation}\label{e1335}
\nu^\ast.
\end{equation}

We want to choose the {\em current locations} or alternatively the {\em ancestral paths} as marks:
\begin{equation}\label{e1333}
V = E \: \mbox{    or    } \; V = D(\R,E).
\end{equation}
The latter choice embeds then the historical process in the genealogical process on $\U^V$ which contains then the combined information of the $\U^E$ and the $\CM (D(\R,E))$-valued process.

In \eqref{e1333} the first case is easier to handle (compare \cite{GSW} for the Fleming-Viot process in that case and \cite{ggr_tvF14,infdiv} for the branching case). 
For the second case we need some further \emph{preparation} we focus on next, a reader only interested in the more classical situation, the first case (genealogical super random walk) might move on to the next step.
Nevertheless, this second example is the highlight, despite the technicalities necessary.

\emph{(ii) State spaces: Path-marked genealogies} \;
Indeed the spatial branching process of \cite{DP91}, the so called \emph{historical process}, can be extended to a $ \U^V$-valued process $ \mfU $ containing most of the relevant genealogical \textit{and} historical information (see the discussion below) and then this $ \nu^\ast $ of \eqref{e1335} is the measure state and with it we get a version of the historical process, formulated in the {\em "classical"} way, i.e. for the measurable $ \Psi $ on $\U^V$ arising from projecting an element of $U \times V$ on $V$ and hence $\nu$ on $\nu^\ast$, set:
\begin{equation}\label{e1336}
(\nu_t^\ast)_{t \geq 0} = \Psi((\mfU_t)_{t\geq 0}).
\end{equation}

We formulate now a setup in which 
to apply the martingale problem techniques (see Remark~\ref{r.timein}). Introduce as mark space $V$:
\begin{equation}\label{e1104}
D=\mathop{\bigcup}\limits_{t \in \R} D_{[t,\infty) } \quad, \quad D^+= \mathop{\bigcup}\limits_{t \in [0,\infty)} D_{0,t} (\R,E).
\end{equation}
We use parallel to the historical process also his time-space process (recall Remark~\ref{r.timein}), hence also as mark resp. state space:
    \begin{equation}\label{e1177}
    V=\R \times D \mbox{  or  } V=D^+ \mbox{  resp.  } \R \times \U^V.
    \end{equation}
Note that we pass here both on the level of the ancestral path \textit{and} on the $ \U^V- $valued process to the time-space process because we want on both levels {\em time-homogeneous} processes (recall Theorem~\ref{T:BRANCHING} is for time-homogeneous processes).
    In that case we choose as marks in points with the explicit time component $ t $, paths in $ D_{0,t}(\R,E) $, which are constant after time $ t $ and before time $ 0 $.

\paragraph{Step 2: The class of genealogical and historical models: Formal construction of $ \U^V$-valued super random walk}\hfil\\
We focus now on two cases, namely $ V=E $ and $V=D$ where $E$ is a metric space, which is Polish.
But we restrict $E$ to be a countable abelian group such that we can define {\em random walks}, which allows simplifications in construction and more is known about populations in such geographic spaces.
The basic process is now the {\em $\U^E$-valued super random walk on $ E $}, which we now introduce in $(1)$ before we come to $(2)$ to the case $V=D$ and in $(3)$ to an alternative path-valued process.
As preparation we begin with the classical case.

\paragraph{\em (0) A classical spatial process.}
The classical \emph{super-random walk} is the following system of interacting diffusions $ X=\Big((x_i(t))_{i \in E} \Big)_{t \geq 0} $ with parameter $ b>0 $ and $ a(\cdot,\cdot) $ a transition probability on $ E \times E $ and  state space contained in $ [0,\infty)^E $:
\begin{equation}\label{e1321}
dx_i(t)=\suml_{j \in E} a(i,j)(x_j(t)-x_i(t)) dt + \sqrt{bx_i(t)} \; dw_i(t) \quad, i \in E, 
\end{equation}
where $ E $ is embedded in a continuum group $ E^\prime $, which is Polish and with $ a(.,.) $ a transition probability on $ E $ some {\em discrete abelian group}, for example $ E=\Z^d, E^\prime = \R^d $.
See \cite{DG96}, \cite{GKW02} for construction and properties of this process.
Recall that in the approximating individual based model individuals migrate from $ i $ to $ j $ with probability $\bar a $ instead of $ a(\text{here  } \bar a (\xi,\xi^\prime) = a(\xi^\prime, \xi)) $. 

The associated {\em marked genealogy-valued}, i.e. $ \U^V$-valued, dynamic has to be defined below rigorously by a {\em well-posed martingale problem} which can be show to arise as a \emph{scaling limit} of the genealogy of a branching random walk on $ E $, equipped with the appropriate marks in $E$ resp. $D$, the walk with transition rate $ a(\cdot,\cdot) $ with critical branching with many individuals-small mass and rapid branching.

\begin{remark}\label{r.contvers}
It is also interesting to consider {\em continuum space} versions of the above process, as the Dawson-Watanabe process. 
The Dawson-Watanabe process as usual arises as spatial continuum limit from the systems indexed by scaled versions of $ E $, for example $ E^\prime = \R^d $ and $ E= \varepsilon \cdot \Z^d $, with $ \varepsilon \rightarrow 0 $, compare Section~\ref{ss.outquest} part {\em (iii)} for a more detailed discussion of the arising problems. \qed
\end{remark}

If the underlying geographic space $ E $ is finite we work with finite measures. 
If however $ E $ is countable it is more natural to work with measures which are only {\em locally finite}.
In that case the construction of the process of total masses (per site population sizes) via the system of SDE's in \eqref{e1321} already requires to restrict the state space to configurations, where the local mass cannot explode in time due to the flow of migration.
The appropriate tool here is the so called {\em Liggett-Spitzer space}, a subset of $ [0,\infty)^E $ of the form

\begin{equation}\label{e1373}
\{\overline{x} \in [0,\infty)^E \mid \suml_{i \in E} \overline{x}_i \gamma_i < \infty \}
\end{equation}
for some summable positive $ \gamma $ satisfying $ a \gamma \leq M \gamma $ for some $ M< \infty $.

On this space the solution of the SSDE can be constructed and any random configuration which is translation invariant with $ E \mid \overline{x}_0 \mid < \infty $ 
is almost surely in this space, regardless which $ \gamma $ we choose, and furthermore the solutions of the SSDE have almost surely \emph{path} which take values in the Liggett-Spitzer space for all $ t \geq0 $.
We do not discuss this in more detail here, the reader may look at \cite{LS81}, \cite{DG96} and \cite{GLW05} for details.

\paragraph{\em (1) The $\U^E$-valued super random walk.}
We have to discuss now state spaces, test functions and operators.

\emph{(i) Marked genealogies of super random walk and their state space.}
For this process above we want to construct now the \emph{genealogy} of the individuals alive at a given time $t$ together with information on \emph{locations}, which is described as an element of $\U^V$.
Primarily this is a $E$-marked genealogy $(V=E)$ which we have to define.

 The next point needed are the genealogies which requires the extension of the theory of processes in {\em marked} ultrametric measure spaces from the state space of the previous subsection.
Here we consider as marks on the genealogy the \emph{current location} of the individual which is changing due to migration or more generally its \emph{path of descent}, here the marks evolve according to the {\em path process}, see \cite{D93}.
Processes on that space of marked genealogies have been introduced in \cite{DGP12}, for the case of finite sampling measures, we explain the basic concepts and facts we need here.
A further point is to allow on infinite sets infinite population sizes in the form of {\em boundedly finite} measures.
This works by considering the localization of the population to ones with marks in bounded subsets of the geographic space, defining the equivalence classes w.r.t. to all the restrictions and defining the topology by defining convergence by the convergence of all localizations to finite (bounded) sets. 
For detail we refer the reader to \cite{GSW} where also a class of genealogical processes is introduced, different from ours though, namely for the spatial genealogical Fleming-Viot process.

The $ \U^V$-valued process is constructed as solution to a \emph{well-posed martingale problem}.
The martingale problem of the $\U$-valued process is treated in \cite{ggr_tvF14} together with its spatial $\U^E$-valued version in all detail for a survey see also \cite{DG18evolution}.

As \emph{starting points} for our dynamic we will allow states of a special nature by requiring that if $[U \times V,r,\nu]$ is such a state, then the restriction of $ \nu $ to $ V $ satisfies \eqref{e1373}, for $ V=E $ and for the case of path that the further projection of the path onto the time $ t $ position satisfies this relation.
This subset of $ \U^V $ satisfying~\eqref{e1373} is called then $\CE $, a set defined by requiring the property of the projection of $  \nu $ on $ V $ that is supported by the measurable subset.

\emph{(ii) Test functions, generators and the genealogical processes.} 
We specify now the {\em generator} of the martingale problem and its \emph{domain}.
The basis is the operator we had in \eqref{e1305}-\eqref{eq:theta} but we have to \textit{lift} it to the marked case and we have to add the migration part giving the {\em dynamics of the marks} by specifying the \emph{generator}.\\
\emph{(a) Test functions.}
The first step is to generalize the concept of a polynomial to cover marks. 
Choose now a function $ \chi : V^n \rightarrow \R, \; \phi \in C_b(V^n, \R) $ and consider {\em the monomial}:

\begin{align}\label{a.101}
&\Phi^{n,\phi,\chi} \left(\left[U \times V,r,\nu \right]\right) =&\\
&\int\limits_{(U \times V)^n} \phi \left(\big(r(x_i,x_j)\big)_{1 \leq i < j \leq n} \right) \;  \chi \left(v_1,v_2, \ldots, v_n \right) \; \nu \left(d(x_1,v_1), \ldots \nu \big(d(x_n,v_n)\big) \right),& \nonumber
\end{align}
where it suffices to consider $ \chi $ of the form (these are still {\em separating})
\begin{equation}\label{e1178}
\chi (v_1,\cdots, v_n)=\chi^1 (v_1) \cdots \chi^n (v_n).
\end{equation}
Then consider the generated algebra the \textit{polynomials}, denoted $ \Pi_E $ see \cite{DGP11}, \cite{GSW}.

The simple case is where $ V $ is simply the geographic space then $ V=E $ and $ \chi^k: E \rightarrow \R $.
Then we can consider again {\em Laplace functionals} $ exp(-\Phi^{n,\phi,\chi}) $ for $ \phi, \chi \geq 0, n \in \N$.
In this marked case where $ E $ is not a finite (bounded) set and where we work with populations which are not necessarily finite but are only \textit{locally finite} we take $ E_m \uparrow E, E_m $, finite (bounded) and restrict the $ \chi $ to be of finite (bounded) support.
We work with $ \chi^k $ of the form that it specifies a single site of observation, i.e.
\begin{equation}\label{e1264}
\chi^k(\cdot) = 1_{\{\xi^k\}} (\cdot) , \xi^k \in E \mbox{  and write for  } \chi \mbox{  in that case  } \chi^{\underline{\xi}} \mbox{  for  } \underline{\xi}=(\xi^1, \cdots, \xi^n) \in E^n.
\end{equation}
Then the polynomial $ \Phi $ depends only on the population in a finite number of sites in $ E $.
Then we can take the generated algebra and have a separating set (\cite{GSW}).
The reader only interested in the super random walk $V=E$ may skip the next point and continue directly with the dynamics.\\
{\em (b) The dynamics and the operator.} 
Next the action of the {\em operator}. 
The growth and branching operator from \eqref{e1307} and \eqref{e1308} act now as follows. 
The \emph{growth operator} acts only as before on $ \phi $ and the \emph{replacement operator} the same but only if the marks of the two chosen points have the {\em same current locations}, i.e. are marked with path with the \emph{same}  current {\em site}.

The branching operator $\Omega^{\uparrow, \rm bra}$ however is now a \emph{sum} of operators $ \Omega_\xi^{\uparrow, \rm bra} $ acting on the population at the site $ \xi $, the sum over $ \xi \in E$.
We have
\begin{equation}\label{e1265}
\Omega_\xi^{\uparrow, \rm bra} \Phi^{m,\phi,\chi} = \frac{2b}{\overline{\mfu}_\xi} \; \Phi^{m,\phi^\prime, \chi}, \mbox{  with  } \phi^\prime = \suml^m_{k,i=1 \atop k \neq i} \; (\theta_{k,i} \; \phi - \phi) 1_{\{v_i=v_k=\xi\}}.
\end{equation}
Note that our polynomial depends only on a \emph{finite} number of sites so that the sum is well defined.

Also the operator $ \Omega^{\uparrow, \rm bra}_\xi \Phi^{m,\phi, \chi} $ is only non zero for such $ \xi \in \{\xi^1, \ldots, \xi^{n(\chi)}\} $ which specify $ \chi $ as in \eqref{e1264} and \eqref{e1322} such that $ \bar \mfu_{\xi^ k} \neq 0$.

Note also at this point that now for $ \overline{\mfu}_\xi=0 $ for some $ \xi $ on which the polynomial depends we have a {\em singularity} and we set then $ \Omega_\xi^\uparrow \Phi^{m, \phi, \chi}=0 $.
Observe that this occurs now even if the population as a whole is \emph{not} yet extinct.
This requires restrictions on the test functions we can use in the martingale problem.
Here we refer for the technicalities also to \cite{DG03}.

Next we need the operator of the mark evolution the \emph{migration} operator. 
We need here that the path of marks arises from a Markov process $ (Y(t))_{t \geq 0} $ on $ E $ solving a {\em well-posed martingale problem}, with operator $ A $.
For example for the branching random walk the migration rate $ \bar a $, where $ \bar a(v,v^\prime)=a(v^\prime,v) $ for $ v,v^\prime \in E $ we have:

\begin{equation}\label{e1323}
(Af)(v)=\Big(\suml_{v^\prime \in E} \bar a(v,v^\prime) f(v^\prime)\Big)-f(v),\quad v \in E, f \in b\CB(E,\R),
\end{equation}  
representing the motion of individuals on $ E $.

We define the operator $ \Omega^{\uparrow,\rm mig} $ describing the evolution of the marks driven by the migration defined by $ A $ above and which acts on the polynomial $ \Phi^{\phi,\chi} $ as follows:
\begin{equation}\label{e1179}
\Omega^{\uparrow, \rm mig} \Phi^{n,\phi,\chi} = \suml_{k=1}^n \; \Phi^{n,\phi,A_k \chi} \; \; , A_k \chi = \Big(\prod_{i \neq k} \; \chi^{\xi^i} \Big) A \chi^{\xi^k}.
\end{equation}

Then summing the {\em migration operator} in \eqref{e1179}, the {\em growth operator} and the {\em resampling operator} $\Omega^{\uparrow,\rm bra}$ (see \eqref{e1265} and above)  results in an operator
\begin{equation}\label{e1180}
(\Omega^{\ast, E}, \Pi_E).
\end{equation}
\begin{definition}[$\U^E$-valued super random walk]\label{def.1332}\hfil\\
The well-posed $(\Omega^{\ast,E},\Pi_E,\mfu)$-martingale problem (see the non-spatial case and the spatial case in \cite{ggr_tvF14} for the wellposedness result) specifies for $\mfu \in \U$ with $\bar \mfu \in \mathcal{E}$ a process :
\begin{equation}\label{e1181}
(\mfU^{\ast, E}_t)_{t \geq 0}
\end{equation}
the {\em $\U^E$-valued super random walk on $ E $.} \qed
\end{definition}

\textbf{\em (2) The (ancestral path)-valued case: $D^+$}
Next we have to focus on the case of {\em ancestral path as marks}.
Recall here \eqref{e1104}.
We will give here first the process using as ancestral path elements in $D^+$ which are the ones directly given by as by the dynamics naturally and will only later introduce the ones marked with $D^\ast$ using the knowledge of the $D^+$ case.
First we discuss the test functions for path-valued marks, then the dynamics and operator.

Here we have to begin by introducing the test functions on $ V $ which is in our context a set of {\em $ E $-valued \cadlag path}.\\
\emph{(i) Test functions}\; 
Here we have the time-inhomogeneous and time-homogeneous case.\\
(a) The test function $ \chi^k $ on the rhs. of \eqref{e1178} each  evaluate the positions of the path at a tuple of time points and will be of the form of a product of functions on $ E $ which are applied to the position of the path at a specific time giving now a tableau of functions representing the $ \chi^i $ of \eqref{e1178}:
\begin{equation}\label{e1311}
\chi^{\uuxi}(v) =\prod^n_{i=1} \; \prod^{m(i)}_{k=1} \chi^{\uuxi}_{i,k} \left(v(t_k^{(i)})\right) \; , \; v \in D(\R,E),
\end{equation}
for some $n \in \N  $ and for $k=1,\cdots,m(i)$ we have $ 0 \leq t_1^{(i)}<t_2^{(i)}< \ldots < t^i_{m(i)}< \infty,$ where next $ i=1,\cdots,n $. The $ \chi^{\uuxi}_{i,k} $ are again indicators as in \eqref{e1264}.
Here $ \underline{\xi} $ from above becomes $ \uuxi=(\xi^i_k)_{i=1,\cdots,n \atop k=1, \cdots, m(i)} $.

Note that $ \chi^{\uuxi} $ above is {\em not} in $ C_b(D(\R,E))$
as required in the basic setup of $ \U^V $-valued martingale problems where we take usually functions from $ C_b(\U^V,\R) $  (fitting the ones in the definition of the topology), but here we have functions being only in $ b\mcB(\U^V,E)) $.
This technical point we have to tackle later in the proof section by giving either an equivalent martingale problem on continuous functions via a moving average of the present test functions or we have to use special properties of the path which are charged by the law.
\begin{remark}\label{r.1244}    
	One could try to work with occupation time integrals of the form:
	\begin{equation}\label{e1311a}
	\chi^{\uxi, \uF}(v) =\prod^M_{m=1} \chi_m(v) = \prod^M_{m=1} \; \int_\R \dx s F_m(s) \1_{\xi_m}(v_s)
	\end{equation}
	for $\xi_m \in E, F_m \in C^1(\R) \cap L^1(\R)$, $m=1,\dotsc, M$.
	This however does not fit together easily with the martingale problem for $ V $ which is the one which allows to access easily and directly important information about the f.d.d. of the ancestral path. 
	We will use this idea in a different form namely once we work with truncations and we have to smoothen.
\end{remark}
(b) For the \emph{time homogeneous} set up of the path process (i.e. the time-space process) we consider
\begin{equation}\label{e1322}
\wh \chi^{\uuxi}(t,v) = \prod^n_{i=1} \; \prod^{m(i)}_{k=1} \wh \chi^{\uuxi}_{i,k} \left(t,v(t \wedge t_k^{(i)})\right),
\end{equation}
with $ \wh \chi^{\uuxi}_{i,k}(t,v) = \Psi(t) \; \Psi^{\uuxi}_{i,k}(t) \chi^{\uuxi}_{i,k}(v) $ and $ \Psi(t) $ and $ \Psi_{i,k}^{\uuxi} \in C^1_b(\R, \R)$ as the functions in (\ref{a.101}) to generate polynomials. Call the set of these polynomials $ \wh \Pi $.

We continue now the discussion of path-marked processes, with giving the dynamics.

\emph{(ii) Dynamics and operator} \;
Here we proceed as in the point $(1)$ on $V=E$ except now that the operator of the mark evolution given in \eqref{e1323}-\eqref{e1180} has to be replaced by a new operator and in particular the test function $\chi$ from \eqref{e1322} instead of \eqref{e1264} has to be used.

Here we proceed in several steps. 
First introduce the operator which describes the change in the path process, then secondly based on this we define  the operator of the mark evolution of the $ \U^V $-valued process with $V=D(\R,E)$ acting on the polynomial $ \Phi $.
The evolution of marks, i.e. of the ancestral path as time evolves is driven by an evolution of an element of $ D(\R,E) $. 
This Markov process is called the {\em path process}.
This process was introduced, for example in Section 12.2.2. in \cite{D93}. 
This evolution is time \emph{inhomogeneous}.
    
    The next step is to write down the generator of the path process. 
    Since already the path process $ Y $ is time-inhomogeneous (recall the path is $ \R- $indexed), we pass first on that level to the {\em time-space process} and need test functions on $ \R \times E $ rather than just $ E $.
    
    Recall the generator of the motion process of a single individual $(Y(t))_{t \geq 0}$ was called $ A $.
    For the process $ Y $ the time-space process $ (t,Y(t))_{t \geq 0} $ then has generator $ \wt A = \frac{\partial}{\partial t}+A $. 
    The corresponding \emph{path process generator} $ \wh A $ acts (see Section 12.2.2 in \cite{D93}) on $ \chi $ of the form \eqref{e1322} for $ t_k \leq s < t_{k+1} $:
    \begin{equation}\label{e1326}
    \wh A \wh \chi(s,v)= \prod_{\ell=1}^k \wh \chi_\ell \left(s,v(s \wedge t_\ell)\right) \wt A \left(\prod_{\ell=k+1}^m \wh \chi_\ell \Big(s,v(s)\Big)\right) \; \mbox{and gives $ 0 $ for $ s>t_m $.}
    \end{equation}
        This operator specifies a well-posed martingale problem on the spaces $ D([0,\infty),\R \times D(\R,E)) $ (Section 12.2.2 in \cite{D93}).
    In a second step this has to be then lifted to an operator on the polynomials on $ \U^V$, where it acts via the action of the path process generator on the function $ \chi $ appearing in the polynomial.
    
    Now turn to the $\U^D$-valued super random walk, the process $\mfU^{\ast, \rm anc}$.
    In order to get a time-homogeneous Markov process we pass to the state space
    \begin{equation}\label{e1229}
    \R \times \U^D \mbox{  and polynomials in  } \wh \Pi=\Pi_D,
    \end{equation}
    which are products of a function $ \psi \in C_b(\R,\R) $ and a polynomial as in \eqref{a.101} with $\chi$ as in \eqref{e1322}.
    
    The operator $ \Omega^{\uparrow, \rm anc} $ corresponding to the change of the marks now acts on the polynomial $ \Phi $ as follows.
    Namely denoting this operator of the mark evolution of $ \mfU $ by $ \Omega^{\uparrow, \rm anc} $ we have for each sampled marked individual the action of the path process generator $\wh A$ but now acting on the corresponding factor $\wh \chi^k$:
    
    \begin{equation}\label{e.1157}
    \Omega^{\uparrow, \rm anc} \Phi^{\phi, \wh \chi} = \suml_{k=1}^n \; \Phi^{\phi, A^\ast_k \wh \chi} \; , \text{  where  } A^\ast_k \wh \chi
    =: \Big(\prod^n_{\ell=1 \atop \ell \neq k} \; \wh \chi^\ell \Big) \wh A  \wh \chi^k.
    \end{equation}

Taking the sum of the \emph{mark} operator $ \Omega^{\uparrow, \rm anc} $ and the adapted and {\em lifted} (from $ \U $ and then from $ \U^V $ to $ \R \times \U^D $) \emph{growth} and the \emph{resampling} operator from point $1$ then gives immediately
\begin{equation}\label{1370}
(\Omega^{\ast, \rm anc}, \wh \Pi).
\end{equation}

We now have to specify precisely the possible \emph{initial states} for starting times $s$ for our martingale problem.
    We choose here the subset of $ \U^D $ which is characterized by a further restriction namely having marks only on a set $ \wt V $ (i.e. $ supp (\mu) \subseteq U \times \wt V $):
    \begin{equation}\label{e1248}
    \wt D = \mathop{\bigcup}\limits_{- \infty < s \leq t < \infty} \{t\} \times D_{s,t} \mbox{  or even  } \mathop{\bigcup}\limits_{s \in \R} \{s\} \times D_{s,s}.
    \end{equation}
    The corresponding subset of $ \U^{\wt D} \cap \CE $ is then called $ \wt \CE $.
    
    We have to show now that the martingale problem on the set of test functions above is \emph{well-posed}.

\begin{proposition}[Existence and uniqueness of the $\U^D$-valued super random walk]\label{prop.1426}
\hfill\\
The $(\wh \Omega^{\ast,\rm anc},\wh \prod, \mfu)$-martingale problem is well-posed.
\end{proposition}    
     
    Here some technical points have to be addressed in case of an {\em infinite} geographic space in particular the potentially infinite total mass of the measure $\nu$ on $U \times D$.
    We address this in Section~\ref{ss.extmc} in more detail as well as the martingale problem establishing wellposedness.
    We now have made sense of what we mean by the ancestral path marked process.
 
\begin{definition}[$\U^D$-valued super random walk]\label{def.1393} \hfil\\
    The solution of the well-posed $(\Omega^{\ast, \rm anc},\wh \Pi,\mfu)$-martingale problem, the $\U^D$-valued super random walk, is denoted
    \begin{equation}\label{e1167}
    (\mfU_t^{\ast, {\textup{anc}}})_{t \geq0}.
    \end{equation} \qed
 \end{definition}
 
\begin{remark}\label{r.diffemark}
We note that this process is different from the $ E $-marked ultrametric Feller process who records genealogy and {\em current} position, which is a Markovian functional of $\mfU^{\ast,\rm anc}$. 
We shall see later how we can verify for this process as well the generalized branching property, similarly for the historical process. 
This will in particular show how these various processes and their $ (T_t)_{t \geq 0}, (S_t)_{t \geq 0}$ and $(\sqcup^t)_{t \geq 0}$ are related. \qed
\end{remark}    

\textbf{\em (3) Pathmarked case: a sufficient variant of the standard state description via $D^\ast$} \; 
We have to specify state space, test functions, operators and processes.

\emph{(i) State space} \;
In this setup for the version we formulate next we can then show the {\em generalized branching property} since we use a more suitable concept of ancestral path which we introduce next.
Namely we have to \emph{modify} the description of the genealogy and mark a bit further to fit our framework of Assumption~\ref{a.6} and of the Theorem~\ref{T:BRANCHING} to be applied here, but in which we are \emph{keeping the really wanted information described above}.
Recall Theorem~\ref{T:BRANCHING} requires a \emph{time-homogeneous}  process and a semigroup setting.

The idea is that the ancestral path describes the situation looking back from the presently living individual and hence in particular  positions  \emph{relative} to this current position give the interesting informations. 
     We define therefore the set of path providing the desired information as follows:
     \begin{equation}\label{e1121}
      D_{-t, 0}(\R, E) = \{ v \in D(\R, E) | v(s) = v(-t) \, \forall s < -t, \ v(u) = v(0) \; \forall u > 0 \}, \ t \in [0,\infty]
     \end{equation} 
     and
     \begin{equation}\label{e1125}
\wh {D}^\ast = \bigcup_{t \in [0, \infty)} D_{-t, 0 } 
     \end{equation}
     as paths which are constant before a specific time $-t$ and after time $0$.
To obtain a Polish mark space we have of course to take 
\begin{equation}\label{e1146}
D^\ast \mbox{  as the closure of } \wh D^\ast.
\end{equation}
The last space $D^\ast$ defines now the \emph{mark space} of interest and the resulting state space for the generalized branching property is 
     \begin{equation}\label{e.Dast}
       \U^{D^\ast}.
     \end{equation}

\emph{(ii) Test functions} \;
The $ V=D^\ast$-set up is handled as follows.
This process arises for us later on as a functional of the process in (2) via a map $ \mfR $ in  \eqref{e1128} and \eqref{e1132}.
In order to specify a martingale problem for this process which is again Markovian we have to begin by introducing the test functions on $ V $ which is in our context an {\em $ E $-valued \cadlag path} constant after time $ 0 $ and before time $ -t $.
This means we get functions as \eqref{e1311} but with $ t_1^{(i)} < t_2^{(i)} < \ldots < t_{m (i)}^{(i)} \leq 0, i=1,\ldots,n $.
The test function $ \chi^k $ on the rhs. of \eqref{e1178} each  evaluate the positions of the path at a tuple of time points and will be of the form of a product of functions on $ E $ which are applied to the $ (-t)$-shifted version of the function $ \chi^{\underline{\underline{\xi}}} $ from above at time $ t $, which means that we work with a function $\wt  \chi^{\underline{\underline{\xi}}} $ evaluating $ D_{-\infty,0}$-functions at fixed times and locations, independent of the current time $ t $ as in \eqref{e1311}.

\emph{(iii) The $D^\ast$-marked genealogical super random mark and the $\Omega^{\rm shift}$ operator}
From the process $\mfU^{\ast,\rm anc}$ we obtain now another process.
We observe that relevant for us is the information on the path viewed from the present position which we can identify with a path which moves only between time $-t$ and $0$.
Therefore we apply the $(-t)$-shift at time $t$ to the marks, which are the ancestral path.
The new path we call the \emph{adjusted} path.
This gives uniquely a new process for which we obtain the \emph{well-posed martingale problem} (see Section~\ref{ss.extmc}) by adding the operator corresponding to the $(-t)$-shift of the ancestral path, which then allows us to apply our theory.
This additional operator we have to specify below.
In addition we act now with the operator $\Omega^{\uparrow,\rm anc}$, acting on the marks on elements in $D_{-t,0}$ at  time $t$ only and call this restriction of $\Omega^{\uparrow,\rm anc}$ now $\Omega^{\uparrow,\rm anc-sh}$.

{\em Observation:}
The interesting information on the path sits at time $ T $ of evolution starting in an element $ \{s\} \times D_{s,s} $ in the piece $ (T > t \geq s) $ i.e. a path in 
the time interval $ [s,T] $.

For our purpose revealing the generalized branching property the following fact is important.
Using the setup for the path space $D^\ast$ we can formulate a \emph{time-homogeneous} dynamics.
This is different from the classical historical process as defined in Chapter~12 of \cite{D93}.
We modify \eqref{e1326} by putting $ s=0 $ and adding at time $ t $ in \eqref{e1179} the generator of the path shift by $ -t $.

Before we can define the generator of our process we have to calculate the effect \emph{of the shift}.
Observe that the underlying path is that of a jump process with generator as in \eqref{e1313} which means that the path are piecewise constant with finitely many jumps in finite time intervals.
At each jump time of the path coinciding with a time of evaluation of the path by our function we may get a contribution to the generator since the jumps then leads to a jump of the evaluation functional of the path as we shift.
Such a situation does occur with probabilities we can control. Proceed as follows.

\paragraph{\em Calculation}
Consider now a polynomial $\Phi^{\phi,\chi}$ where $\chi$ is as in \eqref{e1311}.
Let $ -t_m^{(i)} < -t_{m-1}^{(i)} < \cdots < -t_1^{(i)} \leq 0 $ be the times where such a path is evaluated for the $n$ paths labelled $i=1, \cdots,n$  and which is tested whether it is  in points $ \xi_k^i,k=1,\cdots, m(i) $ at the times $ t_k^{(i)} $.
Furthermore consider the \emph{jump times} of the $ n $ paths which we denote by $ (s_k^{(i)})_{k=1,\cdots,m(i)} $ ordered from the left to the right in $ k $, where the path jumps from the point $ \zeta^i_k $ to $ \zeta_k^{i,+} $.

A contribution arises if $ \xi^i_k \in \{\zeta_k^{i,+}, \zeta_k^i\} $ since then in the concerned factor by a small shift a jump from $ 1 $ to $ 0 $ or $ 0 $ to $ 1 $ occurs and is then causing a change in the complete product of $ +1,0, $ or $ -1 $ respectively.
Note that this effect occurs along the whole path.
The corresponding operator is the jump generator of the switch of $\chi_{i,k}$ to $0$ resp. $1$, if a jump from $\xi^i_k$ away resp. into that point from outside at the rates given below in~\eqref{e1462}.

We have next to obtain the probability for this possibility to occur due to a small time shift $ \Delta t $, which allows for one jump at order $\Delta t$ probability at each of the possible time point $i=1,\cdots, m(i)$ of one path.
However the intensity for a jump from $\xi^i_k$ \emph{away} respectively \emph{into} at time $t_k^{(i)}$ for some $i \in \{1,\cdots,m(i)\}$ is given by
\begin{equation}\label{e1462}
\suml_{\xi \in \Omega \setminus \xi_i^k} \; a (\xi^i_k,\xi)
\mbox{  respectively  } \suml_{\xi \in \Omega} a(\xi,\xi_k^i)
\end{equation}
at both expressions summed over $k \in \{1,\cdots,m(i)\}$.
This can happen in one for each of the $n$ path (corresponding to the index $i$). 

We get therefore as additional generator term the expression:
\begin{equation}\label{e1367}
\Omega_t^{\rm shift} \Phi^{\phi,\chi} = \suml^n_{i=1}  \Phi^{\phi,A^{\rm cum}_i \chi}, \mbox{  with  } A^{\rm cum}_i \chi_{i,k} =\suml_{k=1}^{m(i)} \wt A_k \chi_{i,k}, 
\end{equation}
where $\wt A_k$ acts on the $k$-th component of the samples.
Note that the operator on the rhs. does in fact not depend on $t$.
\qed

Now add $\Omega^{\uparrow,\rm bran}, \Omega^{\uparrow,\rm grow}, \Omega^{\uparrow,\rm anc-sh}$ and $\Omega^{\rm shift}$ to get $\overset{\leftarrow}{\Omega}^{\ast,\rm anc}$.

\emph{(iii) Dynamics} \;
This allows us now to define a new process as follows.
\begin{definition}[$\U^{D^\ast}$-valued super random walk]\label{def.1436}\hfil\\
	We call the solution of this {\em well-posed} $(\overset{\leftarrow}{\Omega}^{\ast,\rm anc},\wh \Pi^\leftarrow)$- martingale problem the {\em adjusted ancestral path-marked genealogy-valued super random walk} a process denoted
	\begin{equation}\label{1364}
	\overset{\leftarrow}{\mfU}^{\ast, \rm anc},
	\end{equation}
	with values in $ \U^{D^\ast} $ as in \eqref{e.Dast}. \qed
\end{definition}

What is the \emph{relation} between states in $\U^{D^\ast}$ and $\U^D$?
Why is the martingale problem on $\U^{D^\ast}$ wellposed?
The key point is now that we have a collection of $ 1-1 $ maps $ \mfR_t, t \geq 0 $ between our process of interest with value in $ \U^V $ the one with $ V=D $ and the one with marks $ V=D^\ast $.
Namely for the path-valued process at time $ t $, starting at time $ s $ with $ t \geq s, t \geq 0 $, we shift to the left all the paths by $ t $ in the time coordinate and we obtain \textit{a time-homogeneous process with marks in $ D^\ast$}: i.e. for a functional of our process at some specific time $t$ we have a map
\begin{equation}\label{e1128}
\mfR_t \left(\left[U \times V, r  \otimes r_V,\nu \right] \right) =\left[ \left(U \times D^\ast, r \otimes r_{D^\ast}, \nu^\ast  \right)\right]
\end{equation}
induced by
\begin{equation}\label{e1132}
(u,v) \to (u,v^\ast) ;\;  v^\ast(r)=v(r-t), \; r \in \R \mbox{  and  } v  \mbox{  the path forward associated with  } u.
\end{equation}
This collection of mappings define on the \emph{paths} $(\mfR_t \mfU_t)_{t \geq 0}$ of the stochastic process a map $ \mfR $ and it will be  
\begin{equation}\label{e1143}
\mbox{the process   } \mfR (\mfU),
\end{equation}
to which we apply our theorem on the generalized branching property.
\emph{However since this is a bijection we do not loose any information we coded in the state description initially, but have it in a technically more convenient form.}

The $\U^{D^\ast}$-valued process arises as a functional (namely \eqref{e1143}), which also solves a well-posed martingale problem in its own right.
Here the wellposedness follows from the fact that $\mfR$ is one-to-one and the wellposedness of the basic martingale problem on $D^+$ mentioned above.

From the construction of the process (since the martingale problem is well-posed) we have 
\begin{equation}\label{e1478}
\overset{\leftarrow}{\mfU}^{\ast,\rm anc} = \mfR(\mfU^{\ast,\rm anc}).
\end{equation}
\begin{remark}\label{r.1364}
	With the choice of test functions as in \eqref{e1311a} we would get the generator actions:
	replace ``$A\chi^{\xi^k}$'' by $\tilde{A}\chi^{\uxi^k}$ where this new operator $\tilde{A}$ acts as follows on one test function as in \eqref{e1311a}:
	\begin{align}
	\tilde{A} \chi^{\uxi, \uF}(v) = \sum_{m=1}^M \frac{\chi(v)}{\chi_m(v)} \Big[ - & \int_{-\infty}^0 \partial_s F_m(s) \1_{\xi_m}(u(s)) \, \dx s + F_m(0) \1_{\xi_m})(u(0)) \\
	& + \int_0^\infty F_m(s)\, dx s \, A(x\mapsto \1_{\xi_m)(x)}|_{x = u(0)}  \Big]. \nonumber
	\end{align} \qed
\end{remark}

\paragraph{Step 3: The framework: Truncation, concatenation, subfamily decomposition and generalized branching property on the spatial level}\hfil\\
We have to define now the {\em concatenation} and the {\em truncation} for {\em $ V $-marked} ultrametric measure spaces, denoted $ \U^V $ for complete separable metric $ V $ space (recall \eqref{e1332} and sequel).
This means we have to extend the concepts in \eqref{eq:theta:a} - \eqref{eq:theta} from the previous subsection now to the {\em marked} case.
Here we have to distinguish the marked case with marks being the current location from the one marked with ancestral path. 

\paragraph{\em Case 1}
Look first at $ V=E $.
Let $\nu_U$ denote the projection of $\nu$ on $U$.
We first define $ \U^V(h) $ as the elements $[U \times V,r,\nu]$ of $ \U^V $ with $ (U,r) $ having $\nu_U$-essential diameter strictly less than $ h $, similarly $ (\U^V(h))^\sqcup $ with less than or equal to $h$.
In order to form a {\em concatenation} of $ (\U^V (h))^\sqcup $-elements $ \mfu_1=[U_1 \times V, r_{U_1} \otimes r_V,\nu_1] $ and $ \mfu_2=[U_2 \times V, r_{U_2} \otimes r_V,\nu_2] $ define, recall $\sqcup$ abbreviates $\sqcup^h$, and \eqref{e1310}-\eqref{e1313}) for $r_{U_1 \sqcup U_2}$:
\begin{equation}\label{e1215}
\mfu_1 \sqcup \mfu_2= \left[(U_1 \uplus U_2) \times V, (r_{U_1 \sqcup U_2}) \otimes r_V,\wt \nu_1 + \wt \nu_2 \right]
\end{equation}
where $ \wt \nu_1 (A) = \nu_1(A \cap (U_1 \times V)), \wt \nu_2(A) = \nu_2(A \cap (U_2 \times V))  $ for $ A \in \CB ((U_1 \uplus U_2) \times V )$.

Next we need the \emph{truncation}. The truncation affects only the distances and acts as before, hence we get again by lifting the operation from $ U $ to $ U \times V $:

\begin{equation}\label{1374}
\lfloor \mfu \rfloor (h) , \mfu \sqcup^h \mfu^\prime, T_t(\mfu).
\end{equation}

\paragraph{\em Case 2}
    In the second case $ V=D(\R,E) $ (respectively $ \R \times D(\R,E) $ as we will pass later to the time-homogeneous formulation) the marks contain themselves some {\em information from the past} in particular they contain information about genealogies. 
    We therefore have to extend the $ h$-truncation to the marks so that the marks contain only the information about the ancestral path for some time $ h $ back analogue to the genealogy which we include till depth $ h $.
\textbf{See here Figure 1.}
   
   \begin{figure}
     \includegraphics[width=14cm]{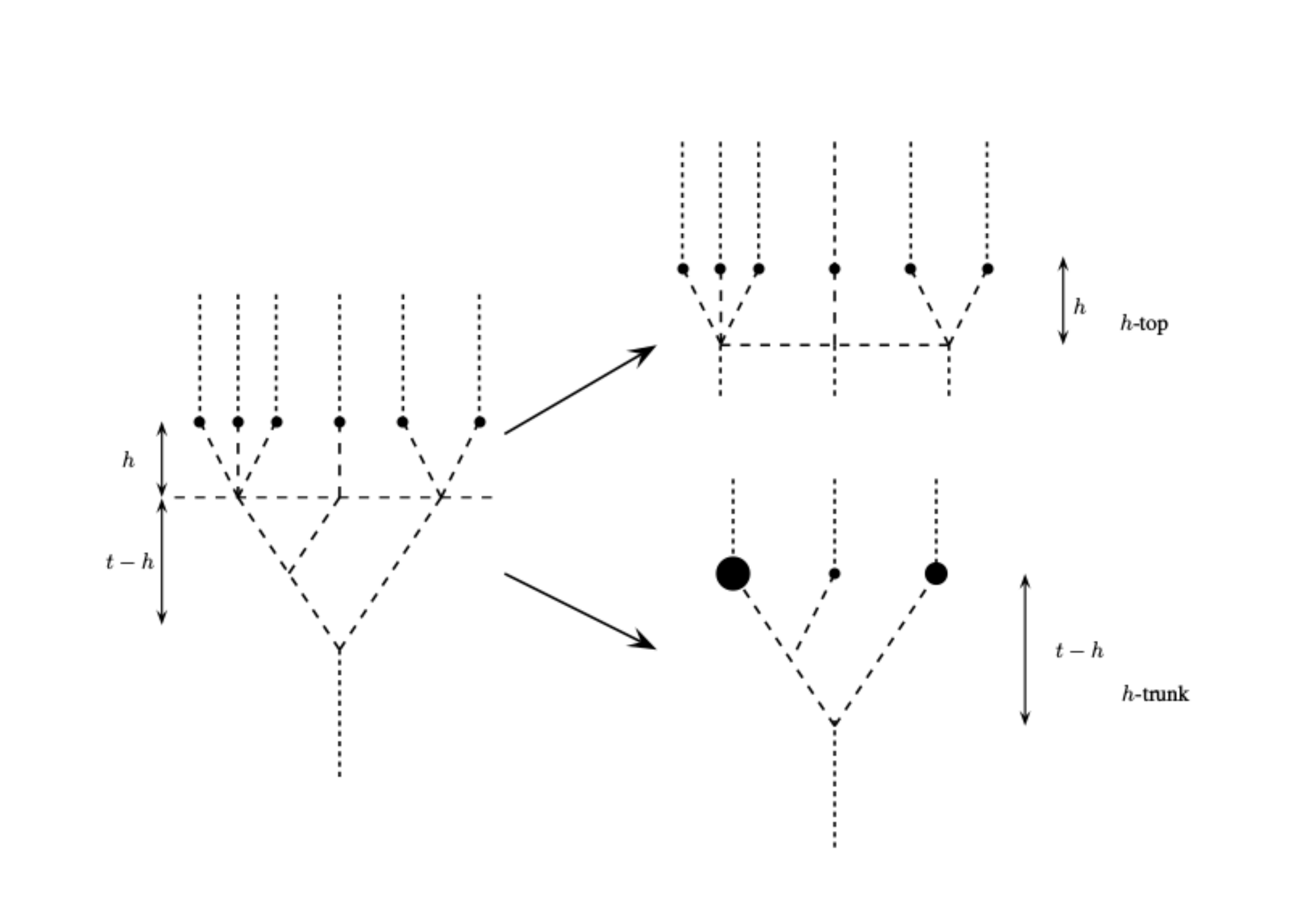}
  \caption{Example of path marked $h$-top and $h$-trunk, cut at height $t-h$. The \textbf{\small $\cdots$} mark the constant parts of the path and {\large --} the parts where different values might be assumed!}\label{f2607131839}
\end{figure} 
       
Fix a time horizon $ T $.
    We define a collection of mark spaces $ (V_t)_{t>0} $ with $ V_t \subseteq V_s \subseteq V $ if $T \geq t > s $ where for a process evolved till time $ T $ we consider the path fluctuating only between time $ t $ and $ T $:
    \begin{equation}\label{1372}
    V^T_t = \{ v \in D(\R,E), \; v(u) = v(0), u < 0; \; v(u) = v(t), u \leq t; v(u)=v(T), u \geq T \}, \quad V_t:=V_t^\infty.
    \end{equation}
    In our time-homogeneous system with marks in $ \R \times V $, the first mark component is preserved and we have to act only on the second, then where at time $ T $ we have ancestral path constant before time $ 0 $ and after time $ T $, we would \emph{truncate the mark} as follows:

Set for $ u \in \R, \; T \geq t \geq 0, $
    \begin{equation}\label{1373}
    T^V_t: V^T_0 \longrightarrow V^T_t, \; (T^V_t(v))(u) = v(u) 1_{[t,\infty)}(u)+ v(t)1_{[u \leq t]} (u).
       \end{equation}
Note that truncation means here \emph{truncations of the fluctuations} not the path as such and we look from the bottom up rather than from the top down keeping only randomness in the path beyond time $t$.
    
    Next define the \emph{$ h-$truncated} $ V-$marked objects.
    Begin with the $ t-$ truncation map $ T_t^V $.
    The truncation map $ T_t^V $ of $ [U \times V,r \otimes r_V,\nu] $ is defined considering first a map $ T_t^U $ on the genealogical part $ [U,r,\mu] $.
    Define $ T_t^U $ acting on $ [U,r,\mu] $ as before $ T_t $. The resulting space we have to equip with the truncated marks, i.e. each point in $ U $ now is \emph{marked} with $ T_t^V(v) $ instead of $ v \in V $.
    Finally we pass to the image measure of $ \nu $ under this combined mapping on $[U \times V,r \otimes r_V,\nu]$.
    
    We define $ \U^V(h) $ as a marked ultra-metric measure space where all distances are strictly less than $ h $ and the mark kernel $ \kappa $ satisfies
    \begin{equation}\label{1209}
    \kappa(u,V \setminus V^T_{T-h})=0 \; , \; \forall \; u \in U.
    \end{equation}
    Next $ \U^V(h)^\sqcup $ has distances $ \leq h $ and the mark kernel $ \kappa $ satisfies again \eqref{1209}.
    
\paragraph{Step 4: The setup for the generator criterion and results  on the spatial level}\hfil\\
The next task is to apply our criterion which is without problems if $ V=E $ but the situation is more subtle for $V=D(\R,E)$.

\paragraph{The result for $E$-marked super random walk.} \mbox{}\\
This will be a consequence of our more general (and more  complicated to formulate) result on ancestral path marked $\U^V$-valued super random walk, see Corollary~\ref{c.treewalk}, later on for that. 
\begin{theorem}[$\U^E$-valued super random walk]\label{TH.SUPWALK}\mbox{}\\
The process $\mfU^{\ast,E}$ has the generalized branching property. \qed
\end{theorem}

\paragraph{Result for path-valued process}\hfill\\
We now treat $V=D^+$ and $V=D^\ast$, which needs more efforts.
\medskip

\textbf{\em (1) The path-valued results: Introduction}\;
The situation is more subtle in the path-valued situation on which we have to focus here next. 
We first explain the idea.

\begin{remark}[Basic idea of truncation with marks in $ D^\ast $]\label{r.1247}
The needed truncation can be understood best as follows using a measure $\R$-tree representation of the state in $\U^{D^\ast}$ even though in our proofs this viewpoint \emph{is not used}.

Recall that for every {\em ultrametric measure space} of diameter $ 2t $ there is a unique (diameter $ 2t $) {\em $ \R-$tree} such that the points in the ultrametric space (better the support of the sampling measure) are the {\em leafs} of the $ \R-$tree.
Consider this representations for the states of our evolving genealogy.
If in addition these $ \R-$trees, see Example~\ref{ex.umspaces} for explanation and more references, for varying $t$ are all embedded in an weighted $\R$-tree of all individuals alive at some time before time $t$, a property of the dynamics (which we have in our case and which we will discuss more in \cite{ggr_tvF14} and we sketch some of it in part (ii) of the next subsection) allowing that we can define a \emph{$E$-valued ancestral path} in this object.
Then the ancestral path can be used to mark the points in the $ \R-$tree such that the $ E$-marked \emph{geodesic} between founding father and the current individuals of a maximal subfamily (i.e. a current leaf descending from the founding father) is a copy of the ancestral path.
Then the point on the geodesic starting from a leaf carries the current position of the ancestor at depth $ h $ in distance $ h $ of the tagged leaf.

In this picture based as the associated weighted $\R$-tree  we want for $h$-truncation to \emph{cut the tree at depth $ s=t-h $} and then use  the \emph{cut marked geodesic} as the ancestral path we associate with the leaf in the truncated state. 
This leads to a replacement of the ancestral path by one which we continue beyond the piece \emph{back time $ h $ constant}, so that we still have an \emph{$ \R$-indexed path} but no information from the past before time $ t-h $ is retained.
Note that this means introducing the one root at depth $ h $ via the truncated and now depth-$h$ $\R-$tree and at this root we have a mark-kernel (even if we had before a mark function).
Note that this new path if evaluated at time positions gives different numbers then the one where we \emph{restrict} to path which are constant before time $ t-h $.
Note that the latter procedure would include information on the past before time $ t-h $.

We may go further and shift the obtained path by $ t $ to the left to obtain for every $ t $ always a path from $ D_{-\infty,0} $ which is in fact for the used initial state in $ D_{-t,0} $ then for $ t>0 $ also in $ D_{-t,0} $ and after truncation in $ D_{-h,0} $ and is for all $ t \geq h \geq 0 $ independent of $ t $.
This fact is the reason why we get for $ \mfR(\mfU) $ the generalized  branching property. \qed
\end{remark}

\textbf{\em (2) A problem in the case of path-valued marks and its solution} \;
    This construction above however does not quite fit our setup in Theorem~\ref{T:BRANCHING} yet since we now do not have images under truncation in the right set {\em since in \eqref{1372}-\eqref{1209} we have still $ T $ around} (but recall our $ \R-$indexed path have started evolving at time $ 0 $ up to the current time $ T $).
    This time dependence we have to remove. 
However we have already introduced a system with \emph{reduced information} namely $ \overset{\leftarrow}{\mfU}^{\ast, \rm anc} $, since we are interested in the ancestral path looked backward from the present time.
    This means if we code the present time in the state we are only interested in the element in $ D^\ast$ we get! 
    We set therefore for $ t>0 $:
    \begin{equation}\label{e1288}
V_t=D^\ast_t = D_{-t,0} \subseteq D^\ast.
    \end{equation} 
\begin{remark}\label{r.1292}
    We may use as marks now $ D^\ast $ rather than $ D^+ $ but we note that we do not loose information this way if we \emph{know} that we start in initial conditions at time $ 0 $ as specified, namely constant before time $ 0 $, since then by piecing that constant piece together with the piece from  $ D^\ast $ shifted by $ t $ beyond time $ 0 $ we reconstruct uniquely the original state at time $ t $. \qed
    \end{remark}
    
We saw above that we may pass via the map $ \mfR $ to the mark space $ V=D^\ast $ if we take the time-space process.
This induces also a map on the truncated objects where we now get with $ t$-truncations elements in $ V^\ast_t=D^\ast_t $.
    The same relations hold then for the truncated marked ultrametric measure spaces induced by $ D_t^\ast, D^\ast $ as we required for $ S_t $ and $ S $.
    
Then we can now define the full truncation:
    \begin{equation}\label{e1374}
    T^\ast_t \left( (U \times V, r_U \otimes r_V, \nu) \right) = (S_t \times D^\ast_t, (r_U \wedge 2t) \otimes r_{V^\ast}, (id \otimes T_t^{D_t^\ast})_\ast(\nu)),
    \end{equation}
    which induces the map on the of the $ D^\ast$-marked ultrametric spaces as well, \emph{correspondingly the $ S_t $ are defined by}:
    \begin{equation}\label{e1207}
    {\Big((\U(t))^\sqcup \Big)}^{D^\ast_t}, \; t \geq 0.
    \end{equation}
    
    The $ h-$concatenation of two $ D^\ast-$marked forests is now defined by:
    \label{1264}\begin{equation}
    \mbox{ in \eqref{e1215} we replace   } V \mbox{  by  } D^\ast.
    \end{equation}
    This means for the corresponding set of polynomials we obtain the new elements
    \begin{equation}\label{1479}
    \wh \prod \mbox{  and  } \wh \prod^\leftarrow
    \end{equation}
and in order to get the \emph{corresponding} elements we replace $ \chi $ by $ \chi^\ast $ where we specify the time points where we evaluate path now in points of the left half axis.
    These objects again satisfy the conditions of Assumption~\ref{a.6}.

\emph{Summarizing we have the setup of our Theorem~\ref{T:BRANCHING}} \quad
The semigroups $(S_h,\sqcup^h), h \geq 0,$ of Assumption~\ref{a.6} are as follows:
\begin{equation}\label{e1473}
 S = \U^{D^\ast} ,\ S_h = \U^{D^\ast_h}.
\end{equation}
For the definition of the mark spaces see \eqref{e1146}.

Truncation is defined via the pull-back mappings.
Therefore let $\mfu = [U\times D^\ast, r\otimes r_V, \nu ] \in \U^{V^\ast}$ and $h \geq 0$.
Define the truncated space $T_h \mfu$ as 
\begin{equation}\label{e1481}
 T^\ast_h \mfu =  [U\times D^\ast, (r\wedge 2h) \otimes r_V, (id \otimes T_h^\ast)_\ast \nu ],
\end{equation}
where $T_h^\ast : D^\ast \mapsto D^\ast_t$ via
\begin{equation}\label{e1485}
 T_h^\ast v = 
 \begin{cases} 
   v(-t) ,& \ s < -t, \\
   v(s), & \ s \geq -t.
 \end{cases}
\end{equation}
Concatenation follows the same definition as \eqref{e1215}.\medskip

\textbf{\em (3) Results on (Adjusted path-)marked genealogical super random walk}
Having completed these preparations we state a key fact about the martingale problem. 
Recall from above truncation and $ h- $concatenation.
The truncation $ T_h $ at depth $ h $ now cuts the genealogical distance at $ 2h $ and the mark, i.e. the path is set constant up to a piece of length $ h $ and lies in $ D_{-h,0} $. 
Therefore the map $ T_h $ now cuts of the $ h $-top of the $ E $-marked weighted $ \R $-tree associated with $ \mfU $ and if we work with the process $ (\overset{\leftarrow}{\mfU}_t^{\ast, \rm anc})_{t \geq 0} $ our approach fits and Theorem~\ref{T:BRANCHING} will apply and give the following.

\begin{theorem}[Functional of $ \mfU_t^{\ast,{\textup{anc}}} $ has generalised branching property]\label{T.ACBRANCH}\mbox{}\\ 
The adjusted ancestral path marked $\U^{D^\ast}$-valued super random walk satisfies:
\begin{equation}\label{e1432}
\left( \left(\overset{\leftarrow}{\mfU}^{\ast, \rm anc}_t \right) \right)_{t \geq 0} \mbox{ has the generalized branching property.} \qquad \square
\end{equation} 
\end{theorem}
If we make a \emph{strong assumption} we can obtain the branching property of a functional.
\begin{corollary}[Generalization]\label{r.1428}
Suppose we have a separable, complete, metric space $ E $ for which our process is solution of a well-posed martingale problem on $ \U^V $ with $ V=E $ as in \eqref{e1333}, then the above theorem holds as well.
\end{corollary}
Since it is known form the literature that the assumption above holds we get:
\begin{corollary}[$ \U^E$-valued super random walk]\label{c.treewalk}
The process $ (\mfU^E_t)_{t \geq 0} $ has the generalized branching property.$ \quad \square $
\end{corollary}
There is another case where we know that the strong assumption of the wellposedness of the martingale problem through existing deep work \cite{DP91} and we obtain:
\begin{corollary}[Historical process of super random walk]\label{cor.histproc}
The historical process associated with \eqref{e1321} has the property that $ (\nu^\ast_t)_{t \geq 0} $ of \eqref{e1331} satisfies that $ (\mfR_t (\nu^\ast_t))_{t \geq 0}$, the historical process of adjusted path, has the generalized branching property.$ \quad \square $
\end{corollary}

Why does this all follow from our Theorem~\ref{T.ACBRANCH}?

\begin{remark}\label{r.Emark}
(a) For the {\em $ E $-marked} genealogy-valued process where we record only the present location , i.e. $ V=E $ the corresponding truncation map will {\em not} change the mark.
The concatenation operation is now as before (with the different $ V $).
A measurable function of $ \mfU^{\ast, \rm anc} $ projecting the mark on the value of the path at time $ t $ gives us a process which is a Markov process in its own right solving a well-posed martingale problem with the operator in \eqref{e1180}, thus having the branching property.
In other words, we may use Theorem~\ref{T:BRANCHING} since those multiplicative functions we need to check for the generator criterion were already considered by the ancestral path marked genealogy-valued super random walk.

(b) The {\em historical process} is a functional of the $V$- marked genealogy, where now $ V=D^\ast(\R,E) $ i.e. we map
\begin{equation}\label{e1324}
[U \times V,r,\mu] \longrightarrow (\pi_V)_\ast \mu
\end{equation}
where $ \pi_V $ is the projection from $U \times V \rightarrow V $.
This functional is again a Markov process, namely a \emph{modification} of the one known as the \emph{historical process} introduced in \cite{DP91} but now using \emph{adjusted} path i.e. the \emph{state space} $ D^\ast $.

However there is a unique lifting to a $ V- $marked case by shifting by $ t$ the path which are constant path before time $ 0 $, i.e. which "start" at $ 0 $ and end fluctuating at the current time $ t $.
For the historical process we will have a truncation map which is the map induced by the replacement of path with the path which is {\em before depth $ h $ kept constant} and the concatenation $ \sqcup^h $ is the sum of the measures.
Hence this process inherits the generalized branching property for these choices of truncation and concatenation. \qed
\end{remark}
\begin{remark}[Infinite divisibility and \Levy-Khintchine formula]\label{r.1935}
We can now use the result we obtained for the process $\overset{\leftarrow}{\mfU}^{\ast,\rm anc}$ to show with results from \cite{infdiv}  that the process has \emph{infinitely divisible} marginal distributions if we start in a fixed initial state.
In particular do we then have a \emph{\Levy-Khintchine presentation for the marginals} and with it an \emph{inhomogeneous Poisson point process representation} of the state at time $t$. For that we can generalize the definition and the proofs of the Theorem 1.37,2.44, Corollary 1.40 in \cite{infdiv} to the object $[(S_t, \sqcup^t)_{t \geq 0}, (T_t)_{t \geq 0}]$ immediately, since we use only what we have postulated as Assumption~\ref{a.6} and ~\ref{a.7}. 
This allows us to apply this representation in our path-marked model and to study the subpopulations defined as individuals in maximal distance say $h$ for $h \in (0,t]$.
We can ask for the number of such families and their structure, see here also \cite{ggr_tvF14} for details on these objects, questions and results.
 \end{remark}
\subsection{Outlook and perspectives: open questions for genealogies}\label{ss.outquest}

We discuss here three directions of extensions for genealogical processes with values in equivalence classes of marked metric measure spaces which should be studied and resolved in the future:
\begin{itemize}
\item genealogies including the {\em fossils},
\item offspring laws with {\em fat tails}, 
\item genealogical processes in {\em continuum geographical space}.
\end{itemize}
The first point can be handled based on some work in progress the two others are open problem here some approaches we investigated but nothing complete exists so far.
\medskip

{\bf {\em (i) The genealogy including fossils and CRT}}

Another type of extension would be to consider the genealogy of all the individuals \emph{ever} alive before the current time $t$.
This is the population including the {\em fossils}.
Then with $t$ running through $(0,\infty)$ we obtain an evolving genealogy.
This will be a subset of the equivalence classes of treelike metric measure spaces.
This gives us an $\M$-valued stochastic process. 
Here again we have the (generalized) branching property, via the criterion.
In fact we could consider $ t \rightarrow \infty $ and take the complete genealogical tree of all individuals ever alive if the population becomes extinct we obtain a limit, see \cite{ggr_tvF14}.
This object has values in $ \M $ the space of equivalence classes of metric measure spaces.
For the Feller branching dynamic the equivalence class in $ \M $ has a representation which is known as the CRT, for the latter see \cite{Aldous90}, \cite{Ald1991}, \cite{Ald1991a}, \cite{Ald1993}, \cite{AldlGall}.

In that case we work with metric measure spaces of a specific form instead of {\em ultra}metric ones, which requires some \emph{new} elements.
The topology on the state space has been treated in this general form, but there is the issue of the dynamic.
First of all the corresponding processes have to be constructed with well-posed martingale problems and then the concatenation and truncation structures have to be introduced and then the criterion has to be checked.

The first point is treated in work in progress \cite{GSWfoss} the  state space is contained in the {\em rooted marked metric measure spaces}, for the state at time $ t $ with distance at most distance $ t $ from the root denoted $ \M_t $ and we comment here on the second point.

\label{3.3state}The state of time $ t $ is in $ \M_t $. Now the $ h- $truncation on $ \M_t $ removes all points in distance less than $ t-h $ from the root, and truncates distances at $ 2h $. 
The $ S_h $ consist of subspaces with points which are in distance at least $ t-h $ from the root and have at most the distance $ h $.
Accordingly $ M_t(h)^\sqcup $ are now the metric measure spaces with a root, points at most in distance $ t $ from the root but at least $ t-h $ from the root and other distances at most $ 2h $, analog only further than $ t-h $ from the root and of the distances less than $ 2h $ for $ \M_t(h) $.
The $ h- $ concatenation in $ \M_t(h)^\sqcup $ is defined as before.
For the dynamic the branching property follows from the criterion.
This induces a generalized branching property on the $ t \rightarrow \infty$ limit the $ \M $-valued version of the CRT.

{\bf {\em (ii) Branching with more general offspring distribution}}

A natural question is how we can treat the branching processes where we have as basis branching processes for an offspring distribution without higher than first moments.

In this case the total mass process does not have anymore all moments and we can not work as before with a martingale problem where the test functions are monomials, which was a key point in the proof of the previous results in~\cite{ggr_tvF14}, so that we are lacking at the moment a characterization by a well-posed martingale problem.
Here we now have to work with {\em local} martingale problems or work with Laplace functionals.

Another point is that now the resampling operator has to be replaced by an operator where in a sample not only distances to one point in the sample change but a whole random set of individuals arises now from one ancestor, the Kingman coalescence mechanism in the dual is replaced by a $ \Lambda- $coalescent type dual transition which would have to be identified, since in general this cannot be just a $\Lambda$-coalescent.
This means it requires \emph{very substantial work} to rigorously construct the $\U$-valued process via a well-posed martingale problem, even though no principal problems seem in the way.

{\bf {\em (iii) Genealogical and historical processes in continuum space}}

In the literature \cite{D77}, \cite{D93} one studies the {\em continuum space} analogue and limit of the super random walk respectively its historical process version the so called {\em  (historical) Dawson-Watanabe process}.
Here one would like to proceed similarly and introduce the {\em genealogy valued version} of these processes.
Here some problems arise at the starting point, namely to establish the wellposedness of the martingale problem.
To show uniqueness the most powerful tool is duality respectively Feynman-Kac duality as presented in Section of \cite{EK86}, which works with our approach only in the case of \textit{strongly recurrent migration},for example in $d=1$. 
Only in that case we can apply the technique of the \textit{Feynman-Kac duality} (a duality where an exponential functinal, as in equation \eqref{e1354} or \eqref{e2610}-\eqref{a2522} appears in the dual expectation, see here \cite{ggr_tvF14} for details) to obtain the needed uniqueness of the solutions of the martingale problem.
Furthermore due to the fact that the continuum space limits for the occupation measures are not given via SPDE's since the states are singular measures the existence problem is also more subtle since we cannot work with conditional dualities so easily (see \cite{ggr_tvF14} or \cite{DG03} for this concept).

In $ d \geq 2 $ the necessary uniform integrability in passing to the continuum space limit fails and the limit expression makes no sense since two path of the migration do not have a joint occupation time.
This is related to the fact that the population is now supported by a set of Hausdorff dimension less than $ d $ (for $ d=2 $, the picture being a bit more subtle) and here one would have to work with an {\em approximate duality}.
Hence we need here a different approach which remains to be developed.
However we can obtain from the historical Dawson-Watanabe process as a functional an  $\U^V$-valued process which then has the generalized branching property. 
However it should be possible to get then directly via the $\U^V$-valued martingale problem.
This is a general problem with continuum spaces and is addressed in forthcoming work in \cite{GSWfoss}.

\paragraph{Outline of the proof section}
In Section~\ref{ss.T:BRANCHING} we prove the criterion and then prove in Section~\ref{S.VERCRIT} that the example from Section~\ref{ss.trvalFD} fits in our framework and satisfies the criterion; this proof is based on some key facts derived beforehand in Section~\ref{ss.2104}-~\ref{ss.pop}.
In Section~\ref{ss.extmc} we give the extensions of the proofs to the spatial models.

\section{Proof of basic criterion: Theorem \ref{T:BRANCHING}}\label{ss.T:BRANCHING}

We prove separately the two parts of the Theorem~\ref{T:BRANCHING}. 

\textbf{\underline{ad \eqref{i:t:1}:}}
We saw in the introduction that the generalized branching property implies that the relation \eqref{eq:4.10.1} holds so we need only the other direction.

Let $t>0$ and $x=x_1 \sqcup^s x_2 \in \S$. Let $(X^{x_i}_t,t)_{t\geq
0}$ be solutions to the martingale problem for $(\tilde{A},\tilde{D},\delta_{(0,x_i)})$, $i=1,2$ 
and $X^{x_1} 
\independent X^{x_2}$. For $f\in \tilde{D}$, i.e.~$f(s,x) = \psi(s)h_s(x),\, s\geq0, \, x \in \S$ 
we will show that
 \begin{align}\label{eq:mg}
  { \left(\psi(t) h_t(X^{x_1}_t \sqcup^t X^{x_2}_t) - \psi(0) h_0(X^{x_1}_0 \sqcup^s X^{x_2}_0) - \int_0^t 
\tilde{A} \psi(r) h_r(X^{x_1}_r \sqcup X^{x_2}_r) \, \dx r \right)_{t \geq 0}}
 \end{align}
defines a martingale and, thus, $(X^{x_1}_t \sqcup^t X^{x_2}_t,t)_{t\geq0}$ 
is a \cadlag solution to the $(\tilde{A},\tilde{D},\delta_{(x_1\sqcup x_2,0)})$ martingale problem. 
Due to the uniqueness assumption for the martingale problem, this implies  $T_t X^{x_1\sqcup 
x_2}_t\overset{\textup{d}}{=} T_t (X^{x_1}_t\sqcup^t X^{x_2}_t)$. 
Then, we get the branching property for any $t\geq0$, $x_1,x_2 \in \S$:
\begin{equation}\label{e1337}
 P_t(x_1\sqcup^s x_2, h_t) \stackrel{\text{def}}{=} \E [h_t(X^{x_1 \sqcup^s x_2}_t)] 
= \E[ h_t(X^{x_1}_t \sqcup^t X^{x_2}_t) ] \stackrel{\text{def}}{=} \left( 
P_t(x_1,\cdot) \ast^t P_t(x_2,\cdot)\right) (h_t), \quad h_t \in D_t\,.
\end{equation}

It remains therefore verify to \eqref{eq:mg}. First set $\psi \equiv 1$.
After the argument it will become clear how to generalize. The proof follows that of Lemma 4.3.4 in \cite{EK86}:
By independence of the two processes for $t_2 > t_1 \geq0$:
\begin{align}\label{e1338}
 \E \left[ \left( h_{t_2}(X_{t_2}^{x_1}) - h_{t_1}(X_{t_1}^{x_1}) - \int_{t_1}^{t_2} \tilde{A} 
h_{r} 
(X_r^{(x_1)}) \, \dx r \right) h_{t_2}(X_{t_2}^{x_2})  | \mcF_{t_1} \right] = 0  \\
 \E \left[ \left( h_{t_2}(X_{t_2}^{x_2}) - h_{t_1}(X_{t_1}^{x_2}) - \int_{t_1}^{t_2} \tilde{A} 
h_{r} 
(X_r^{(x_2)}) \, \dx r \right) h_{t_1}(X_{t_1}^{x_1}) | \mcF_{t_1} \right] = 0, \label{e1338b}
\end{align}
using the filtration $(\mcF_t)_{t\geq0}$, the {\em joint} filtration of $X^{x_1}$ and $X^{x_2}$.
Combined we get
\begin{align}\label{e1339}
 \E & \left[ h_{t_2}(X_{t_2}^{x_1})  h_{t_2}(X_{t_2}^{x_2}) - 
h_{t_1}(X_{t_1}^{x_1})h_{t_1}(X_{t_1}^{x_2}) \right. \\
\label{e1339b} & \qquad \quad \left. - \int_{t_1}^{t_2} h_{t_2}(X_{t_2}^{x_2}) \tilde{A} h_{r} (X_r^{(x_1)})  + 
h_{t_1}(X_{t_1}^{x_1}) \tilde{A} h_{r} (X_r^{(x_2)})  \, \dx r | \mcF_{t_1} \right] = 0.
\end{align}
Using a partition of $[s,t]$, $s = t_0 < t_1 < \dotsm < t_n = t$ we get
\begin{align}\label{e1340}
 \E & \left[ h_{t}(X_{t}^{x_1})  h_{t}(X_{t}^{x_2}) - h_{s}(X_{s}^{x_1})h_{s}(X_{s}^{x_2}) \right. 
\\
 & \qquad \quad \left. - \int_{s}^{t}\tilde{A} h_{r} (X_r^{(x_1)}) h_{r}(X_{r}^{x_2}) + 
h_{r}(X_{r}^{x_1}) \tilde{A} h_{r} (X_r^{(x_2)})  \, \dx r | \mcF_{s} \right] \label{e1340b}\\
& + \sum_{i=1}^n \int_{t_{i-1}}^{t_i} \E \left[ (h_{t_i}(X_{t_i}^{x_2}) -h_r(X_r^{x_2})) \tilde{A}h_r(X_r^{x_1}) + (h_{t_i}(X_{t_i}^{x_1}) - h_r(X_r^{x_1}))  \tilde{A}h_r(X_r^{x_2}) | \mcF_{s} \right] \, \dx r  = 0. \label{e1340c}
\end{align}
Denote by $ F_1(r), F_2(r) $ the functions $ r \rightarrow E[h_r(X^{x_2}_r) \wt Ah_r(X^{x_1}_r) \mid \CF_s] $ respectively with $ x_1,x_2 $ interchanged.
We have here the expression $ h_r(X_s)-h_t(X_u) $, which is evaluated for $ u=s $ and $ r=t $ and hence we need properties of the two compared functions.
Observe that for $\max |t_{k+1}-t_k| \to 0$ the continuity assumptions on $r\mapsto h_r$ and $r \mapsto X_r$ imply that the {\em last term vanishes} leading to \eqref{eq:mg} via \eqref{eq:4.10.1}.
Namely we observe that for as $ max_k \mid t_{k+1} - t_k \mid \rightarrow 0 $ the functions $ F^1, F^2 $ are approximated in $ L_1  $ by $ \wt F^1, \wt F^2 $ given by replacing $ h_r $ by $ h_{t_i} $ in the interval $ [t_i,t_{i+1}] $.
This follows from the continuous differentiability of $ r \rightarrow h_r  $ we assumed and the stochastic continuity of $ X $, together with the conditional independence of $ X^{x_1} $ and $ X^{x_2} $ which allows to rewrite the first expression of the integral in the third term now as:
$ \E [h(X^{x_2}_{t_i})-h_r(X_r^{x_2}) \mid \CF_s] \E [Ah_r(X_r^{x_2}) \mid \CF_s] $ and similar for the second term making the claim immediate.
 
For general $\psi$ the argument proceeds starting in \eqref{e1338} replacing $h_\cdot$ by $\psi(\cdot) h_\cdot$.
\medskip

\textbf{\underline{ad \eqref{i:t:2}:}}
This result is a corollary to \eqref{i:t:1} and it suffices to verify \eqref{eq:4.10.1} having \eqref{eq:lin:gen}.
Drop in $g_{\psi,h}$ the indices and calculate:
 \begin{align}\label{e1341}
  \tilde{A} \psi(t) h_t(x_1 \sqcup x_2) & = \psi'(t) h_t(x_1 \sqcup x_2) + \psi(t) g(t,x_1 \sqcup 
x_2) h_t(x_1 \sqcup x_2) \\
  & = \psi'(t) h_t(x_1 \sqcup x_2) + \psi(t) (g(t,x_1) + g(t, x_2)) h_t(x_1)  h_t( x_2) \label{e1341b}\\
  & = \psi'(t) h_t(x_1) h_t(x_2) + \psi(t) \left( A h_t(x_1) h_t(x_2) + h_t(x_1) Ah_t(x_2) \right).\label{e1341c}
 \end{align}
On the other hand having the branching property we set $g_{\psi,h}=(Ah_t)/h_t \cdot \psi(t)$ to obtain a homomorphism using the multiplicity of $h_t$ the only point to check is that the expression is well defined for a multiplicative function $h_t$, which is the case for $h_t(\cdot) >0$.

%
%

\section{Formulation and proofs of key facts to be used in Section~\ref{S.VERCRIT}} \label{s1501141604b}

In this section we formulate and prove the statements which give the \emph{key tools} used subsequently in our argument that the assumptions needed to apply the \emph{criterion} indeed do hold, see the next section.
In the following proofs we will use the notation $x_1 \sqcup x_2$ to denote the generic $x_1 \sqcup^s x_2$ which is required for the branching property.

\subsection{Formulation of the key tools}\label{ss.2104}
{\em The key point is to verify the uniqueness property of the martingale problem, to choose $ \wt D $ and to calculate the $ g $ in our criterion and show its $ \sqcup- $ additivity} (Step 3).
Everything is put together in Step 4.
The truncation necessary to define the branching properties raises some technical problems in applying stochastic analysis tools. 
Therefore before carrying out the proof of Theorem~\ref{T.TVF.BRAN.PROP}  and working with the duality techniques in Section~\ref{s1501141604b} to establish that we have a well-posed martingale problem we need some preparations (Step 1 and 2).

\paragraph{{\em Step 1: Preparations}} \; 
To prepare the proof of Theorem~\ref{T.TVF.BRAN.PROP} we state some results which are proven later in Section~\ref{ss:pr:dual}.
We recall the notation $\bbD_n = \{\dr \in \R^{n \choose 2}:\, 0 \leq r_{ik} \leq r_{ij} + r_{jk},\, 1\leq i < j < k \leq n\}$ as the subset of mutual distances which can be realized by $n$ points in a metric space, $n \in \N$.
By convention $\bbD_1 = \{0\}$.

We will need in the sequel a function, $ \varrho $ : $ \R \times \bigcup\limits_n \bbD_n \rightarrow \R $, which generates a sliding window of truncation which filters out the information in the cut out pieces, like the truncation operator, but which has {\em smoothness properties} which allow for the calculus of martingale problems. 
We use the notation $\varrho_t(\dr) = \varrho^{(n)}(t,\dr)$, $t \in \R$, $\dr \in \bbD_n$, $n\in \N$.
 
\begin{lemma}\label{L.EX}
Assume that $\varrho^{(n)} \in C^1(\R \times \bbD_n, [0,1])$, $n \in \N$. The process $(t,\mfU_t)_{t\geq 0}$ with $(\mfU_t)_{t \geq 0}$ from Definition~\ref{d.treeval} is a solution to the $(\tilde{\Omega}, D_{\textup{lin}})$ martingale problem, where
 \begin{align}\label{e1316}
  \tilde{\Omega} & = \Omega^{\uparrow} + \partial_t ,\\
  D_{\textup{lin}} & = \{ (t,\mfu) \mapsto \psi(t) \Phi^{n,\phi \varrho_t} (\mfu)\, : \ \psi \in C_b^1(\R,\R),\, 
\phi \in C_b^1 \}.\label{e1316b}
 \end{align}
 Note that $D_{\textup{lin}}$ depends on $\varrho$, but we do not explicitly state that dependence. \qed
\end{lemma}

The lemma will be proved in Section~\ref{ss:pr:dual}.
The next lemma shows that another domain of the operator can be chosen giving an equivalent martingale problem.

\begin{lemma}\label{l.D_lin:tildeD} The following are equivalent:
\begin{enumerate}
\item\label{i.tr311} $(t,\mfU_t)_{t\geq0}$ solves the  $(\tilde{\Omega}, D_{\textup{lin}})$-MGP
\item\label{i.tr312} $(t,\mfU_t)_{t\geq0}$ solves the $(\tilde{\Omega}, \tilde{D})$-MGP, where
 \begin{equation}\label{e.tr6}
  \tilde{D} = \{(t,\mfu) \mapsto \psi(t) \exp (-\Phi^{n,\phi\varrho_t}(\mfu)) :\, \psi \in C^1_b(\R), 
 \, \phi \in C_b^1(\bbD_n),\, n \in\N \}.
 \end{equation}
\end{enumerate}
Note that $\tilde{D}$ depends on $\varrho$, but we do not explicitly state that dependence. \qed
\end{lemma}
This lemma is important since it allows to work with multiplicative functions as required in Theorem~\ref{T:BRANCHING}.

\paragraph{\em Step 2: Sliding window of functions}
Return to the functions $ \varrho $ from above and specialize to the present context.

We define the $ t- $truncated polynomial:
 \begin{equation}\label{e1319}
  \Phi^{n,\phi}_t(\mfu) = \Phi^{n, \phi \cdot c_t}(\mfu) \,, \text{ where } c_t(\dr) = \prod_{1\leq i< j \leq n} \1(r_{ij} < 2t ) \, 
 \end{equation}
 for $\Phi^{n,\phi} \in \Pi$.
 Truncated polynomials are additive on $\U(t)^\sqcup$ (see Theorem~\ref{Inf-p.trunc.poly} in \cite{infdiv}) which makes $\tilde{D}$ a set of {\em multiplicative} functions on $\U(t)^\sqcup$.
 Unfortunately, truncated polynomials $\Phi^{n,\phi}_t$ do not have $C^1$-functions $\phi c_t$ and that is why we use an \emph{approximation argument} which makes use of the following assumptions on $\varrho$.

\begin{assumption}\label{a.1}
 The functions $\varrho^{(n)}: \R \times \bbD_n \to [0,1]$ are in $C^1$ and of the form that for $t \in \R, \dr \in  \bbD_n$:
 \begin{equation}\label{e.tr5} 
  \varrho^{(n)}(t,\dr) = 0 \text{ if and only if there is }1\leq i< j \leq n \text{ with  }r_{ij} \geq2t.
 \end{equation}
 Moreover, we require that $\varrho^{(n)}(t,\cdot)$ is non-increasing in any coordinate for any $t \geq0$. \qed
\end{assumption}
Recall the notation $\underline{\underline{t}} = (t)_{1\leq i < j \leq n} \in \bbD_n$, for an array with the constant entry $t$ in all ${n \choose 2}$ coordinates.
We will not specify the dimension of the array in order not to overload notation.
\begin{assumption}\label{a.2}
 For any $n \in \N$, $ c \geq0$ : $\varrho^{(n)}(t,\dr) = \varrho^{(n)}(t+c,\dr + 2 \underline{\underline{c}})$, $t\in \R, \dr \in \bbD_n$, $n\in \N$. \qed
\end{assumption}
Under these assumptions we get the {\em approximation property} below.
\begin{lemma}\label{l.tr2}
 Let $t>0$. Suppose Assumptions~\ref{a.1} and~\ref{a.2} hold for functions $\varrho = \varrho^{(n)} : 
\R \times \bbR^{n\choose 2} \to [0,1]$, $n\in \N$. 
Then for any truncated polynomial 
$\Phi_t^{n,\phi} \in \Pi(C^1(\D_n))$, we can find a sequence of polynomials $\Phi^{n,\phi_N \varrho_t}$ in the 
family of polynomials $\{\Phi^{n,\phi \varrho_t}: \, \phi \in C_b^1(\bbD_n), n \in \N\}$ such that for 
all $\mfu \in \U$:
 \begin{equation}\label{e1318}
   \Phi^{n, \phi_N \varrho_t} (\mfu) \nearrow \Phi_t^{n,\phi}(\mfu)\,, \text{ as } N \to \infty.
 \end{equation} \qed
\end{lemma}
\begin{proof}[Proof of Lemma~\ref{l.tr2}]\label{pr.l.tr2}
 Recall that $(\phi\cdot \varrho_t) (\dr) = \phi(\dr) \cdot \hat{\varrho}(\dr -2 \underline{\underline{t}})$ by Assumption~\ref{a.2}.
 Let $g_N \in C^1(\R,\R)$ with $g_N|_{[-\infty, 0]} \equiv 0$ and $g_N|_{[N^{-1},\infty)} \equiv 1$.
 Define
 \begin{equation}\label{e1320}
  \phi_N (\dr ) = \frac{\phi(\dr)}{\varrho_t(\dr)} \cdot \prod_{1\leq i<j \leq n} g_N(2t - r_{ij})  \, .
 \end{equation}

 Clearly, $\phi_N \in C^1_b(\bbD_n)$ for any $N \in \N$.  
 This is the case since $\varrho^{(n)}(t,\cdot)$ is decreasing by Assumption~\ref{a.1}.
 With a similar argument as that of {\em Lemma~\ref{Inf-l.tr1}} in \cite{infdiv} we can see that $\Phi^{n,\phi_N \cdot \varrho_t} (\mfu) \to \Phi^{n,\phi\cdot h_t} (\mfu) $ as $N \rightarrow \infty) $.
\end{proof}

This completes our technical preparations and we have to check later the basic assumptions to work with our approach in this model.

\begin{proposition}[Feynman-Kac duality and uniqueness~\cite{ggr_tvF14}]\label{p.dual.ext}
\mbox{}\\
Under Assumptions~\ref{a.1} and~\ref{a.2} there is a dual process for the process 
$(t, \mfU_t)_{t\geq0}$ with a Feynman-Kac duality relation. If the initial condition $\P [\mfU_0 \in \cdot]$ is deterministic, then uniqueness holds for the $(\tilde{\Omega}, \tilde{D})$ martingale problem 
in the sense that for any other solution $(\mfU_t',t)_{t\geq0}$ we have 
$\mfU_t(t)\eqd\mfU_t'(t)$, for every $t>0$. \qed
\end{proposition}
This means that any two solutions at time $t$ have the same $t$-top.
One can see that as a one-dimensional uniqueness result. 
It is not surprising that we do not obtain a finer result: Assumption~\ref{a.1} cuts off information beyond that level.
The uniqueness result also holds more general if we require moment bounds on the initial conditions.

Next, the formula for $ g $ and the important property of $\sqcup^t$-additivity is stated in the next proposition. This is the key for the proof of Theorem~\ref{T.TVF.BRAN.PROP}.

\begin{proposition}[Key formula for $\U$-valued branching]\label{l.tr5}
\mbox{}\\
 For $f = \psi e^{-\Phi} \in \tilde{D}$ of Lemma~\ref{l.D_lin:tildeD}  with $\varrho$ satisfying Assumptions~\ref{a.1} and~\ref{a.2}: 
 \begin{equation}\label{e1317}
  \tilde{\Omega} f(t,\mfu) = \psi'(t) e^{-\Phi(\mfu)} + \psi(t) g_\phi(t,\mfu) e^{-\Phi(\mfu)} ,
 \end{equation}
 where 
\begin{equation}\label{e1032}
g_\phi(t,\mfu)= \Omega^{\uparrow, \rm grow}   \Phi^{n,\phi \varrho_t} (\mfu) + \frac{bn}{2\bar{\mfu}} \Phi^{2n, (\phi\varrho_t) \times (\phi \varrho_t) \circ \theta_{1,n+1}} (\mfu)
\end{equation} 
is $\sqcup^t$-additive.
 Moreover, under Assumption~\ref{a.1} the function $\U(t)^\sqcup \mapsto \R,\, \mfu \mapsto \Phi^{n,\phi \varrho_t} (\mfu)$ is additive for any $\phi \in C(\R^{n\choose 2})$, $n \in \N$. \qed
\end{proposition}

The proofs of all the previous results and in particular the {\em calculation of $ g_\phi $} are contained in Subsection~\ref{ss:pr:dual} and~\ref{ss.pop}.

\subsection{Proof of tools: Lemma~\ref{L.EX},~\ref{l.D_lin:tildeD} and Proposition~\ref{l.tr5}}\label{ss:pr:dual}

\begin{proof}[Proof of Lemma~\ref{L.EX}]\label{p.L.EX}
 The proof checks the conditions of Lemma 4.3.4 in \cite{EK86} giving the claim. Define the following functions in our case not depending on $\omega$, but we use the notation of the reference).
 \begin{align}\label{e1342}
  u:& \begin{cases}
      [0,\infty) \times \U  \times \Omega & \to \R \\
      (t,\mfu,\omega) & \mapsto \psi(t) \Phi^{n,\phi \varrho_t} (\mfu) \, ,
     \end{cases}\\
  v:& \begin{cases}
      [0,\infty) \times \U  \times \Omega & \to \R \\
      (t,\mfu,\omega) & \mapsto \psi'(t) \Phi^{n,\phi \varrho_t} (\mfu) + \psi(t) \Phi^{n,\phi 
\partial_t \varrho_t} (\mfu) \, ,
     \end{cases} \label{e1342b}\\
  w:& \begin{cases}
     [0,\infty) \times [0,\infty) \times \U  \times \Omega & \to \R \\
      (t,s,\mfu,\omega) & \mapsto \psi(t) \left( \Phi^{n,2 \overline{\nabla} (\phi \varrho_t)} (\mfu) +
      an \Phi^{n,\phi \varrho_t}(\mfu) +
\frac{b}{\bar{\mfu}} \sum_{1\leq k<l\leq n}\Phi^{n,\phi \varrho_t \circ \theta_{k,l}} (\mfu) \right) 
\,.\label{e1342c}
     \end{cases}
 \end{align}
The following {Assumptions (4.3.10) and (4.3.11)} of the reference are by the construction of the process (note we vary here either only in the explicit time coordinate or in the state of the genealogy)
satisfied for $s\leq t$:
\begin{align}\label{e1343}
 \E \left[ u(t,\mfU_t) - u(s,\mfU_t) - \int_s^t v(r,\mfU_t) \, \dx r | \mcF_s \right] = 0, \\
 \E \left[ u(s,\mfU_t) - u(s,\mfU_s) - \int_s^t w(s,r,\mfU_r)\, \dx r | \mcF_s \right] = 0. \label{e1343b}
\end{align}
Additionally, $\mfU$ is right-continuous by construction and $\mfu \mapsto v(t,\mfu,\omega)$ is 
continuous for fixed $t, \omega$, since $\phi \varrho_t$ and $\phi \partial_t \varrho_t$ are continuous 
and so the terms involved in $v$ are classical polynomials. 
The left-continuity of $t \mapsto 
w(t,s,\mfu,\omega)$ is clear by continuity of $t \mapsto \overline{\nabla}\varrho_t$ and $t\mapsto 
\varrho_t$. 

Note that our functions are {\em not bounded} as required in \cite{EK86}. However, to obtain 
convergence in their equation (4.3.17), the necessary {\em integrability criteria} follow by a dominated convergence argument with 
moment conditions on $\sup_{s\leq t} \bar{\mfU}_s$ and the bounds on $\|\psi\|,\|\psi'\|,\|\phi\|, 
\|\partial_t \varrho_t\|,\|\overline{\nabla}\varrho \|,$ etc. More concrete: boundedness of $v$ can be 
replaced by $\E[\sup_{s\in[t_1,t_2]} |v(s,\mfU_s)| ] < \infty$ and similarly for $w$.
This is a consequence of properties of the total mass process, the Feller diffusion, where all moments exist for all $ t $  starting in a fixed point  and the total mass process is a semi-martingale.

We now calculate the generator action of the Markov process which was introduced in Lemma~\ref{L.EX}, namely for $f = \psi\Phi \in D_{\textup{lin}}$:
\begin{align}\label{e1344}
 \Omega^{\uparrow}  \psi(t) \Phi^{n,\phi \varrho_t} (\mfu) &= \psi(t) \Big( an \Phi^{n,\phi \varrho_t}(\mfu) + \frac{b}{\bar{\mfu}} 
\sum_{1\leq k<l\leq n}\Phi^{n,\phi \varrho_t \circ \theta_{k,l}} (\mfu) \\
& \phantom{xxxxxxxxxxxxxxxxxxxxx} + \Phi^{n,2(\overline{\nabla} \phi) \varrho_t} (\mfu) +  \Phi^{n,2\phi (\overline{\nabla} \varrho_t)} (\mfu) \Big) \, , \nonumber \\
\label{e1344b} \partial_t \psi(t) \Phi^{n,\phi \varrho_t} (\mfu) &= \psi'(t) \Phi^{n,\phi \varrho_t} (\mfu) + \psi(t) 
\Phi^{n,\phi \partial_t \varrho_t} (\mfu) \, . 
\end{align}

combining \eqref{e1342}-\eqref{e1344b} gives the claim of the lemma.
\end{proof}

\begin{proof}[Proof of Lemma~\ref{l.D_lin:tildeD}]\label{pr.l.D_lin}
We start with showing \eqref{i.tr312}.
 Before we start we have to use the definition of $\Omega$ on $D$ to obtain now the action of the operator which is induced on other functions, namely $\tilde{\Omega}$ acting on $\tilde{D}$.
 This will give us then the compensator for the process $exp(-\Phi^{n,\phi}(\mfU_t))$, allowing us to show that the following process $(M^{n, \phi}_t)_{t \geq 0}$ is a martingale (here $ \phi \times \phi $ refers to the function $ \phi $ evaluating two independent samples):
\begin{align}\label{e1345}
  M^{n,\phi}_t := & \exp(-\Phi^{n,\phi}(\mfU_t)) - \exp(-\Phi^{n,\phi}(\mfU_0) \\
  & - \int_0^t \exp(-\Phi^{n,\phi}(\mfU_t)) \left[ \Omega^{\uparrow, \rm grow} \Phi^{n,\phi}(\mfU_s) + \frac{n b}{2 \bar{\mfU}_s} \Phi^{2n, \phi \times \phi \circ \theta_{1,n+1}}(\mfU_s) \right] \, \dx s , \, t\geq0. \label{e1345b}
\end{align}

 Therefore, we obtain the operator $\wt \Omega=\Omega^\uparrow + \frac{\partial}{\partial t}$ on $\wt D$ by calculating:
\begin{align}
  \label{e.tr8}\Omega^{\uparrow}  \psi(t) \exp(-\Phi^{n,\phi \varrho_t}) (\mfu) &= \psi(t) \exp(-\Phi^{n,\phi \varrho_t} (\mfu)) \Big( \Omega^{\uparrow, \rm grow}   \Phi^{n,\phi \varrho_t} (\mfu) + \frac{bn}{2\bar{\mfu}} \Phi^{2n, \phi\times \phi \circ \theta_{1,n+1}} (\mfu) \Big) \, ,\\
  \label{e.tr9} \partial_t \psi(t) \exp(-\Phi^{n,\phi \varrho_t}) (\mfu) &= \psi'(t) \exp(-\Phi^{n,\phi \varrho_t}) (\mfu) + \psi(t) 
  \exp(-\Phi^{n,\phi \varrho_t}(\mfu))  \cdot \Phi^{n,\phi \partial_t \varrho_t} (\mfu) \, .
\end{align}

 Arguing as in the proof of Lemma~\ref{L.EX} one gets then indeed,
\begin{equation}\label{e.tr10}\begin{split}
 N_t^{\psi,\phi} :=  \psi(t) & \exp(- \Phi^{n,\phi \varrho_t}(\mfU_t)) - \psi(0) \exp(-\Phi^{n,\phi \varrho_t}(\mfU_s)) \\
  & - \int_0^t \Omega^{\uparrow}  \psi(t) \exp(-\Phi^{n,\phi \varrho_t}) (\mfU_s) + \partial_t \psi(t) \exp(-\Phi^{n,\phi \varrho_t}) (\mfU_s) \, \dx s, \, t \geq0
  \end{split}
\end{equation}
defines a martingale.
This is \eqref{i.tr312}.

We therefore have to verify \eqref{e.tr8} and \eqref{e.tr9}.
The second is simply calculus. Next to the first.
A standard calculation as in Corollary 2.13 of \cite{GPWmp13} shows this via the \Ito-formula.
Namely we expand the exponential and calculate the compensator of $\Phi(\mfU_t)$ and the quadratic variation of $(\Phi(\mfU_t))_{t \geq 0}$, where the latter requires to calculate the compensator of $(\Phi^2(\mfU_t))_{t \geq 0}$.
Follow here \cite{GPWmp13} equation (8.1) and (8.3) to get this.

Now \eqref{i.tr311} follows by differentiation of $N_t^{\psi,\lambda \phi}$ w.r.t.~$\lambda$ at $\lambda = 0$.
\end{proof}

\begin{proof}[Proof of Proposition~\ref{l.tr5}]\label{pr.l.tr5}
 To get the action of $\tilde{\Omega}$ on $\tilde{D}$ we need to add up terms in \eqref{e.tr8} and \eqref{e.tr9}.
 Using Assumption~\ref{a.2} as before this allows to eliminate the terms with $\nabla \varrho_t$ and $\partial_t \varrho_t$ to obtain:
\begin{align}
\label{e.tr7} \tilde{\Omega}  \psi(t) \exp(-\Phi^{n,\phi \varrho_t}) (\mfu) & = \psi(t) \exp(-\Phi^{n,\phi \varrho_t}) (\mfu) \Big(  \Omega^{\uparrow, \rm grow}   \Phi^{n,\phi \varrho_t} + \frac{bn}{2\bar{\mfu}} \Phi^{2n, \phi\varrho_t \times \phi \varrho_t \circ \theta_{1,n+1}} (\mfu) \Big) \\
& \quad + \psi'(t) \exp(-\Phi^{n,\phi \varrho_t}) (\mfu) \, .\label{e.tr7b}
\end{align}
Using the notation of the lemma we need to show that the expression below is $\sqcup^t$-additive:
\begin{equation}\label{e1347}
 g(t,\mfu) = \Omega^{\uparrow, \rm grow}   \Phi^{n,\phi \varrho_t} (\mfu) + \frac{bn}{2\bar{\mfu}} \Phi^{2n, (\phi\varrho_t) \times (\phi \varrho_t) \circ \theta_{1,n+1}} (\mfu) =: g_1(t,\mfu) + g_2(t,\mfu).
\end{equation}

We see that the expressions are truncated polynomials.
It is elementary using Assumption~\ref{a.1} and Proposition~\ref{Inf-p2907131206} in \cite{infdiv} to establish that $g_1(t,\cdot)$ is $\sqcup^t$-additive. Similar reasoning applies for $g_2$. 
Again by Proposition~\ref{Inf-p2907131206} in the same reference we see with Assumption~\ref{a.1} that $\mfu \mapsto \exp(-\Phi^{n,\phi \varrho_t}(\mfu))$ is multiplicative on $\U(t)^\sqcup$ and note here that the operator of our process relates via \eqref{e1307} to the operator of Fleming-Viot in the reference via $\bar \mfu \Omega^\uparrow-Id$.
\end{proof}

\subsection{Proof of duality: Proposition \ref{p.dual.ext}}\label{ss.pop}
Here we will derive a duality related for the operator 
$(\tilde{\Omega}, D_{\textup{lin}})$.
This allows to deduce {\em uniqueness} for the $(\tilde{\Omega},\tilde{D})$ 
martingale problem.

It is easy to verify the following preparatory lemma.
\begin{lemma}\label{l.tr3}
 Under Assumption~\ref{a.1} and~\ref{a.2} the function $\varrho^{(n)} \in C^1(\R_+ \times \bbD_n, [0,1])$ is of the form
 \begin{equation}\label{e1346}
  \varrho^{(n)} (t, \dr) = \hat{\varrho}^{(n)}(\dr -2 \underline{\underline{t}}) , \ t \geq0, \dr \in \bbD_n, n \in \N \,,
 \end{equation}
 for a function $\hat{\varrho}^{(n)} \in C^1(\bbD_n,[0,1])$ with $\hat{\varrho}^{(n)} |_{\bbD_n \setminus (0,\infty)^{n \choose 2}} \equiv 0$ and positive on $(0,\infty)^{n \choose 2}$. Moreover $(\partial_t + 2 \bar{\nabla} ) \varrho^{(n)} (t,\dr) = 0$. \qed
\end{lemma}

Using the two assumptions we can thus derive simpler expression statement for the action of $\tilde{\Omega}$:
\begin{align}
\label{e.tr4} \tilde{\Omega}  \psi(t) \Phi^{n,\phi \varrho_t} (\mfu) &= \psi(t) \Big( bn \Phi^{n,\phi \varrho_t}(\mfu) + \frac{b}{\bar{\mfu}} 
\sum_{1\leq k<l\leq n}\Phi^{n,(\phi \varrho_t) \circ \theta_{k,l}} (\mfu) \\
& \phantom{xxxxxxxxxxxxxxxxxxxxx} + 
\Phi^{n,2(\overline{\nabla} \phi) \varrho_t} (\mfu) \Big) + \psi'(t) \Phi^{n,\phi \varrho_t} (\mfu) \, .\label{e.tr4b}
\end{align}
This makes the function of the two assumptions clear .

We restrict ourselves to the case that $a=0$, i.e.~there is no drift; all calculations can be done without that restriction, see Section~\ref{ss.2104}, Step 2 in \cite{ggr_tvF14} for details.
All this is done for a fixed function $\hat{\varrho}^{(n)}$ as in Lemma~\ref{l.tr3}, i.e.~Assumptions~\ref{a.1} and~\ref{a.2} are fulfilled and \eqref{e.tr4} can be used.

Define the following set of functions:
\begin{equation}\label{e1348}
 H^{\phi,\psi} : \begin{cases} 
                  (\bbU\times \R_{\geq0}) \times (\bbS \times \bbR^{\bbN \choose 2}\times 
\R_{\geq0})  \to \R \\
                  ((\mfu,t), (p,\dr',s))  \mapsto  \psi(t+s) \int \mu^{\otimes \# p}(\dx 
\underline{x}_p)\, \phi \left( \dr^p(\underline{x}_p) +  \dr' \right) 
\varrho(2(\underline{\underline{t}} + \underline{\underline{s}}) - \dr^p(\underline{x}_p) -  \dr').
                 \end{cases}
\end{equation}
Define for $p(n) = \{\{1\},\dotsc, \{n-1\},\{n,n+1,\dotsc \} \}$ the following sets of functions:
\begin{equation}\label{e1349}
 \mcH = \{ H^{ \phi,\psi}(\cdot, \cdot):\ \phi \in C_b^1(\bbR^{\N \choose 2}) \text{ with finite 
support},\, \psi \in C_b^1(\R_{\geq0}) \},
\end{equation}
\begin{align}\label{e1350}
 \mcG^{\uparrow} &= \{ H^{ \phi,\psi}(0, (p(n),\underline{\underline{0}})):\  n \in \N, \psi 
\in C_b^1(\R_{\geq0}), \phi \in C_b^1(\bbR^{\N \choose 2}) \text{ with finite support} \} \qquad 
\text{and}\\
 \mcG^{\downarrow} & = \{ H^{ \phi, \psi}((t,\mfu), \cdot):\ \mfu \in \bbU,\, t\geq0, \phi \in 
C_b^1(\bbR^{\N \choose 2}) \text{ with finite support}, \psi \in C_b^1(\R_{\geq0}) \} \, .\label{e1350b}
\end{align}
Then $D_{\textup{lin}} = \mcG^\uparrow.$ 

Next turn to the dual process. For a function $G: \bbK \times \R_{\geq0} \to \R$ depending on only
finitely many coordinates define
\begin{align}\label{e1351}
 L^{\downarrow,\textrm{grow}} G(s, p,\dr') &= \partial_s G(s,p,\dr') + \sum_{i \nsim_p j} 
\frac{\partial}{\partial r_{ij}'} G(s,p,\dr') \qquad \text{and} \\
 L^{\downarrow,\textrm{coal}} G(s,p,\dr') &= b \sum_{\pi, \pi' \in p} 
\left(G(s,\kappa_p(\pi,\pi'),\dr') - G(s,p,\dr') \right),\label{e1351b}
\end{align}
for $p\in \bbS$, $\dr' \in \bbR^{N \choose 2}$ and $s\geq0$. 
Here $ \chi_p(\pi,\pi^\prime) $ is the partition where $ \pi $ and $ \pi^\prime $ are replaced by their union.

The coalescent operator is now:
\begin{equation}\label{e1352}
 L^{\downarrow,\textrm{K}} = L^{\downarrow,\textrm{grow}} + L^{\downarrow,\textrm{coal}}.
\end{equation}
One obtains readily for the Kingman coalescent a marked ultrametric measure space which gives a solution to a martingale problem related to that operator.
\begin{lemma}\label{lem:K+time:ex}
Let $n\in \N$ and let $(s,k_s)_{s\geq0}$ be the time-space genealogy-valued Kingman coalescent started in 
$(0,p(n),\underline{\underline{0}})$ defined on page 809 of \cite{GPWmp13}. Then the process 
$(s,k_s)_{s\geq0}$ is a solution of the martingale 
problem for $(\delta_{(0,p(n),\underline{\underline{0}})}, 
L^{\downarrow,\textrm{K}}, \cdot, \mcG^{\downarrow})$. \qed
\end{lemma}
\begin{proof}
Follow Lemma 4.3.4 in \cite{EK86}.
Let for any fixed $(t,\mfu) \in \R_{\geq0} \times \U$:
 \begin{align}\label{e1353}
  u:& \begin{cases}
      [0,\infty) \times \bbS \times \R^{\N \choose 2}  \times \Omega  \to \R \\
      (s,(p,\dr'),\omega)  \mapsto  \psi(t+s) \int \mu^{\otimes \# p}(\dx \underline{x}_p)\, \phi 
\left( \dr^p(\underline{x}_p) +  \dr' \right) \varrho(2(\underline{\underline{t}} + 
\underline{\underline{s}}) - \dr^p(\underline{x}_p) -  \dr') ,
     \end{cases}\\
  v:& \begin{cases}
      [0,\infty) \times \bbS& \times \R^{\N \choose 2}  \times \Omega  \to \R \\
      (s,(p,\dr'),\omega) &\mapsto   \psi'(t+s) \int \mu^{\otimes \# p}(\dx \underline{x}_p)\, \phi 
\left( \dr^p(\underline{x}_p) +  \dr' \right) \varrho(2(\underline{\underline{t}} + 
\underline{\underline{s}}) - \dr^p(\underline{x}_p) -  \dr'), \\
                         & \qquad + \psi(t+s) \int \mu^{\otimes \# p}(\dx \underline{x}_p)\, \phi 
\left( \dr^p(\underline{x}_p) +  \dr' \right) \partial_s \varrho(2(\underline{\underline{t}} + 
\underline{\underline{s}}) - \dr^p(\underline{x}_p) -  \dr')
     \end{cases} \label{e1353b}\\
  w:& \begin{cases}
     [0,\infty) \times [0,\infty) \times \bbS \times \R^{\N \choose 2}  \times \Omega  \to \R \\
      (s,\vartheta, (p,\dr'),\omega)  \mapsto (L^{\downarrow,\text{coal}} + \sum_{i\nsim_p j} 
\frac{\partial}{\partial r_{ij}'} ) H((t,\mfu),(s,p,\dr')).
     \end{cases}\label{e1353c}
 \end{align}
To check (4.3.10) and (4.3.11) in \cite{EK86} use calculations as in
Section 4 in \cite{GPWmp13}.

Additionally, $(k_s)_{s\geq0} = (p_s,\dr_s')_{s\geq0}$ is right-continuous by construction and 
$(p,\dr') \mapsto v(s,p,\dr',\omega)$ is continuous for fixed $s, \omega$: Continuity in $p$ is 
obvious, since $\bbS$ is a discrete space. Continuity in $\dr'$ is true,  $(\dr',\dr) \mapsto 
\phi(\dr + \dr') \varrho(2(\underline{\underline{t}}-\underline{\underline{s}})-\dr - \dr')$ is 
continuous and bounded. This allows to apply dominated convergence to get the continuity in $\dr'$.

The left-continuity of $s \mapsto w(s,\vartheta,p,\dr',\omega)$ is clear by continuity of $s \mapsto 
\overline{\nabla}\varrho_s$ and $s\mapsto \varrho_s$.
Note that now our functions involved are {\em not bounded} as required in \cite{EK86}. However, $(\# 
p_s)_{s\geq0}$ is decreasing and so for any initial state the convergence in (4.3.17) can be shown.
\end{proof}

\begin{lemma}[Feynman-Kac Duality]\label{l.tr1}
 Let $k=(s,k_s)_{s\geq0}$ be a solution of the $(\delta_{(0,p(n),\underline{\underline{0}})}, 
L^{\downarrow,\textrm{K}}, \cdot, \mcG^{\downarrow})$ martingale problem. Let $(t,\mfU_t)_{t\geq0}$ 
be a solution of the $(\tilde{A},D_{\textup{lin}})$ martingale problem started in $\mbfP_0 \otimes 
\delta_0 \in \mcM_1(\U \times \R)$. Then for all $H^{\phi,\psi} \in \mcH$ the following Feynman-Kac duality holds:
 \begin{equation}\label{e1354}
  \E_{\mbfP_0} \left[ H^{\phi,1} ((t,\mfU_t),(0,p,\underline{\underline{0}})) \right] = \E_{(0,p,\underline{\underline{0}})} \left[ H^{\phi,1}(\mfu, (t,p_t,\dr_t')) e^{\int_0^t {\# p_s \choose 2} \, ds} \right].
  \end{equation} \qed
\end{lemma}
\begin{proof}\label{pr.1860}
We show the {\em generator criterion for duality} relations namely (4.4.41) in \cite{EK86}:
\begin{align}\label{e1355}
 ( L^{\downarrow,\textrm{grow}} + \partial_s ) H^{\phi,\psi} ((t,\mfu), (s,p,\dr')) =  
(\Omega^{\uparrow,\textrm{grow}} + \partial_t) H^{\phi,\psi} ((t,\mfu), (s,p,\dr')) , \\
 \left(L^{\downarrow,\textrm{bran}} + {\# p \choose 2} \right) H^{\phi,\psi} ((t,\mfu), (s,p,\dr'))  
=  \Omega^{\uparrow,\textrm{bran}}  H^{\phi,\psi} ((t,\mfu), (s,p,\dr')) .\label{e1355b}
\end{align}
The latter line holds as in the case without marks. To prove the first statement is the same as without marks.
Note that the exponential term is bounded and no integrability problems arise.
\end{proof}

%

%
\begin{proof}[Proof of Proposition \ref{p.dual.ext}]\label{pr.p.dual.ext}
Let $(\mfU_t,t)_{t\geq0}$ and $(\mfU_t',t)_{t\geq0}$ be two solutions of the 
$(\tilde{\Omega},\tilde{D})$ martingale problem. By Lemma~\ref{l.D_lin:tildeD} it is then also a 
solution to the $(\tilde{\Omega},D_{\textup{lin}})$ martingale problem. Fix $\phi \in C^1_b(\bbD_n)$. Then, by duality in Lemma~\ref{l.tr1} we have for 
$p=\{\{1\},\dotsc,\{m-1\},\{m,m+1,\dotsc\}\}$ and $\psi \equiv 1$:
\begin{align}\label{e1356}
  \E_{\mbfP_0} \left[ H^{\phi,1} ((t,\mfU_t),(0,p,0)) \right] &= \E_{(p,0,0)} \left[ H^{\phi,1}(\mfu, (t,p_t,\dr_t')) e^{\int_0^t {\# p_t \choose 2} \, ds} \right] \\ 
  &= \E_{\mbfP_0} \left[ H^{\phi,1} ((t,\mfU_t'),(0,p,0)) \right] .\label{e1356b}
 \end{align}
Hence
\begin{equation}\label{e1357}
 \E_{\mbfP_0} \left[ \Phi^{m,\phi \varrho_t}(\mfU_t) \right] = \E_{\mbfP_0} \left[ \Phi^{m,\phi 
\varrho_t}(\mfU_t') \right].
\end{equation}
Above holds for \emph{any} $\phi \in C^1_b(\bbD_n)$. By Lemma~\ref{l.tr2} and monotone convergence this implies
\begin{equation}\label{e1358}
 \E_{\mbfP_0} \left[ \Phi_t^{m,\phi}(\mfU_t) \right] = \E_{\mbfP_0} \left[ \Phi_t^{m,\phi}(\mfU_t') \right].
\end{equation}
But $\Phi^{m,\phi}_t(\mfU_t) \leq \|\phi\|_\infty \bar{\mfU}_t^m$ and the bound $\E[\bar{\mfU}_t^m] \leq c(t) (\mbfP_0[\bar{\mfU}_0^m] + 1)$ holds by the classical estimates for the Feller diffusion for some function $t \mapsto c(t)< \infty$. 
Therefore we continue with the following statement following the standard argument combining Proposition 2.6. in \cite{GPW09}, Proposition 4.4.6. (page 115) in \cite{EK86} and the discussion after equation 4.4.21. therein:
\begin{lemma}\label{lem:Pi:separating}
The algebra generated by $\Pi$ is separating on 
\begin{equation} \label{tv4}
\tilde{\mcM} = \left\{ P \in 
\mcM_1(\bbU): \, \limsup_{K\to \infty} \frac{1}{K} \left( \int \bar{\mfu}^{nK} \, P (\dx \mfu) 
\right)^{1/K}<\infty \;\forall n\in\bbN\right\}. 
\end{equation}
The algebra generated by $\Pi$ is convergence determining, whenever the limit point is in $\tilde{\mcM}$. \qed
\end{lemma}
We conclude with this Lemma that:
\begin{equation}\label{e1359}
 \lfloor \mfU_t \rfloor (t) \eqd \lfloor \mfU_t'\rfloor(t) \, .
\end{equation}
This means that the $t$-tops of the both processes coincide in law at each time $t\geq0$.
\end{proof}
%

\section{Verification of the criterion in our models: Theorem~\ref{T.TVF.BRAN.PROP}}\label{S.VERCRIT}

The application of our criterion requires to check whether the list of assumptions in Theorem~\ref{T:BRANCHING} can be verified in a given situation.
We first give a detailed proof for the $\U$-valued Feller diffusion and then later provide the needed modifications for the spatial case the $\U^E-valued$ super random walk in Section~\ref{ss.extmc}.

\subsection{Verification of criterion for $\U$-valued Feller: Theorem~\ref{T.TVF.BRAN.PROP}}\label{ss.theoftheo}

\paragraph{\em Step 1: Why the setup applies.}
Due to Lemma~\ref{l.tr2} we make use of truncated polynomials in the martingale problem of Lemma~\ref{l.D_lin:tildeD}, to get a \emph{multiplicative domain}.
Two things are missing for the application of Theorem~\ref{T:BRANCHING}:
a uniqueness result and the linearity of the generator.
Both these facts are provided by Proposition~\ref{p.dual.ext} and Proposition~\ref{l.tr5}.

\begin{proof}[{\bf {\em Step 2}} Proof of Theorem \ref{T.TVF.BRAN.PROP}]\label{pr.t.tvF.bran}
\noindent
 We want to use Theorem~\ref{T:BRANCHING}. Fix a function $\varrho$ as in Lemma~\ref{L.EX} which satisfies Assumptions~\ref{a.1} and~\ref{a.2}.
 We have $\S_t = \U(t)^\sqcup$, which is a semigroup by Proposition~\ref{Inf-p.delphic} in \cite{infdiv}.
 For $s \leq t$ we have that $S_s \subset S_t$ trivially and the embedding is topologically consistent.
 Therefore we are in the setting of Section~\ref{s1501141603}.
 Moreover, Proposition~\ref{l.tr5} says that $D_t = \{  \mfu \mapsto \exp(-\Phi^{n,\phi \varrho_t}(\mfu)): \, \phi \in C(\R^{n\choose 2}),\, n \in \N \}$ is a $t$-multiplicative family on $\S_t$.
 The uniqueness of the martingale problem $(\tilde{A},\tilde{D})$, where uniqueness is understood in 
the sense specified in Theorem~\ref{T:BRANCHING} and $\tilde{D}$ is defined as in \eqref{e.tr6}, 
was shown in Proposition~\ref{p.dual.ext}. 
 It remains to establish the additivity of the generator in \eqref{eq:lin:gen} with respect to the multiplicative functions in $D_t$.
 But this is Proposition~\ref{l.tr5}.
\end{proof}

\subsection{Proof of Theorem \ref{t.recurfeller} }\label{ss.truncMGP}

We work here with the truncated martingale problem. First some preparations.

\begin{assumption}\label{a.3}
 There is a function $\hat{\varrho} \in C^1 (\R\times \R, \R)$ such that for all $n \in \N$:
 \begin{equation}\label{e1361}
  \varrho^{(n)}(t,\dr) = \prod_{1\leq i < j \leq n} \hat{\varrho}(t,r_{i,j}) \, .
 \end{equation} \qed
\end{assumption}

The following lemma is easy to verify.
\begin{lemma}
 Under Assumptions~\ref{a.1},~\ref{a.2} and~\ref{a.3} there is a function $\hat{\varrho} \in C^1(\R,\R)$ with $\hat{\varrho}|_{[0,\infty)} = 0 $ and 
 \begin{equation}\label{e1362}
  \varrho^{(n)} (t, \dr) = \prod_{1\leq i < j\leq n} \hat{\varrho}(r_{i,j} - 2 \underline{\underline{t}}) \, , \text{ for all } t, \dr, n \, .
 \end{equation} \qed
\end{lemma}

\begin{proof}[Proof of Theorem~\ref{t.recurfeller}]\label{pr.t.recurf}
 We will use Assumptions~\ref{a.1},~\ref{a.2} and~\ref{a.3} for the truncation functions $\varrho^{(n)}$.
 Using Assumption~\ref{a.2} and consider $\psi \equiv 1,\, \phi \equiv 1$ by \eqref{e.tr4}
 \begin{align}\label{e1363}
  \tilde{\Omega}  \psi(t) \Phi^{n,\phi \varrho_t} (\mfu) &=  an \Phi^{n, \varrho_t}(\mfu) + \frac{b}{\bar{\mfu}} 
\sum_{1\leq k<l\leq n}\Phi^{n, \varrho_t \circ \theta_{k,l}} (\mfu) \, .
 \end{align}
Under Assumption~\ref{a.3} then we can make the following elementary calculation for $1\leq k < l \leq n$:
\begin{align}\label{e.tr1}
 \varrho_t^{(n)} \circ \theta_{k,l} (\dr) & = \prod_{1\leq i< j \leq n} \hat{\varrho}\left( (\theta_{k,l} (\dr))_{i,j} - 2t  \right) \\
 & = \prod_{1\leq i< j \leq n, i\neq l, j \neq l} \hat{\varrho}(r_{i,j} - 2t) \cdot \prod_{i\neq l} \hat{\varrho}(r_{k\wedge i, k \vee i} -2t)  \, .\label{e.tr1b}
\end{align}
This is a function which is independent of $r_{ij}$ where either $i=l$ or $j = l$. Thus, $\frac{b}{\bar{\mfu}} 
\sum_{1\leq k<l\leq n}\Phi^{n, \varrho_t \circ \theta_{k,l}} (\mfu)$ is a polynomial of degree $n-1$.
Using Proposition~\ref{p.dual.ext} we know that
 \begin{equation}\label{e1601141601}
  \Phi^{n,\varrho_t} (\mfU_t) - \Phi^{n,\varrho_0} (\mfU_0) - \int_0^t b{n \choose 2} 
\Phi^{n-1,T(\varrho_s)} (\mfU_s) + na \Phi^{n,\varrho_s}(\mfU_s) \, \dx s
 \end{equation}
is a martingale, where $T(\varrho_s)$ is the transformation outlined in \eqref{e.tr1}. 
More precisely
\begin{equation}\label{e1364}
 \sum_{1\leq k < l \leq n} \Phi^{n,\varrho_t \circ \theta_{k,l}}(\mfu) = {n \choose 2} \bar{\mfu} \int \nu^{n-1, \mfu} (\dx \dr) \, \varrho_t^{n-1}(\dr) \prod_{2\leq i \leq m-1} \hat{\varrho} (r_{1,i} -2t) \hat{\varrho}(-2t) \, .
\end{equation}
In particular for $n=2$ the following is a a martingale
\begin{equation}\label{e1365}
  \Phi^{2,\varrho_t} (\mfU_t) - \Phi^{2,\varrho_0} (\mfU_0) - \int_0^t b \hat{\varrho}(-2s) \Phi^{1,\varrho_s^{(1)}} (\mfU_s) + 2a \Phi^{2,\varrho^{(2)}_s}(\mfU_s) \, \dx s  .
\end{equation}
Note that $\Phi^{1,\varrho_t^{(1)}}(\mfu) = \hat{\varrho}(-2t) \bar{\mfu}$.

Thus, we obtain the following ODE:
\begin{equation}\label{e.tr2}\begin{split}
 \partial_t \E[ \Phi^{2,\varrho^{(2)}_t}(\mfU_t) ] &= b \hat{\varrho}(-2t) \hat{\varrho}(-2t) \E[\bar{\mfU}_t] + 2a \E[ \Phi^{2,\varrho^{(2)}_t}(\mfU_t) ] \\
 & =b \left( \hat{\varrho}(-2t) \right)^2 e^{bt} \bar{\mfu}_0  + 2a \E[ \Phi^{2,\varrho^{(2)}_t}(\mfU_t) ] \, .
\end{split}\end{equation}
Let $m(n,t) = \Phi^{n,\1(r_{i,j} < 2t\ \forall i<j)} (\mfU_t)$ be the sum of the $n$-th power of the $2t$-families at time $t$.
Claiming that $m(2,0) = 0$ (which holds since $\hat{\varrho}(0)=0$ and $\hat{\varrho}$ is a continuous function by Assumption~\ref{a.1}) we obtain from \eqref{e.tr2} that
\begin{equation}\label{e.tr3}
 m(2,t) = e^{2at} a \bar{\mfu}_0 \int_0^t e^{-as} \left( \hat{\varrho}(-2s)\right)^2 \, \dx s \, .
\end{equation}
Now take a sequence of $\hat{\varrho}$ such that the limit $\hat{\varrho} \nearrow \1_{\R_+}$ holds point-wise. 
The integral on the right hand side of \eqref{e.tr3} converges to $a^{-1}(1-e^{-at})$.
The expectation of the left hand side convergences by monotone convergence.
Thus, we obtain for the expected sum of the squares of the subfamily sizes
\begin{equation}\label{e1366}
 \E [ \Phi^{2,\1(r_{12}< 2t)} (\mfU_t) ] = b \bar{\mfu}_0 \frac{1}{a}\left( e^{2at} - e^{at} \right) \, .
\end{equation}
In fact the second moment of the branching process satisfies \eqref{e.tr2} with $\hat{\varrho}(-2t) = 1$ as well.
\end{proof}

\section{Extensions to the marked case: Proofs of Theorem~\ref{TH.SUPWALK},~\ref{T.ACBRANCH} and their Tools}{\sectionmark{ Extensions to the marked case}}\label{ss.extmc}

We have already in Section~\ref{ss.histbranpro} generalized the basic concepts of the present work to the \emph{marked} case, so that we now work with objects in the framework for our criterion.
Next in Section~\ref{ss.vercrittree} we show how to formulate some further concepts needed for the proofs, as \emph{truncated} polynomials or \emph{smooth} truncation and then make up a list of things still to be proven, in order to be able to apply the criterion to the extensions to the spatial case.
This will conclude the proof of the main results.
This involves to have the tools we used in the non-marked case available in a suitably modified form.

Later in Section~\ref{ss.7.1} we have to show the tools can indeed be generalized to spatial versions as well and provide the necessary extensions of the tools lemma by lemma, proposition by proposition.
Then we can continue along the path of proof taken for the non-spatial model.
Then in Section~\ref{ss.2592} we prove as the \emph{key} element, namely the properties of the martingale problem, which is the spatial version of Proposition~\ref{p.dual.ext} (FK-duality) and Proposition~\ref{prop.1426} (wellposedness).

\subsection{Verification criterion for $\U^E$-valued super random walk: Theorem~\ref{TH.SUPWALK} and~\ref{T.ACBRANCH}}\label{ss.vercrittree}
Now we are ready to check the generator criterion for the branching property of the {\em $\U^{D^\ast}$-marked genealogy-valued super random walk.}
Note that the $ E-$marked case does not pose problems since then the marks are \emph{not} truncated.
The work has to be done to include \emph{historical information} on ancestral path so that we can lift the argument we gave before to the \emph{spatial process} with marks which are paths. 
This \emph{lifting} we address in the sequel.
The key concepts we need are {\em truncated} and {\em smoothly truncated} polynomials.

\paragraph{Truncation}
To study the truncated states we have introduced in Section~\ref{ss.histbranpro} we need truncated polynomials.
This means we first have to introduce the {\em truncated} monomials in the spatial context, recall here \eqref{a.101}, by proceeding analog to  \eqref{e1319} and adding now the indicator on distances as before and another truncation map acting on the {\em mark variables}.
We begin by looking at the time-inhomogeneous situation, before we pass to the time-space process and then to the \emph{adjusted paths}.

Consider now polynomials based on the function $ \chi $ of the mark which evaluates the path at a tuple of time points $ 0 \leq t_1< \cdots t_k < \infty $.
For these polynomials we have to define now truncated polynomials which do not contain information about times before $ t $, for some $ t \geq 0 $.
Recall we have (recall \eqref{1372}, \eqref{1373}) a truncation operation which has two components which separately act on the distances and on the marks.
Therefore we will in the definition of {\em truncated polynomials} use polynomials build on $ \phi $ and $ \chi $ on which we can then let the truncation act \emph{separately} on $ \phi $ and namely of product form.
This means the truncation has the form that the $ t-$truncated polynomials $ \Phi_t^{\phi, \chi} $ given via $ (\phi, \chi) $ is \emph{truncated} by switching to a polynomial given by $ (\phi^{(t)}, \chi^{(t)}) $:
\begin{equation}\label{e1446}
(\phi, \chi) \rightarrow (\phi^{(t)}, \chi^{(t)}), \mbox{  with  } \phi \rightarrow \phi \cdot 1_{\{r_{i,j} \leq \; 2t\}}
\end{equation}
and the transformation of $ \chi $ is as follows.

In the mark function we "truncate" the monomial by replacing the $ \chi $ by taking {\em $ \chi $ applied to the truncated path} i.e. $ \chi^{(t)} (\underline{v}) = \chi(\underline{v}^{(t)}) $ (recall \eqref{1373} and the comment afterwards.
This is again a polynomial, but now with a different $ \chi $, which we call $ \chi^{(t)} $, for a $ t- $truncation.

Summarized this means that for marks we now consider:

\label{e3803}\begin{equation}
\wh \chi^{(t)} (t,v)= \wh \chi(t,\underline{v}^{(t)}), \;  \underline{v}=(v_i)_{i=1, \ldots, n} \; , \quad v_t^{(t)} \mbox{  given by \eqref{1373}   }.
\end{equation}

\begin{remark}\label{r.1454}
This function $ \chi^{(t)} $ is not of the form $ \chi \cdot \varrho $ as with $ \phi $ and the truncation there. 
The reason is that this would not fit with the truncation extended to $\U^V$ according to the reasoning we explained in Remark~\ref{r.1247}.
Note also we have individuals with small distances associated with the tree top which corresponds to the innovation part of the evolution, while for the path this is an  increment.
What we need is to just remove the information on the complementary part. \qed
\end{remark}

We have to {\em lift this truncation} now to the case of $ D^\ast $ marks and deal with the time-shift operation.
Here of course we do this by just shifting the truncated path by $ -t $ to obtain the truncated adjusted path, so that we need no new notation.

\paragraph{Smooth truncation}
The truncation map on path $ \chi \rightarrow \chi^{(t)} $ involves in general a change in value since the values evaluated before time $ t $ are replaced by the ones derived from $ v(t) $ instead of $ v(t_k) $ for $ t_k < t $.
This means if we vary $ t $ we get a change in $ t \rightarrow \chi^{(t)} (v) $ in fact typically a jump at $ t $ if $ t $ is a jump point of the path $ v $ or a limit point of such jump points.
Furthermore recall that these polynomials based on evaluation functionals at time points are \emph{not} continuous on $D(\R,E)$.

We note that the path process generator involves a derivative w.r.t. the explicit time only so that we might think we need not smooth this part as in the case of $\phi$.
\textit{However we use this by using the truncation level $t$ being equal to current time, so} that we need the total differential so that we have for $\wh \chi$ smoothness in both variables, explicit time and location.
Note that for the case $ V=E $ there is \emph{no} need for truncation of the mark, this is different for $ V=D(\R,E) $.

{\em Explanation}
Fix a function $ \wh \chi $, i.e. we fix the sequence of the $ (t_k)_{1,\cdots, m} $ where paths are evaluated.
Note that these are numbers in $ \R $ with $ -\infty<t_1<t_2<\cdots<t_m<+\infty $.
Therefore if $ t $ increases we keep evaluating path which satisfy that we evaluate for $ t \geq 0 $, points which observe for $ t_k $ with $ t_k<t $ only values which after $t$-truncation are all equal to $ v(t) $ for the truncated object.
However before we might have had other positions.
Therefore all positions change to the position at $ t $, hence we get here a jump at the truncation time if the path at $ t $ is different from the position at the $t_k < t$ and this change jumps with a varying $t$ if there are jumps at time $t$ or close to it.
This jump we have to dampen in a differentiable way in time, if $ \underline{v} $ is such that such a jump occurs, i.e. if the $ v_k $ are not such that they are constant for $ s \leq t $.

In this spirit above we have to define the \emph{smooth truncation} operator, i.e. we have to extend now the map $ \varrho: \R \times \D_n \longrightarrow [0,1] $ to a map $ \varrho : \R \times (\D_n \times V^n) \rightarrow [0,1] \times V_t$  of the form $ (s, \underline{\underline{r}}, \underline{v}) \longrightarrow (s, \underline{\underline{r}}\;  \varrho^1_t,      \varrho_t^2(\underline{v}^{(t)})) $ and then formulate the analogues of the Assumptions~\ref{a.1},~\ref{a.2}.
How to choose $\varrho^2_t$?

We note that $ \varrho^2_t $ must then provide a sliding window of the ancestral path with a window defined w.r.t. the current time.
As function of the distances we want to preserve the properties in the non spatial case and {\em add} properties as a function of the mark.

\paragraph{(1)}
Return to the martingale problem and the needed modification of Assumption~\ref{a.1}. Here we have removed every information on the ancestral path time $ h $ or more back and this cutoff we now approach smoothly from the top.
What smoothness properties do we have to impose now?

We adapt the smoothness property to the special form we have for the test functions of the mark.
The function is in $ C^1 $ as function of the {\em explicit time variable}.
As a function of the \emph{marks} we have to dampen the jump explained above.
We therefore have to interpolate here between the two values.
We consider the {\em convex combination} of $ \chi^{(s)} $ with $ s \geq t $.
We generate this convex combination by taking a function $ \wt \varrho_t^2 (s), s \geq t $ which takes values in $ [0,1] $ is in $ C^1_b $, monotone decreasing with $ \intl_t^\infty \wt \varrho_t^2 (s) ds=1 $.
Therefore we now have that $ \chi^{(t)} $ is a function of $ (\underline{v}^{(s)})_{s \geq t} $ which we denote by $ \varrho^2_t(\chi) $ so that now
\begin{equation}\label{e1487}
(\phi, \chi) \rightarrow (\phi \varrho^1_t, \varrho^2_t (\chi)).
\end{equation}
This is easily lifted to the time-space process replacing $ \chi $ by $ \wh \chi $, with $ \wh \chi(s,\cdot)= \psi(s) \chi(\cdot) $.

We have to pass now to {\em functions of the adjusted paths}.
We achieve this by just considering for current time $ T $, the $ (-T)$-shifted function $ \wt \varrho_t^2(\cdot) $.
This now gives us the modification needed in Assumption~\ref{a.1}.

\paragraph{(2)}
In order to get the marked version of Assumption~\ref{a.2} we proceed as follows.
If we want to $ (t+c)$-truncate instead of $ t$-truncation, hence the basic function $ \wt \varrho_t(\cdot) $ is shifted by $ -c $, i.e.
\begin{equation}\label{e2042}
\varrho_t(s)=\varrho_{t-s}(0)=\varrho_0(t-s).
\end{equation}

\paragraph{Test-functions for martingale problems}
Consider the \emph{smoothly truncated} test function on $ \U(t)^\sqcup $ of the forms:

\begin{equation}\label{e1372}
h_t(\mfu)=\mbox{  exp  } (-\Phi^{n,\phi\varrho^1_t, \varrho^2_t (\wh \chi)}), \mbox{  with  } \phi,\chi \geq 0,
\end{equation}

\begin{equation}\label{e1371}
F(t,\mfu)= \Psi (t) \mbox{ exp } (-\Phi^{n,\phi\varrho^1_t,  \varrho^2_t(\wh \chi)}).
\end{equation}

We have to calculate the generator now for a function for $ n \in \N, \phi, \chi \geq0 $ and $ \phi \in C_b^1(\R^{n \choose 2}, \R) $ and $ \wh \chi \in b\CB(\R \times V, \R) $ of the form as constructed in \eqref{e1322}, which are built from functions $ \chi^k_i \in b\CB(E,\R) $ and functions $ \Psi_k \in C^1_b(\R,\R) $, furthermore $ \Psi \in C_b^1(\R,\R) $ and $ \varrho_t $ satisfies for all $ t \geq0 $ respectively $ \varrho_\cdot(\cdot) $ the path-marked version of the Assumption~\ref{a.1} and~\ref{a.2}. 
This will be given in Section~\ref{ss.7.1}.

We can now define $ \D_{\rm lin} $ for the spatial case as the set of truncated polynomials again, but now we have to include also the spatial test function in the specification (compare \eqref{a.101}).

The first task is to verify that we can choose the set $ \wt D $ as the set of all functions of the form above, and as the set $ D_t $ the set of $ t- $multiplicative functions on $ S_t $ the functions in \eqref{e1371}.

This has to be lifted to $ D^\ast $ now.
We see that we can lift the expression above easily by replacing $\wh \chi  $ respectively $ \varrho^2_t(\wh \chi) $ by their associated function of the adjusted path, i.e. $ (\wh \chi)^\ast $ resp. $ (\varrho^2_t (\wh \chi))^\ast $ shifted by $ -T $ if $ T $ is the current time to obtain finally the truncated function:
\begin{equation}\label{e1600}
F^\ast.
\end{equation}
\paragraph{Remaining tasks}
As we pointed out in order to lift the argument in Section~\ref{ss.T:BRANCHING},~\ref{s1501141604b}  to the spatial case we have given the recipe how to lift test functions which are truncated.
What remains to be done?
The function $ F^\ast $ has to be shown to be {\em $ t-$ multiplicative} and then we need to verify the properties w.r.t. to the \emph{wellposedness} of the martingale problem.
We formulate these two points and the other needed extensions next to prepare the conclusion of the proof as given earlier with the concepts introduced in this  Section~\ref{ss.vercrittree}, which amounts to carrying out the arguments for the claims we used above to verify the criterion in the sequel.

{\bf (i)} This is again the $ t-$additivity of the {\em truncated} polynomials, which also holds including now polynomials as function of the mark, which is {\em truncated} as well, but it will turn out that the point is that we still never sample points from different $ 2t-$balls which contribute since we work with $\phi \cdot \chi$.

{\bf (ii)} We need that with $ \wt D $ as test function we get a solution of the $ (\Omega^{\uparrow,\rm anc}, \wt D)$-martingale problem and that we have uniquely determined laws of the $ t-$tops of the solutions.

For \textbf{(ii)} the ideas are as follows.
For the first point we need the spatial version of Lemma~\ref{l.D_lin:tildeD} in combination with a version of Lemma~\ref{L.EX}.
The uniqueness property follows again if we establish the marked version of Proposition~\ref{p.dual.ext}. 
This means we have to define the {\em dual process} first.

Here we replace the coalescent by the spatial coalescent which is a $ E- $marked partition valued process equipped with a {\em distance matrix} and a \emph{vector} recording the {\em path of the "individuals"}, which are the initial partition elements containing one element.
In the Feynman-Kac term we have the joint occupation time of partition elements sitting in the same sites of $ E $.

The final point is to show the $ t-$ multiplicativity of $ h_t $, here we need the spatial version of Theorem (\ref{Inf-p.trunc.poly}) in \cite{infdiv}.
Finally we have to {\em calculate $ g $ and show its $ t- $additive}, i.e. the spatial analog of Proposition~\ref{l.tr5}.

\emph{All these extensions to the spatial case of statements given in Section~\ref{ss.theoftheo} to the non-spatial genealogy-valued Feller diffusion will be given in Section~\ref{ss.7.1} and~\ref{ss.2592}.}
Once we have the generalized statements the argument given above is closed.

\subsection{Verifying the key Lemmata, Propositions in their spatial version}\label{ss.7.1}
We recall Section~\ref{ss.vercrittree} where we saw how to apply our criterion requires to verify the spatial versions of the Lemmata and Propositions of Section~\ref{ss.theoftheo}.
Start with some preparations needed involving some important  observations concerning the form of the action of the generator on smoothly truncated polynomials, which is basic.
Then it remains to show why the spatial versions of the key Lemmata and Propositions hold, and then we have to explain why the flow of arguments given in the subsections above can be modified appropriately.
We also need to show the ancestral path-marked process exists and is unique characterized by a \emph{well-posed martingale problem} both which we defer to the next subsection.

\paragraph{Generator action on truncated polynomials}\hfil\\
The key point was to operate with the generator action on \emph{truncated} polynomials and we had to develop a \emph{smoothed} version of this operation to be able to apply stochastic calculus.
We consider first the time-inhomogeneous version, i.e. marks $ V=D^+ $, and then pass to $V= D^\ast $.

\textbf{(i)} 
If we consider a polynomial it observes the state on a geographically finite window.
Therefore the measure restricted to this set is a finite measure and two of the operators, the operators $ \Omega^{\uparrow, \rm grow}, \Omega^{\uparrow, bran}_\xi $ act as before on a polynomial and also for $ exp(-\Phi^{m,\phi,\chi}) $ we have the same expressions we dealt with in the previous sections and therefore we can work with the smoothed truncation as before after carrying out the lifting of the functions to ones on $\U^V$.
The only point is to control the total mass in the window of observation depending of course on the potentially {\em infinite} mass on all of $ E $ (on finite $ E $ or with finite initial mass nothing changes).
Therefore we focus below mainly on the {\em new} operator $ \Omega^{\uparrow, \rm mig} $ or $ \Omega^{\uparrow, \rm anc} $ in particular if they are acting on \emph{truncated} polynomials and on the {\em behaviour of the total masses} in the spatial finite window of observation.

This latter point however is a feature of the total mass process which in itself is a Markov process, namely the super random walk for which these issues are by now well understood and we refer here to the needed facts in the literature, see \cite{LS81}, \cite{GLW05} for details of such arguments which we don't copy here.
Hence it remains the first issue.

The generator is acting on truncated polynomials as follows.
Here the reader should recall the actions in the case without truncation, the point 2 below \eqref{e1181}.
The new element in the generator calculations compared to Section~\ref{ss:pr:dual} is now the fact that we want to work with {\em truncated} polynomials, \emph{truncated in addition in the mark} rather than the truncated polynomials we had before.
Since we have smoothed the truncation via the $ \varrho_t $ we can again use calculus and obtain now as additional terms the ones which arise from the operator $ \frac{\partial}{\partial t} $ on the level of the explicit time coordinate with which the $ \U^V- $valued process was augmented as well as $ \frac{\partial}{\partial s} $ which is part of the generator of the path process which was also augmented with an explicit time coordinate.
Now we get:
\begin{align}\label{e2005}
\Omega^{\ast, anc} \Phi^{n,\phi \varrho^1_t,\varrho^2_t (\wh \chi)}(t,\mfu)&= \left[ \suml_{\xi \in E} \; \;  \suml_{1 \leq k < l< n} \; \frac{b}{\overline{\mfu}_\xi} \; 1_{\{v_k(t)=v_\ell(t)\}} \; (\underline{v}) \; \Phi^{n,(\phi \varrho^1_t, \varrho^2_t(\wh \chi)) o \theta_{k,l}} \; (t,\mfu)\right]\\
&+ \Phi^{n,2(\overline{\nabla}_1\phi)\varrho^1_t,\varrho^2_t(\wh \chi)} \; (t,\mfu) \nonumber \\
&+ \Phi^{n,\phi(\frac{\partial}{\partial t} \varrho^1_t),\varrho^2_t(\wh \chi)} \; (t,\mfu) + \Omega^{\uparrow,\rm anc} \; \Phi^{n,\phi \varrho^1_t,\varrho^2_t(\wh \chi)}(t,\mfu), \nonumber 
\end{align}
where $\overline{\nabla}_1 $ is w.r.t. to the distance variables.
It remains to write out the \emph{last term} explicitly.

Since we have now the $ \varrho_t^2$ smooth truncation this leads to a convex combination of such product functions resulting after the action of $ \Omega^{\uparrow, \rm anc} $  in a convex combination of product functions.
We also have to take into account the contribution to the generator of the time change due to the weight function leading to the operator we call $ B $ and which is again resulting in a polynomial.
We have to write this out now in formulas.

We note that $ \varrho_t^2(\wh \chi) $ is a function which is an integral over functions which are a product of factors of evaluations at \emph{specific times} of the path.
Through $ s-$truncation some of these factors are now evaluated at different times.
This concerns factors to be evaluated at times $ t_k \wedge s $ which are replaced effectively by the path at the truncation level if $ t_k  $ lies before the truncation time.
Nevertheless this is a new function of the same type.
Now the same operator $ A^\ast_k $ acts on this {\em new function} $ \wh \chi^\prime $.
We saw in \eqref{e.1157} how the operator acts on functions of the path.

Similarly the time evolution leads to a change in $ \varrho^2_t (\chi) $ by changing the applied weight function which results in a change of the mixing measure of the truncations of $ \wh \chi $ leading to an integral with weights $ (\wt \varrho^2_t)^\prime $, defining an \emph{operator $ B $}.
Define $ (\varrho^2_t)^\prime $ as the operation where the weight function $ \wt \varrho^2_t $ is replaced by $ \frac{\partial}{\partial t} \wt \varrho^2_t $.
We get then for the last term in \eqref{e2005}:
\begin{equation}\label{e2259}
= \suml_{k=1}^n (\Phi^{n,\phi \varrho^1_t,A^\ast_k(\varrho^2_t(\wh \chi))} + \Phi^{n,\phi \varrho^1_t, B\varrho^2_t \wh \chi})
 \mbox{   (recall (\ref{e.1157}))}.
\end{equation}

\textbf{(ii)} 
   From this expressions above we have to pass now to the one with marks in $ D^\ast $ and identify the action of the operator of the mark evolution.
This action arises from evolving the path on $D_{0,t}$ and then shifting it back by $t$.
   
This means after evolving for time $t$ replacing $ \varrho^2_t(\wh \chi) $ by $ (\varrho^2_t(\wh \chi))^\ast $
(recall here \eqref{e1487},\eqref{e2042}) where we introduce for a function $ \chi $ on $ D(\R,E) $ a map $ \ast: \chi \to \chi^\ast $ where the value of $ \chi^\ast $ on $ D(\R,E)^\ast $ is the number obtained on $ D_{-t,0} (\R,E) $ as value after shifting the path by $ -t $.
    
    To clarify this we need to explain especially the action of $ \Omega^{\uparrow,\rm anc} $ on truncated polynomials.
    Our goal is now to write out the formula for (dealing with the explicit time variable in the path process):
    
    \begin{equation}\label{e1979}
  \Omega^{\uparrow, \rm anc-sh}\left(\Phi^{m,\phi \varrho_t^1,\varrho_t^2 (\wh \chi^\ast)}\right) =
\Omega^{\uparrow, \rm anc} \left(\Phi^{m, \phi \varrho^1_t, \varrho^2_t(\wh \chi)}\right).
    \end{equation}
    
    Now we have to adapt the expressions and act on elements of $ D^\ast$.
    Observe however that the expression in the quoted equation gives the evolution of a function from $ D_{0,t} $ truncated at the present time.
    This action is now shifted back by $ t $ onto elements of $ D_{-t,0} $.
    
    We can write the generator of the mark evolution as \eqref{e1979} plus shift in the form (note in the time-space process $ A^\ast_k $ does not explicitly depend on $ t $ and furthermore the operators $ B $ and $ ^\ast $ commute):
    \begin{align}\label{e1267}
    = \suml_{k=1}^n (\Phi^{n,\phi \varrho^1_t,(A^\ast_k(\varrho^2_t(\wh \chi)))^\ast} + \Phi^{n,\phi \varrho^1_t, (B\varrho^2_t \wh \chi)^\ast}) + \Omega_t^{\rm shift} \left(\Phi^{m, \phi \varrho^1_t, \varrho^2_t(\wh \chi)} \right)
     \mbox{   (recall (\ref{e.1157}, \ref{e1367}))},
    \end{align}
    where $ A^\ast_k $ arises from transitions at the $ k- $th sampled individual including actions of migration, recall that $ \varrho_t $ depends on both distances and marks as written out explicitly.
    Note at this point we see already that the term is a polynomial, however for its properties we need the following explicit formula:
    \begin{equation}\label{e2046}
    (A^\ast_k(\varrho^2_t(\wh \chi)))^\ast = (A^\ast_k \wh \chi)^\ast: = A^{\ast,\ast}_k(\wh \chi), \quad
    B \varrho^2_t(\wh \chi) = (\varrho^2_t)^\prime(\wh \chi).
    \end{equation}
   
    We need now that the rhs. of \eqref{e2046} "is a truncation", i.e. is zero if two sampled points are from different $ t- $balls.
    This is the case since the smoothly truncated case contains in each term the factor which is zero for distances above the truncation level and note that this is not affected by the $ -t$ shift.
    
    Note that the form of the generator implies that acting on polynomials we have a polynomial again with a new function $ \phi,\chi $.
\begin{remark}\label{r.2370}
Recall the approach given in Remarks~\ref{r.1244},~\ref{r.1364}.
Then we would get the following.
The truncated marked monomial is defined via
\begin{equation}\label{e.PhiTruncMarked}
 \Phi_t^{n,\phi,\chi}(\mfu) = \Phi^{n,\phi\cdot c_t, \chi \cdot i_t}(\mfu), \ \mfu \in \U^{D^\ast},
\end{equation}
for $t \geq 0$ and $D^\ast$ as in \eqref{e.Dast}.
Here, $c_t$ as in \eqref{e1319} and $i_t: V^n \to \R$ via
\begin{equation}\label{e2393}
 i_t(v_1, \dotsc, v_n) = \prod_{j=1}^n \1(v_j(s) = v_j(-t)\ \forall s < -t ) .
\end{equation}
It is easy to see that monomials of the form \eqref{e.PhiTruncMarked} are additive on $S_t = \U^{D_t^\ast}$. \qed
\end{remark}

\paragraph{\emph{Spatial version of Lemma~\ref{L.EX}}}\mbox{}\\
Reviewing the proof in the non-spatial case, we see that we have to explain how to modify the function $ w $ from \eqref{e1342c}.
Here the new element is the term arising via the generator of the {\em mark evolution} for which we have to give the argument.
However this term is again a drift term i.e. a {\em first order operator}.
Hence we get an additional term appearing in the expression, which however gives in the expression for $ w $ just another truncated  polynomial of the order $ n $ as we saw above in \eqref{e1267} and \eqref{e2046}.

Next we have to replace the {\em moment calculations} for the Feller diffusion by those for the super random walk to account for the fact that polynomials are not bounded.
We have to obtain bounds on the expectations of polynomials of order $ n $ to conclude they remain finite after finite time and to show that they still determine laws in spatial models.
These calculations are well known in the literature, see \cite{D93}. 
If we want moments \emph{of all orders} for positive times which are law determining we need stronger restrictions then just requiring a finite Liggett-Spitzer norm , namely $ \sup\limits_n \Big((n!)^{-1} \suml_{i \in \Omega} (\overline{\mfu}_i)^n \gamma_i \Big) < \infty $ (recall here \eqref{e1373}), or we have to pass to {\em local} martingale problems on the Liggett-Spitzer space.
Then the argument proceeds as before, we omit the standard details.

Therefore the argument goes through again in the spatial case.

\paragraph{\emph{Spatial version of Lemma~\ref{l.D_lin:tildeD}}}\mbox{}\\
We deal now with test functions $ \psi(t) \exp(-\Phi^{m,\phi \varrho^1_t,(\varrho^2_t(\wh \chi))^\ast}), (\psi \in C_b^1(\R,\R)) $ incorporating now the time derivate operator as well, recall \eqref{e1344} and \eqref{e1344b}.
The only change is that on the rhs. of \eqref{e1345} now the generator of the mark motion appears as well besides $ \Omega^{\uparrow, \rm grow}, \Omega^{\uparrow, \rm bran} $ and $ \frac{\partial}{\partial t} $ which give the contribution exhibited in \eqref{e2005} and \eqref{e1267}, \eqref{e2046}.

The calculation from \eqref{e1344} - \eqref{e.tr10} carries over once we understand the new term namely the operator $ \Omega^{\uparrow, \rm anc} $ since the other terms are truncated polynomials and are $ t- $ additive.

Now the new term, which as a first order operator lowers the order by one, but since we have coefficients which are linear, i.e. first order monomials, we get a polynomial of the same order back again, where coefficients are given via the $ a(\cdot,\cdot) $, see here formula \eqref{e1267} combined with \eqref{e2046} where this term is analyzed in detail.  
Then the argument works as before.

\paragraph{\emph{Spatial version of Lemma~\ref{l.tr2}}}
\begin{lemma}[Polynomials: approximation by truncated polynomials]\label{l.tr2.marked}
 Let $t>0$. Suppose Assumptions~\ref{a.1} and~\ref{a.2} hold for functions $\varrho = \varrho^{(n)} : 
\bbR^{n\choose 2} \to [0,1]$, $n\in \N$. 
Then for any truncated polynomial 
$\Phi_t^{n,\phi} \in \Pi(C^1(\D_n))$, we can find a sequence of polynomials $(\Phi^{n,\phi_N \varrho_t, (\varrho_t^{2,N}\chi)^\ast})_{N \in \N}$ in the 
family of polynomials $\{\Phi^{n,\phi \varrho_t, (\varrho_t^{2,N} \chi)^\ast}: \, \phi \in C_b^1(\bbD_n, \R^+),\, \chi \in C_b(V^\ast,\R^+), n \in \N\}$ such that for 
all $\mfu \in \U$:
 \begin{equation}\label{e1318.marked}
   \Phi^{n, \phi_N \varrho_t, (\varrho_t^{2,N} \chi)^\ast} (\mfu) \nearrow \Phi_t^{n,\phi, \chi}(\mfu)\,, \text{ as } N \to \infty.
 \end{equation} \qed
\end{lemma}
For a proof we have to deal only with the marked part, i.e. the approximation of the function $ \chi $ by its smooth truncations.
This is standard and follows from the right continuity of the path.

\paragraph{\emph{Spatial version of Proposition \ref{l.tr5}}}\mbox{}\\
Recall the formula for the homomorphism $g_f$ of our criterion in the non-spatial case in \eqref{e1320},\eqref{e1317}.
The change we need to provide is the new form of the $g_f$ of the criterion which is now a $g_{\phi,\wh \chi}$ and of which we have to show that it is \textit{$t-$additive}.
We have here (recall that this object here depends on $ (\phi\varrho^1_t,(\varrho^2_t(\wh \chi))^\ast) $), but we suppose this in the formula:
\begin{equation}\label{e1375}
g(t,\mfu)=\Omega^{\uparrow} \Phi^{n,\phi\varrho^1_t,(\varrho^2_t(\wh \chi))^\ast}+ \Omega^{\uparrow, \rm anc} \Phi^{n,\phi \varrho^1_t, (\varrho^2_t(\wh \chi))^\ast}, \mbox{   where   } \Omega^\uparrow = \Omega^{\uparrow, \rm grow} + \Omega^{\uparrow, \rm bran}.
\end{equation}
We have to see here still that the second term on the rhs., which is the new term, is {\em $ t- $additive}.
Additivity follows however, since $ \Omega^{\uparrow, \rm anc} \Phi^{n,\phi\varrho^1_t,(\varrho^2_t(\wh \chi))^\ast} $ is a {\em mixture} of {\em truncated polynomials}.
Namely in detail we show:\\
{\bf {\em (i)}} we have again that truncated polynomials are additive and\\
{\bf {\em (ii)}} we can wright out the second term explicitly to see it is a truncated polynomial.

This has been proved in \cite{infdiv} in Theorem~\ref{Inf-p.trunc.poly} for the case without marks.
The point is that a sample only contributes if all sampled points are in the same open $ 2h- $ball, which \textit{still holds with marks} since the monomial is defined by $ \phi \cdot \chi $.

\subsection{Basic properties of the martingale problem: Proof Proposition~\ref{prop.1426}, spatial version of Proposition~\ref{p.dual.ext} (FK-duality)}\label{ss.2592}

Crucial for our criterion is the characterization of the process via an  \textit{operator specifying a well-posed martingale problem}.
One issue which we treat first is the \emph{uniqueness} for the solution of the martingale problem using duality, then we discuss the \emph{existence} of a solution and prove Proposition~\ref{prop.1426}, which in particular requires proving the spatial version of Proposition~\ref{p.dual.ext} establishing the Feynman-Kac duality.

For the required \emph{uniqueness} we have to start as first point by introducing the \emph{dual process} and proving a \emph{duality relation} in the spatial and path marked context, which requires to pass to a spatial coalescent on $ E $ and to augment it for the path valued case by a vector describing the paths of the sampled individuals taken up to the present backward time.

The second point we have to deal with the question of the \emph{existence} of a solution for the martingale problem which is well known (see \cite{ggr_tvF14}) except for the \emph{path} marked case which we therefore discuss in more detail below.
Existence follows via \emph{approximation} by individual based models as we point out below however it is {\em not} standard for marks from $ D(\R,E) $ and that is why we give more details for this point.

{\bf{\em Step 1: Uniqueness and FK-duality}} We shall now argue that we have again a duality relation and the {\em analog of Proposition~\ref{p.dual.ext}} holds in the marked case.
A first important observation is that since migration and branching occur independently and the generator consists therefore of a sum of three operators growth and branching we had before already and each allowing a duality and in addition the {\em evolution of the mark} for  which we will establish duality below. 
If we now take as test functions polynomials which are based on test functions on $ \D_n $ and $ V^n $ in {\em product form} and we get nice expressions for the generator such that for each of these operators we can establish duality.
For that reason we get again a {\em Feynman-Kac duality relation} for this spatial model.
Here are the details.

We give now first explicitly the \emph{dual process}, then the {\em duality function} for first the $ E- $ marked and then the path-valued case and finally verify that the \emph{duality} holds.
All is based on the \emph{generator criterion} \textit{for duality} again.
Recall the notation and setup for duality in Section~\ref{ss.pop}, where the duality is derived for the case without marks.
We distinguish two cases $V=E$ and $V=D$ or $D^\ast$.

{\em (i) Case $ V=E $.} 

Begin with the {\em dual process} which is based on an \emph{$E$-marked partition valued} process enriched by a \emph{distance matrix}. 
Here the partition elements are each marked by an element of $E$.

We specify the \emph{initial state} as follows.
Start the dual process with $ m $ individuals.
The initial distance matrix between these individuals is $ \underline{\underline{r^\prime}}=\underline{\underline{0}} $.
Furthermore by fixing a set of {\em $ m $ initial locations}, which might be assigned in multiplicity, $ \underline{\xi}=(\xi^1, \cdots, \xi^m) \in \Omega^m $ we determine a function $ \chi $ for the location-marked case, as in (\ref{e1178})-(\ref{e1264}).
Furthermore choose $\phi \in C_b(\D_n,\R)$.

Then the dynamic is as follows, partition elements \emph{migrate} with the kernel $ a $ from \eqref{e1321} and \emph{coalesce} if they share a site.
The elements of the distance matrix grow at rate $2$ as long as the two respective elements are in \textit{different} partition elements.
This specifies the dual process
\begin{equation}\label{e2610}
(\mathfrak{C}_t)_{t \geq 0} \; , \; \mathfrak{C}_t=(p_t,r_t^\prime, \underline{\xi}_t)$ for $ V=E.
\end{equation}

The \emph{duality function} is given as
\begin{align}\label{a2516}
H^{n,\phi,\chi}(\mfu,(p,\uur^\prime,\xi))= \intl_{(U \times V)^n} \phi(\uur^p+\uur^\prime) \chi (\uxi) \mu^{\otimes n} (d(\underline{u,v})), \\
\mbox{  where  } \uur((\underline{u,v}))=(r(u_i,u_j)_{1 \leq i < j \leq n}) \mbox{ and } \uxi(\underline{u,v})=\underline{v}. \nonumber
\end{align}

Finally we need the {\em Feynman-Kac potential} suitably modified in the spatial case.
Here we count the occupation time of partition elements {\em at the same site}.
Denote by $ \# p_{s,\xi} $ the number of partition elements at time $ s $ in $ \xi $.
Then the {\em Feynman-Kac potential} on the time interval $ [s,t] $ is 
\begin{equation}\label{e2156}
\beta_{s,t}= \intl_s^t \; \suml_{\xi \in E} ((^{\# p_{u,\xi}}_2 ) ) du. 
\end{equation}

Then we have the following \emph{FK-duality relation}:
\begin{align}\label{a2522}
E_{\mfu_0} \left[ H^{n,\phi,\chi} \left(\mfU_t, (p,\uur^\prime, \uxi) \right) \right] = 
E_{\mfk_0} \left[ H^{n,\phi,\chi} \left( u_0,(p_t,r^\prime_t,\uxi_t) \right) \exp (\beta_{0,t}) \right]
\end{align}

To \emph{determine the f.d.d. via the dual}, we get an expression based on the \emph{time-space dual} giving the f.d.d. formula of the $ E$-marked process based on the Markov property following \cite{FG94},\cite{GSW} page 13 which we formulate next.

We consider now the \emph{time-space process} of $ \U^E$-valued super random walk and consider the so called \emph{time-space coalescent} on $ E $ to derive the following duality formula determining the f.d.d. of the original process.

Consider time points $ 0 \leq t_1 < t_2 < \cdots < t_\ell=t \quad, \ell \in \N $.
Then consider the following functional of the forward process:
\begin{equation}\label{e2438}
\wt \Phi \left((\mfU_t)_{t \geq 0} \right) = \prod^\ell_{k=1} \; \Phi_k (\mfU_{t_k})
\end{equation}
where $ \Phi_k \in \prod \quad, k=1,\cdots,\ell $.

Next consider the spatial coalescent with \textit{frozen} partition elements, $ \left( \mfK_s \right)_{s \in [0,t]} $.
Here we start with partition elements in the time-space points $ \left( (t_k,i_k^j)_{j=1,\cdots, \ell_k} \right)_{k=1,\cdots, \ell} $ and the partition elements in $ i^j_k,j=1,\cdots, \ell_k $  will be \emph{frozen} for the dual evolution till times $ t-t_k $, for $ k=1,\cdots,\ell $.
(Precisely: the unfrozen, called \emph{active} partition elements, evolve as before as spatial coalescent.
Similarly the Feynman-Kac potential at time $ s $ of the backward evolution does include only the \emph{active}  partition elements at that time, i.e. $\# p_{s,\xi}$ is the number of \emph{active partition elements}.
Similarly the distances between frozen particles is zero and between a frozen and active one is initially zero and grow at speed $1$ till they are both active and the growth is $2$.)

Define now
\begin{align}\label{2450}
\phi(\underline{\underline{r}}) & = \prod^\ell_{k=1} \; \phi_k \left((r_{i,j})_{\ell_1 + \cdots + \ell_{k-1}<i<j\leq \ell_1 + \cdots + \ell_k} \right), & \\
\chi \left(\underline{t},\underline{v}^\pi \right) & = \prod^\ell_{k=1} \; g_k(t_k,\underline{v}^{\pi_k}),& \nonumber
\end{align}
which characterizes the polynomial $ \wt \Phi $ in \eqref{e2438}.
We see that $\wt \Phi(\mfU_t)=H(\mfU_t,\mfK_0)$ with $H$ based on $K_0$ and $\wt \Phi$ as before for the classical objects.

Now following the argument for the Fleming-Viot case in \cite{FG94}, \cite{GSW} we obtain the formula relating the  \emph{time-space forward} and the \emph{time-space dual} process namely we get with expectation on the l.h.s. with respect to $\mfU$ on the rhs. will respect to $\mfK$:

{\em Duality formula for f.d.d.}
\begin{equation}\label{e2463}
E \left[\wt \Phi \left((\mfU_t)_{t \in [0,T]} \right) \right]
=E \left[H\left((\mfU_s)_{s \leq T}, \mfK_0 \right)\right]
=E \left[H(\mfU_0,\mfK_T)\exp(\beta_{0,T})\right]. 
\end{equation}
Since we have here a FK-duality we have to argue here that the Feynman-Kac term behaves as claimed.
However the argument works with the Markov property and applying the duality relation to the time pieces between the $(t_k)_{k=1,\cdots,\ell}$ so that this can be imitated with the FK-duality the result follows from the fact that the Feynman-Kac term is an additive functional.

{\em (ii) Case $ V=D ([0,\infty),E) $ and $V=D^\ast$}.\\
For the path-valued case we have the {\em path marked spatial} respectively \emph{time-space coalescent} which is the following modification of the above. 
Only the \emph{mark and their evolution is different} compared to the dual process in case (i) above.
We let $ t $ be the time horizon for the duality relation.
This means the original process started at time $ s < t $ will evolve till time $ t $.
The dual process will evolve from time $ t $ backwards till time $ s $, its time of evolution runs therefore for time $ t-s $.

We focus on the time-space coalescent including the "simple" one. 
We enrich the coalescent analog to the $V=E$ case.
In addition to the \emph{distance matrix} $ \underline{\underline{r^\prime}} $ we record now the {\em vector of the paths of locations} $ \underline{\xi} = \Big((\xi^1(u))_{u \in \R}, \cdots, (\xi^m(u))_{u \in \R}\Big) $ of all initial individuals recall here the description around \eqref{e2438}, with $ m= \suml_1^\ell  \ell_k$, which also enters the duality relation the same way as in the forward evolution.
The coalescent is denoted $\wt \mfC$ and the time-space coalescent is denoted $ (\wt \mfK_s)_{s \geq 0} $.

Note that here we keep the path of descent of every of the initial individuals, even though they may be piecewise joint path beyond some backward time.
Note that the dual path evolves backwards from $ t $ to $ s $, rather than forward from $ s $ to $ t $.
On the other hand the input in the duality function is, as we shall see, the same over the full interval $ [s,t] $.

Next we come to the {\em duality function}.
To write down the duality function we need the sampling measure restricted to the population at a site, called $ \mu_i $ and given by 
\begin{equation}\label{e2129}
\mu_i(A)=\nu(A \times \{i\}), A \in \mcB(U).
\end{equation}

Now we can define for $ t>s \geq 0 $ the time-space {\em duality function} for that path marked case for the situation where we are starting in $s$ in constant path and distance matrix $r$.

For the duality function we chose again a {\em number of individuals $ m $} in the dual process, the vector $ \underline{\xi} $ of their initial path determined by an $ m$-tuple of locations $ \underline{\xi} $, a function $ \Psi \in C_b(\R^2, \R) $ with $ \Psi(t,s)= \wt \Psi(t-s)$ with $\wt \Psi \in C_b(\R,\R), \wh \chi \in \CB(\R \times D(\R,E^m), \R) $ and then {\em a function} $ \phi \in C_b(\D_n,\R) $.

Then set for $ \mfu \in \U $ and $ s,t \in \R $ with $s<t$ and $\wt \mfC_t=(p,\uur^\prime,\uxi^\prime,t)$:
\begin{align}\label{e1266}
 H^{\phi, \psi, \underline{\xi}, \wh \chi} \; \Big((\mfu,s), (p,\underline{\underline{r}}^\prime, \underline{\xi}, t)\Big) &= &\\
 \psi (s,t) \intl_{(U \times V)^n} d \left(\left( \bigotimes\limits_{i=1}^n \mu_{\xi_i(s)}\right) (d \underline{x})\right) \phi \Big[\left(r^p(x_i,x_j)\right)_{1<i<j \leq n} + & \left(r^\prime(i,j) \right)_{1 \leq i < j \leq n}\Big] \wh \chi \left(s,(\xi(u))_{s \leq  u \leq t}\right).  & \nonumber
\end{align}

This amount to having a $\chi$-function with \emph{two}  factors one as in \eqref{a2516} generating the locations where we sample with the sampling measures the individuals from the population and a second factor to explore the corresponding path at different time points. 

Then the {\em Feynman-Kac duality relation} reads (with $ s = 0 $):
\begin{lemma}[FK-duality:path process]\label{l.2401}
We start the process $\mfU$ in a state with constant path. Then:
\begin{eqnarray}\label{e2166}
E [H^{\phi,\psi, \underline{\xi},\wh \chi} ((\mfU_t,t),(p_0,r^\prime_0, \underline{\xi}_0,0))] 
= E[H^{\phi,\psi, \underline{\xi},\wh \chi  } ((\mfU_0,0),(p_t,r^\prime_t,\underline{\xi}_t,t)) \exp(\beta_{0,t})]. 
\end{eqnarray} \qed
\end{lemma}
However in order to obtain above duality later via the generator relation some additional concepts are needed, due to the time \emph{in}homogeneity.
These operators corresponding to the processes above involve the explicit time coordinate and therefore we will need a generator relation for all times $s$ between $0$ and the time horizon $t$ and hence we have as state of the forward process path which already evolved for some time $u$ and are not in the constant state, which we assumed writing down $H(\cdot,\cdot)$ in a specific way.
In fact we have to be more careful writing down the duality function, so that we can use it for the intermediate times $s$ between $0$ and the time horizon $t$.
In fact we just saw that we have to generalize this a bit now.

Consider the following objects.
For two path $\xi^\uparrow$ and $\xi^\downarrow$ one from $D_{s,u}$ one from $D_{u,t}$ we introduce, for those ones with $\xi^\uparrow(u) = \xi^\downarrow (u)$ the \emph{glued path} $\xi ^{\ast,u}=\xi ^\uparrow \vdash \xi^\downarrow$ from $D_{s,t}$ arising by setting
\begin{equation}\label{e2696}
\xi^{\ast,u} (r) = \left\{ \begin{array}{lcr} 
\xi^\uparrow & \mbox{  for  } & r \leq u \\
\xi^\downarrow & \mbox{  for  } & r > u.
\end{array}  \right.
\end{equation}

This is used to generalize the duality as follows. Now the duality function between time $s$ and $t$ given by replacing the path $\xi$ in the formula \eqref{e1266} by $\xi^{\ast,s}$ from above.

We need above relation also for the \emph{$ t-$truncated process}, including the smoothed versions.
For this purpose the duality function has to be changed by replacing in \eqref{e1266} the function $ \phi $ and $ \wh \chi $ as follows.

For the truncated case we replace $ \phi $ by the truncated $ \phi $ which is $ \phi 1_{[0,t]^n} $ respectively its smooth version.
The function $ \wh \chi $ is replaced by the function of the truncated path respectively the smoothly truncated one.
We have to see to it below that the duality still holds on these truncated test functions by approximation (spatial version of Lemma~\ref{l.tr2} see \eqref{e1318.marked}).

\begin{proof}[Proof of Lemma~\ref{l.2401}]\label{pr.2534}
We have established the needed generator relation for the FK-duality for the non-spatial and time-homogeneous case. 
We note that passing to the time-space process is fine since $ \frac{\partial}{\partial t} $ and $ - \frac{\partial}{\partial t} $ are \textit{operators in duality} with our condition on $ \psi $.

Consider next the $E$-marked case. 
Then the argument for growth and branching operator is easily lifted from $U$ to $U \times V$.
Therefore it remains to verify the duality criterion now only for the \emph{migration operator}.
Hence we have to show that the generator of the mark evolution is in duality with the one for the dual process, the spatial coalescent enriched with the distance matrix and the mark vector which is a vector of locations. 
But this is in the literature.
For details for the mark part of the operator note that this is the same as in the case of the Fleming-Viot process, see therefore Remark 1.19 in \cite{GSW}.

Similarly in order to now establish the Feynman-Kac duality for the \emph{path marked} case we have to show that the mass flow on path is dual to the migration in the path in the $a(\cdot,\cdot)$-spatial coalescent.
Namely in order to check the duality in that case one uses the generator criterion and as test functions polynomials of the form as given in \eqref{e1266} with the mentioned restriction.

Here one might wonder whether this is not just the duality of the time-space processes from above.
That is unfortunately a bit more tricky as we explain next.

\begin{remark}[Time-space process for $E$-marked versus $D$-marked process]\label{r.2630}
If we observe the state at time $ t $ of the \emph{path valued}  process $\mfU$ and evaluate the specified polynomial on the one hand and compare it with  the functional of the path of the \emph{$ E$-marked} process with the specified test-functions are similar if in the latter we consider the time-space process but there is the following difference.

We must observe that in the \emph{$ E$-marked time-space process} we sample from the population at times $ t_1,\ldots,t_m,t $ and observe the position at this time say $ t_k $ only, while for the path marked case such positions appear also for individuals sampled at later time  $ t_i $. 
In particular in the path marked case we sample from the population at time $ t_i $, so that we prune individuals at those time $ t_k $ which do not have descendants at time $ t_i $.
Note however that the pruning is independent of path and genealogy \emph{up to time $ t_i $}.
This means that the ones sampled at a time $ t_i $ but whose path is evaluated at positions $ t_k < t_i $ must in the time-space coalescent in the backward picture be activated at time $ t_i $ and not $ t_k $ that is in the backward time at $ t-t_i $ and not $ t-t_k $.
Similarly the Feynman-Kac term changes. \end{remark}
\end{proof}

Next we have to argue how to get the duality relation respectively the analogue of \eqref{e2463} for the \emph{path-marked} process, first in the time-inhomogeneous setting in (i) then for $D^\ast$ in (ii).

\paragraph*{(i)}
Denote coalescent and functional by $\wt \mfK $ respectively $ \wt \beta_{0,t} $.
We have to address two points (1) How is duality transfered from the mass flow of locations to the mass flow of the {\em path process} and (2) how this then is transfered to the {\em measure-valued historical process} i.e. the process of measures on these path and even further from here to $ \U^V$-valued processes. 
How to prove the claim?

To get the duality we have to prove essentially first ((1) and the first part of (2)) that the mass flow induced by the historical process and a system of independent random walk \emph{path} processes are dual w.r.t. $H(\cdot,\cdot)$ from \eqref{e1266} using the generator criterion.
Some care is needed here due to the time inhomogeneity as we will see below.

To verify the generator criterion for duality, we note that the generator action of $\Omega^{\uparrow,\rm anc}_s$ (recall this is exclusive the action in the explicit time coordinate) is on the polynomial via the function $\chi$.
Here we have to recall the operators defined in \eqref{e1326} and \eqref{e.1157}.
Namely for a polynomial of degree $n,\Phi^{n,\phi,\wh \chi}$ we have:
\begin{equation}\label{e2611}
\Omega^{\uparrow,\rm anc}_s \Phi^{n,\phi,\wh \chi} = \suml^n_{k=1}\Phi^{n,\phi,\wh \chi^{\ast,k}},
\end{equation}
where (recall here \eqref{e.1157})
\begin{equation}\label{e2615}
\wh \chi^{\ast,k} := A^\ast_k \; \wh \chi.
\end{equation}
Note here that the rhs. involves $s$.
This expression we have to compare with the action of the \emph{dual process generator} on a polynomial. 

We now need the action of the mark evolution of the dual process on the duality function for fixed first argument and how this acts on the test function.
To see this we have to calculate the action of the generator of the  random walk path process of one moving individual on $H^{n,\phi,\chi}(\mfu;\cdot)$.
This latter process is a pure jump process and the generator action we look for is the \emph{sum} of the $n$ one--individual generators which is given in equation \eqref{e1326}.
We have to argue now what this formula implies if we apply it to the function given via the duality function above.
For that purpose we have to view this function as a functional of the function $\wh \chi$ of the path on which we can act directly with the generator.
We claim this object we can write again as a polynomial with a \emph{new} $\wh \chi$-function.

To see this we need to analyze what type of function $H(\mfu,\cdot)$ actually is.
This function is an integral over parameters, namely $u \in U$ which is derived from $\mfu$, of polynomials of the form (recall \eqref{e1335}):
\begin{equation}\label{e2675}
\intl_{V^n} \wh \chi \; d \nu_\ast^{\otimes n}.
\end{equation}

In particular does it suffice to have the self-duality in terms of the operation acting on $\wh \chi$, on the function of one path.
Here the dual random walk path goes in the revers direction as the underlying motion hence we have the same $a(\cdot,\cdot)$ as in the mass flow term in the operator and we have here therefore the same coefficients for the action on $\wh \chi$.

Hence the dual migration operator acts only on the function $\wh \chi$ and through it on the polynomial $\Phi^{n, \phi,\wh \chi}$ and the operator $\Omega^{\downarrow,\rm anc}$ does this acting on the polynomial via
\begin{equation}\label{e2621}
\Omega^{\downarrow,\rm anc} \Phi^{n,\phi,\wh \chi} = \suml^n_{k=1} \Phi^{n,\phi,\wh \chi^{\ast,\downarrow}_k},
\end{equation}
where
\begin{equation}\label{e2625}
\wh \chi^{\ast, \downarrow}_k:= A^\ast_k \; \wh \chi.
\end{equation}
We see that $\wh \chi^\ast_k = \wh \chi^{\ast, \downarrow}_k$ which proves the duality of the mass flow.

This proves the Feynman-Kac duality relation for the path valued process.
In particular we get also a formula for the finite dimensional distributions i.e. using instead of the mechanism and the functional $ \beta_{0,t}(\cdot) $ the one of the dual process $ \wt \mfK $ and $ \wt \beta_{0,t} $ as specified above then the analog of \eqref{e2463} holds replacing $ \mfK,\beta $ by $ \wt \mfK,\wt \beta $ in the formulas .q.e.d.

Therefore the uniqueness result for the martingale problem carries over to the $ D^+$-marked case.

\paragraph*{(ii)}
In the final step we have to {\em lift} this FK-duality now to the marks given by the \emph{adjusted} path, shifted by $ t $ to the left for the time $ t $ duality i.e. we have to pass to $ D^\ast$-marks.

Here we apply the shift by $ t $ to the path of the dual process (which runs backward and is shifted by $t$ further)  
i.e. in the duality relation we have to take for a dual path  $ \zeta $ running backwards from time $t$ the adjusted path (with a fixed choice of $ t $):
\begin{equation}\label{e2181}
\zeta (\cdot) \longrightarrow \zeta ( \cdot + t) := (\zeta)^{\ast, \downarrow} (\cdot)
\end{equation}
in the duality relation.
We need that this does not change the two sides of the duality relation for the \emph{chosen} function.
This is true since we shift by $ t $ the path on both sides.

Altogether we have now verified the assumption needed to get the uniqueness of the martingale problem, which is needed in the verification of our criterion for the branching property.

{\bf {\em Step 2: Existence}} Next we need {\em existence} of a solution.
\label{ss1791}In order to prove the existence of a solution to the martingale problems we proceed in principle in two parts.
One considers in the first part a solution to the martingale problem on a {\em finite} geographic space with the "induced" dynamic from the larger space and obtains then in the second part a solution for the possibly infinite space by approximation via a sequence of processes constructed for embedded \emph{finite}  geographic spaces namely for  $ E_n \uparrow E, \mbox{  with  } \mid E_n \mid <\infty $.
Here we use a suitably chosen \emph{modified migration part} of the dynamic.
This approximation argument is standard for measure-valued processes and carries over to the $ \U^V$-valued setup making use of the duality for the limit evolution for the latter, see \cite{ggr_tvF14}, where all details are spelled out. 
Hence it remains to show here the existence of a solution on {\em finite} geographic space.

Here, on the finite geographic space (and \emph{finite total population size}), one uses the \emph{approximation} with an {\em individual based model}. 
One takes per site a finite number of individuals and then lets this  number tend to infinity, see \cite{GPW09}, \cite{GSW} where this is carried out for the Fleming-Viot model and \cite{Gl12} where the branching model is treated in the non-spatial case, the spatial case in \cite{ggr_tvF14}. 
We refrain from giving details in this paper {\em except for the path-marked situation} where a \emph{new} point arises.

Namely we want to be able to approximate our process by a sequence of individual based branching models.
In order to then show that limits of the approximating models solve the martingale problem it is most convenient if we have {\em continuous} test functions.
This is here a problem with {\em path-marked} genealogies, the spatial component in the polynomials is in that model {\em not continuous} 
for these polynomials we use as test functions in the martingale problem, namely in the path-valued case our polynomials are based on evaluations in \emph{fixed time points} and hence are {\em not continuous} in the mark variable in the Skorohod space of path.

We can obtain continuous test functions by considering moving averages of the functions $ \chi $ which we are using, more precisely choosing the time points there at random on $ \R $, recall~\ref{r.1244}.
(This we also use for the topology to define convergence.)
Because of the linearity of the operator $ \wh A $ we can then obtain $ \Omega^{\uparrow, \rm anc} $ for these test functions as well, but now the test functions are continuous.
Furthermore solutions to the martingale problem are also solutions to the one with these test functions, but {\em not immediately the other way around}.
We must therefore directly work with the given martingale problem.

We can use two known properties here.
Namely for the convergence of the projection of the sampling measure on the marks, which yields the historical process, we can use the existence and uniqueness results of Section 7 in \cite{DP91}. 
On the other hand projecting on $\U$ by projecting $U \times V$ on $ U $ gives for finite space the $\U$-valued Feller diffusion, where we can use \cite{ggr_tvF14},  \cite{Gl12}.
These two facts are given in detail below and can be used below to establish the convergence of the genealogies of the subpopulations descending from ancestors in some finite subset of geographic space and then it is well known to lift this to the full population as explained above.

The \emph{tightness} of the approximating individual based processes follows, since tightness of genealogies and marks separately implies the \emph{tightness} of the \emph{joint law}, Theorem 3 in \cite{DGP11}.
However we gave the reference for these two points above.
Hence we need to prove only \emph{uniqueness} of the \emph{limit points} to prove {\em convergence}.
We know our martingale problem has a unique solution by duality.
Hence we get convergence if we show that the limit points of our tight sequence \emph{must solve the martingale problem}.

The convergence to solutions of the martingale problems which are well-posed are known both for the projection on $ (\nu^\ast_t)_{t \geq 0} $ and for the projection on the genealogy.
Hence it only remains to deal with the {\em coupling} of the two components i.e. the {\em joint distribution} of genealogy {\em and} marks to conclude convergence as we saw above.

The process has test functions which are polynomials induced by a product of a function $\phi$ of the distances and $\wh \chi$ a function of the marks.
We have by results on the non-spatial case respectively the historical process (\cite{Gl12},\cite{DP91}) the convergence of the compensators to the limit compensator if we put either $\phi$ or $\wh \chi$ equal to a constant.
We also know that the operators associated with the genealogy and the mark evolution sum to the complete operator.
The branching operator acts on the function $\phi$ as well as the growth operator, the mark evolution operator on $\chi$ but the branching operator in the limit dynamic acts also on $\chi$ through the action of duplicating a path i.e. its action on the measure.
The latter is of course what connects the evolution of marks and of genealogies here in this model which is for the path marked case not just affecting mass at the site.
Still this gives rise to technical problems.

We will circumvent this problem and first prove \emph{convergence} to some limit \emph{without} using uniqueness of the martingale problem.
We will see from the proof that this limit point process is a Markov process.
We also use that we have tightness in the space $D([0,\infty],\U^V)$ so that it suffices to show f.d.d.-convergence, to have convergence in the path-space for the process.
Then we will have to show that this process solves the given martingale problem using the way it is \emph{constructed}.
Since we have the distribution of the genealogy part and the mark part converging we obtain the joint distribution by giving a construction of the \emph{conditional distribution} of the \emph{marks given the genealogy}.
This we can do by using the so called trunks of the state (recall Figure1) which we introduce below.

For a given $\mfU$ we denote by $\wt \mfU$ the projection from the marked genealogy on the genealogy, i.e. we consider the map $\tau: \U^V \to \U$ given by $[U \times V,r \otimes r_V,\nu] \rightarrow [U,r,\mu] \mbox{  with  } \mu(\cdot)=\nu(\{\cdot\} \times V) $ and set 
\begin{equation}\label{e3027}
\wt \mfU=\tau(\mfU).
\end{equation}
Then $\nu=\mu \otimes m$ with $m$ a kernel from $\U$ to $V$ describing the mark configuration given the genealogy, the latter means, we condition on the value of this functional $\tau$ to condition on the process of genealogies.
Hence the object we have to focus on is the $\U \times \CM_{\rm fin}(D(\R,E))$-\emph{transition kernel} $m$.

Fix for the state at time $t$ an $h \in [0,t)$.
We consider here the \emph{marked $(t-h)$-trunks} $\lceil \mfu \rceil (t-h)$  of the states $\mfu$ at time $t$ to show that the marginal distribution at times $t$ converge, by showing that the law of the \emph{marked} $(t-h)$-trunks \emph{conditioned} on the \emph{genealogy} which is a law on the mark distribution converges for all $h \in [0,t)$.
We recall next (resp. define here) $(t-h)$-trunks and \emph{marked} $(t-h)$-trunks.

The $(t-h)$-trunk of an element $\mfu$ of $\U$ is an element of $\U(t-h)^\sqcup$ (recall \eqref{eq:theta:a}) which complements the $\lfloor \mfu \rfloor (t-h)$-top of the $\mfu$, see \eqref{e1314} and as $h \uparrow \mfu$ the trunks converge to $\mfu$.
Here is the definition.

For $\mfu \in \U(t)^\sqcup$ \emph{consider} the $(t-h)$-top denoted $\lfloor \mfu \rfloor(t-h)$ and write $\lfloor \mfu \rfloor(t-h)={\mathop{\sqcup}\limits_{i \in I}}^{t-h} \; \mfu_i$, the decomposition into open $(t-h)$-balls $\mfu_i=[U_i,r_i,\mu_i]$ for $i \in I$.
Then the \emph{$(t-h)$-trunk} of $\mfu$ is defined:
\begin{equation}\label{e2892}
\lceil \mfu \rceil (t-h)=[I,r^\ast,\mu^\ast], \mbox{  with  } r^\ast (i,i^\prime)= \inf \left(r(u,v)-2(t-h) \mid u \in U_i,v \in U_{i^\prime} \right)
\end{equation}
and weights
\begin{equation}\label{e2896}
\mu^\ast (\{i\})=\mu_i(U_i).
\end{equation}

Next the \emph{marked} $(t-h)$-trunk. We decompose the $(t-h)$-top of the marked state $\mfu$ in the elements of $\U^D(t-h)$.
Now we proceed as above but the measures $\nu_i$ are now projected onto the set of path which are only non-constant in $[0,h]$, i.e. $v^h(s)=v(s) 1_{[0,h]} (s) + v(h) 1_{[h,\infty)} (s)$.
This gives us an element in $\U^D$.
Furthermore we have again that the marked $(t-h)$-trunk of $\mfu$ namely $\lceil \mfu \rceil (t-h)$ satisfies $\lceil \mfu \rceil (t-h) \to \mfu$ as $h \uparrow t$.

We begin by constructing the announced candidate for the limit dynamic.

\paragraph{Construction} 
We know the process for the genealogy exists (i.e. the state projected on $\U$ from $\U^V$), so we have to \emph{construct} the \emph{conditional law of $\mfU$ given the genealogy}.
We identify first the conditional law of the marks at time $t$ of the \emph{marked} $(t-h)$-trunk \emph{given} the $(t-h)$-trunk.
Note here that conditioning on the $(t-h)$-trunk is here equal to  conditioning on the genealogy up to time $t$ and then passing to the $(t-h)$-trunks, because of the Markov property and the independent increments of the path process.
For $h<t$ the $h$-trunk can be represented by an $\R$-tree with \emph{finitely} many leaves, so that the $(U^h_t,r^h_t)$ is a finite ultrametric space, where the leaves carry also $\R^+$-valued weights namely the mass of the time $t$ descendants (see \cite{infdiv}) so that we get an element of $\U$,denoted $[U^h_t,r^h_t,\mu^h_t]$.
The corresponding $\R$-tree is represented as a tree with finitely many vertices, binary split points and edges and leaves carrying a weight.
The $\R$-tree has a finite number of founding fathers, $F_1,F_2,\cdots,F_K$.
The set of leaves with the distances and weights determine an equivalence class of ultrametric measure spaces, $\lceil \wt \mfU \rceil (t-h)$.
We represent conditional distribution of the marks of the trunk, by constructing a version explicitly for every realization of $ \lceil \wt \mfU \rceil (t-h)$ more generally for every finite element in $\U(t-h)^\sqcup$.
For the $t-h$-trunk of a $\U$-valued Feller diffusion and via the backbone construction carried out in our framework, see \cite{ggr_tvF14} for more details on the diffusion case, while in the particle case one can use results of Chauvin, Rouault and Wakolbinger \cite{CRW1991}.

We now generate the \emph{marked} object as follows.
We have a collection of independent $\overline{a}$-random walks starting in $e \in E$:
\begin{equation}\label{e2873}
\left\lbrace \left(Y_t^{\{e\},k} \right)_{t \geq 0} , e \in E, k \in \N \right\rbrace.
\end{equation}
We consider the \emph{length of the edges} from the founding fathers to the first split point, say $T_1(F_1),\cdots,T_K(F_K)$.
Then we take the random walks
\begin{equation}\label{e2880}
\left(Y_t^{e_1,1} \right)_{t \in [0,T_1(F_1)]} \;, 
\left(Y_t^{e_2,2} \right)_{t \in [0,T_2(F_2)]}, \cdots,
\end{equation}
where $e_1,e_2$ are the positions of the founding fathers in $E$.
How to choose these points. 
They are the value of the mark attained for $t \leq 0$.

Then we continue with the split points replacing the founding fathers, where their position is given as the end-point of the corresponding random walks above.
We continue until we reach the level $h$.
Then we can associate with each path in the $\R$-tree from a founding father to a leave, parametrized via the length, a \emph{path} with values in $E$ generated by the random walk pieces.
Continuing beyond $h$ and before level $0$ as constant path, we have marked $\lceil \wt \mfU \rceil_t(h)$ with a path in $D(\R,E)$ and we can assign the weight which we get from the weight of the leaf.
This way we have constructed for a fixed trunk of a given $\mfu \in \U$ a random marked trunk, i.e. a random element in $\U(t-h)^\sqcup$.
This way we have constructed $m_h(\mfu,\cdot)$ a kernel from $\U \to \mcM_{\rm fin} (D(\R,E))$ based on the $(t-h)$-trunk of $\mfu$, $\mfu$ an element of $\U$.
We have now constructed a realization of the marked $(t-h)$-trunk of the element of $\U^D$ we look for.

From the $(t-h)$-trunks we obtain in the limit $h \uparrow t$, the $\U$-valued state at time $t$  and we want to define the marks, i.e. we need $m(\mfu,\cdot)$ as the limit $h \uparrow t$ of $m_h(\mfu;\cdot)$, which exists because of the concrete construction.

Suppose next we have a \textit{random variable} with values in $\U^D$.
Observe that above construction defines a random measure on path and its \emph{law} defines a transition kernel
\begin{equation}\label{e2917}
M_h(\mfu,dv) \mbox{  from  } \U(h)^\sqcup \mbox{  into  } \mcP \left(\mcM_{\rm fin} (D)\right).
\end{equation}
Hence we specified the law of the marked $(t-h)$-trunk of the stochastic process we look for and this works for every $h < t$ and the arising laws are by construction consistent, i.e. for $h'<h$ we have the law of the $h'$-truncation of the law for $h$.
As the limit $h \uparrow t$ is taken we obtain $M(\mfu,dv)$ from $\U$ into $\mcP(\mcM_{\rm fin} (D))$ the searched for conditional law of the $t$-marginal.

Precisely we have now the information on the marks which we have to insert in our formalism, namely suppose that we have an $\U$-valued process $(\mfU_t)_{t \geq 0}$, then we focus on the measure $\nu$ in $\mfU_t=[U_t,r_t \otimes r_V,\nu]$, we write as $\mu \otimes m$, where $m$ is a kernel from $\U$ to $(V,\mcB(V))$ which is read of from the state $(U,r,\mu)$ and our construction precisely as follows.

We need for given projection on the genealogy for the measure $\nu$ the conditional law given the first component, more precisely a regular version of the conditional distribution of $\nu^\ast$, which we called above $m(\mfu,\cdot)$, which is a finite measure on $(V,\mcB(V))$.
This measure is for us only relevant projected on the trunks that is we keep the path constant beyond $t-h$, equal to its value at this time.
This projection we read of from our construction by taking the realization of our random walks, leading from the founders to a leaf in the trunk and the weight prescribed there from the condition is their weight and gives $m_h(\mfu,\cdot)$, as an atomic measure on $D(\R,E)$ and as $h \uparrow t$ we obtain $m(\mfu,\cdot)$ as we saw above and applying this to $\mfU_t$ we get $m_t,M_t$.
Finally we set therefore for our derived object: $M_t(\mfu,\cdot)=\CL[m_t(\mfu,\cdot)]$ which gives us averaging over $\mfu$ with the law of the $\U$-valued process now the $\U^D$-valued one namely $\CL [\mfU_t]$.

We claim and prove further below that this way we have constructed a realization of the conditioned distribution of $\lceil \wt \mfU_t \rceil(h)$ given $\wt \mfU_t$, with $(\mfU_t)_{t \geq 0}$ being the limit process arising from the individual based models.

We need next the conditional law given the $\U$-projection of the \emph{finite dimensional distributions}.
We have for that purpose to carry out the construction for the states at times $0 \leq t_1 < t_2 < \cdots < t_n \leq t$ jointly. 
This means we consider the $n$-states and the trunks formed at heights  $h_1,h_2,\cdots,h_n$ of these states.
Here we have $h_k< t_k$ and we make the convention to consider the case where $h_k > t_{k-1}$, for all $k$.

We have now to construct the conditional distribution of the marks given the states in $\U$ in the $n$-time points.
We consider first the $n$ trunks specified by the $(h_i)_{i=1,\cdots,n}$ i.e. $(t-h_1)$-trunks.
We begin the construction as before for each time $t_1,\cdots, t_n$ by using random walk increments for the time intervals $[0,t_1],[t_1,t_2]$, etc by taking \emph{$n$ independent copies} of our random walk collection, the point then is to match the pieces of the path by choosing the starting points of the random walks consistently.
However we observe that the $(t-h_1)$-trunk of the state at time $t_1$ is as ultrametric space embedded in the $(t-h_2)$-trunk of the state at time $t_2$, there is only pruning since some of the leaves of the first will have no descendants at time $t_2$.
But therefore there is no consistency problem, constructing the random walk path.

Similarly we proceed with the rescaled approximating branching random walk population and construct a representation as the one we used to define the \emph{infinite} population per site process.

We are now in the situation that we have constructed the finite dimensional distributions of an \emph{$\U^{D^+}$-valued Markov process}, the Markov property being a consequence of the Markov property of the $\U$-valued process we use and the independent increments property of the random walks used in the construction of the paths which are independent of the chosen genealogy.
This \emph{process is the candidate for the limit process}  arising from the \emph{sequence of individual based approximations}. \qed 
\medskip

A consequence of the construction above is that in order to establish the convergence of the individual based approximations it will suffice to show convergence of the marginal distributions for fixed time $t$, if we can establish the \emph{Feller property} for the involved processes.

Now we turn to the issue of the \emph{convergence} of the approximations with individual based processes.
Now we claim that our conditional distributions converge weakly, i.e. $\CL [m_n(\mfu,\cdot)]$ converge to $\CL[m(\mfu,\cdot)]$.
For the approximating individual based models we have however  the same $m(\mfu,\cdot)$ and $M(\mfu,\cdot)$ by construction.
The map $\mfu \rightarrow M(\mfu,\cdot)$ is continuous.

We can use for the approximating system now the same random walk system and we pick the same increments and obtain this way a coupling by matching the starting positions cleverly.
We also note that we can choose the genealogies all on one probability space such that we have convergence in the sense of ultrametric measure spaces for the trunks on which we condition (by conditioning on the state at time of the genealogy process).
Therefore the approximation by individual based models converges to a limit law in the sense of convergence of the finite-dimensional marginals.
This limit is given by our \emph{construction} we have given based on the realization of the $\U$-valued process and the collection of random walks.

We have to show that this process we have constructed solves the martingale problem \emph{based on the construction} not its property as limit.
The latter we would know so far only if in our polynomial either $\phi$ or $\chi$ is constant (recall the results we quoted in the beginning).

We have to show that the branching operator $(\Omega^{\rm grow}+\Omega^{\uparrow,\rm bra})$ and the mark evolution operator $\Omega^{\uparrow,\rm mig}$ act on a polynomial $\Phi^{n,\phi,\wh \chi}$ by changing $\phi$ resp. $\wh \chi$ \emph{separately} by replacing $\phi$ by a suitable $\phi^\prime$ and $\wh \chi$ by a suitable $\wh \chi^\prime$ as given in \eqref{e1265}, \eqref{e1179}.
Since the construction uses the $\U$-valued process as given element the action of the branching operator which is only lifted to $U \times V$ from $U$ and the mark evolution is constructed based on the collection of independent random walks which evolve according to the path process dynamic replacing $\chi$ by $\chi^\ast$ as calculated earlier and  which merge based on the underlying genealogy giving then the lifted, from $V$ to $U \times V$, term.

Precisely we proceed as follows.
We define again for $\mfw \in \U^V,\mfw \sim (\mfu,\mfv)$ with $\mfu \in \U$ and $\mfv \in \mcM(V)$ by $\mfw=[U \times V,r \otimes r_V,\nu], \mfu=[U,r,\mu]$ and $\mfv=\nu(U \times \cdot)$.
We define for $\Omega^{\uparrow, \rm bran}$ the lifted version $(\wt \Omega^{\uparrow,\rm bran} \Phi^{n,\phi,\chi})(\mfu,\mfv)=(\Omega^{\uparrow,\rm bran} \Phi^{n,\phi,\chi} (\cdot,\mfv)) (\mfu,\mfv)$ and similarly for $\Omega^{\uparrow,\rm grow}$.
On the other hand we have an operator describing the evolution of the measures on path driven by the path process which is defined on $\Phi^{n,\wh \chi}$ to describe the evolution of the historical process and which we define on $\Phi^{n,\phi,\wh \chi}$ now, to describe the mark evolution.
To derive the expression we have to use the construction we gave using the collection of random walks, which shows that $\Omega^{\uparrow,\rm anc} (\Phi^{n,\phi,\wh \chi} (\mfu,\cdot))(\mfu,\mfv)= \Omega^{\rm hist} (\Phi^{n,\phi,\wh \chi} (\mfu,\cdot))(\mfu,\mfv) = \suml_{k=1}^u \Phi^{n,\phi,\wh \chi^{\ast,k}} (\mfu,\mfv)$, where $\wh \chi^\prime$ is defined in \eqref{e2611}.

Altogether we have now constructed a solution to the martingale problem which concludes the Step 2 on the \textit{existence} of the solution.

Both Step 1 and Step 2 together prove the \emph{wellposedness} of the martingale process as claimed in Proposition~\ref{prop.1426} q.e.d.

\bibliography{ggr_GenBran100}

\begin{thebibliography}{CLUB09}

\bibitem[Ald90]{Aldous90}
D.~Aldous.
\newblock The random walk construction of uniform spanning trees and uniform
  labeled trees.
\newblock {\em SIAM Journal of discrete Mathematics}, 1990.

\bibitem[Ald91a]{Ald1991a}
D.~Aldous.
\newblock The continuum random tree {I}.
\newblock {\em Ann. Probab.}, 19:1--28, 1991.

\bibitem[Ald91b]{Ald1991}
D.~Aldous.
\newblock The continuum random tree. {II}: An overview.
\newblock {\em Stochastic analysis, Proc. Symp., Durham/UK 1990, Lond. Math.
  Soc. Lect. Note Ser.}, 167:23--70, 1991.

\bibitem[Ald93]{Ald1993}
D.~Aldous.
\newblock The continuum random tree {III}.
\newblock {\em Ann. Probab.}, 21:248--289, 1993.

\bibitem[AN72]{AN11}
K.~B. Athreya and P.~E. Ney.
\newblock {\em Branching processes}.
\newblock Springer-Verlag, New York-Heidelberg, 1972.
\newblock Die Grundlehren der mathematischen Wissenschaften, Band 196.

\bibitem[BLG00]{BLG00}
Jean Bertoin and Jean-Fran\c{c}ois Le~Gall.
\newblock The {B}olthausen-{S}znitman coalescent and the genealogy of
  continuous-state branching processes.
\newblock {\em Probab. Theory Related Fields}, 117(2):249--266, 2000.

\bibitem[BLG03]{BLG03}
Jean Bertoin and Jean-Fran{\c{c}}ois Le~Gall.
\newblock Stochastic flows associated to coalescent processes.
\newblock {\em Probab. Theory Related Fields}, 126(2):261--288, 2003.

\bibitem[BLG05]{BLG05}
Jean Bertoin and Jean-Fran{\c{c}}ois Le~Gall.
\newblock Stochastic flows associated to coalescent processes. {II}.
  {S}tochastic differential equations.
\newblock {\em Ann. Inst. H. Poincar\'e Probab. Statist.}, 41(3):307--333,
  2005.

\bibitem[CLUB09]{caballero2009}
M.E. Caballero, A.~Lambert, and G.~Uribe~Bravo.
\newblock Proof(s) of the {L}amperti representation of continuous-state
  branching processes.
\newblock {\em Probab. Surveys}, 6:62--89, 2009.

\bibitem[CRW91]{CRW1991}
B.~Chauvin, A.~Rouault, and A.~Wakolbinger.
\newblock Growing conditioned trees.
\newblock {\em Stochastic Process. Appl.}, 39(1):117--130, 1991.

\bibitem[Daw77]{D77}
D.~A. Dawson.
\newblock The critical measure diffusion process.
\newblock {\em Z. Wahrscheinlichkeitstheorie verw. Gebiete}, 40:125--145, 1977.

\bibitem[Daw93]{D93}
D.~A. Dawson.
\newblock {\em Measure-valued {M}arkov processes}, volume 1541 of {\em
  Springer, Lecture Notes in Mathematics}, pages 1--261.
\newblock 1993.

\bibitem[DG]{DG18evolution}
A.~Depperschmidt and A.~Greven.
\newblock Stochastic evolution of genealogies of spatial populations: state
  description, characterization of dynamics and properties.
\newblock {\em https://arxiv.org/abs/1807.03637v2 (2019), Lecture Notes Series,
  Institute for Mathematical Sciences, National University of Singapore: Vol.
  38 (2020)}.

\bibitem[DG96]{DG96}
D.~A. Dawson and A.~Greven.
\newblock Multiple space-time scale analysis for interacting branching models.
\newblock {\em Electron. J. Probab.}, 1:no. 14, 1--84, 1996.

\bibitem[DG03]{DG03}
D.~A. Dawson and A.~Greven.
\newblock State dependent multitype spatial branching processes and their
  longtime behavior.
\newblock {\em Electronic Journal of Probability}, 8(4):1--93, 2003.

\bibitem[DG14]{DG14}
Donald Dawson and Andreas Greven.
\newblock {\em Spatial Fleming-Viot models with mutation and selection.},
  volume 2092.
\newblock 01 2014.

\bibitem[DG19]{ggr_tvF14}
A.~Depperschmidt and A.~Greven.
\newblock Tree{-}valued {F}eller {D}iffusion.
\newblock {\em ArXive https://arxiv.org/abs/1904.02044v2}, submitted April
  2019, last revised August 2019.

\bibitem[DGP11]{DGP11}
A.~Depperschmidt, A.~Greven, and P.~Pfaffelhuber.
\newblock Marked metric measure spaces.
\newblock {\em Elect. Comm. in Probab.}, 16:174--188, 2011.

\bibitem[DGP12]{DGP12}
A.~Depperschmidt, A.~Greven, and P.~Pfaffelhuber.
\newblock Tree-valued {F}leming-{V}iot dynamics with selection.
\newblock {\em Annals of Appl. Probability}, 22(6):2560--2615, 2012.

\bibitem[DGP13]{DGP13}
A.~Depperschmidt, A.~Greven, and P.~Pfaffelhuber.
\newblock Path-properties of the tree-valued {F}leming-{V}iot dynamics.
\newblock {\em EJP}, 18:1--47, 2013.

\bibitem[DLG02]{DLG02}
T.~Duquesne and J.-F. Le~Gall.
\newblock Random trees, {L}\'evy processes and spatial branching processes.
\newblock {\em Ast\'erisque}, (281):vi+147, 2002.

\bibitem[DP91]{DP91}
D.A. Dawson and E.~Perkins.
\newblock Historical processes.
\newblock {\em Memoirs of the AMS}, 93(454), 1991.

\bibitem[EK86]{EK86}
S.N. Ethier and T.~Kurtz.
\newblock {\em Markov {P}rocesses. {C}haracterization and {C}onvergence}.
\newblock John Wiley, New York, 1986.
\newblock Paperback reprint in 2005.

\bibitem[EM14]{EM14}
Steven Evans and Ilya Molchanov.
\newblock The semigroup of metric measure spaces and its infinitely divisible
  probability measures.
\newblock {\em Transactions of the American Mathematical Society}, 369, 01
  2014.

\bibitem[EPW06]{EPW06}
S.~N. Evans, J.~Pitman, and A.~Winter.
\newblock Rayleigh processes, real trees, and root growth with re-grafting.
\newblock {\em Probab. Theory Related Fields}, 134(1):81--126, 2006.

\bibitem[Eth00]{Eth00}
Alison~M. Etheridge.
\newblock {\em An introduction to superprocesses}, volume~20 of {\em University
  Lecture Series}.
\newblock American Mathematical Society, Providence, RI, 2000.

\bibitem[Eva00]{Ev00}
S.N. Evans.
\newblock Kingman's coalescent as a random metric space.
\newblock In L.G Gorostiza and eds. B.G.~Ivanoff, editors, {\em Stochastic
  Models: Proceedings of the International Conference on Stochastic Models in
  Honour of Professor Donald A. Dawson, Ottawa, Canada, 1998}, volume~26.
  Canadian Mathematical Society / American Mathematical Society, 2000.

\bibitem[FG94]{FG94}
K.~Fleischmann and A.~Greven.
\newblock Diffusive clustering in an infinite system of hierarchically
  interacting diffusions.
\newblock {\em Probab. Theory Related Fields}, 98(4):517--566, 1994.

\bibitem[GGR]{infdiv}
A.~Greven, P.~Gl\"ode, and T.~Rippl.
\newblock Branching trees {I}: {C}oncatenation and infinite divisibility.
\newblock {\em Electron. J.Probab.}, 24 (2019)(52):55 pp., June.

\bibitem[GKW02]{GKW02}
A.~Greven, A.~Klenke, and A.~Wakolbinger.
\newblock Interacting diffusions in a random medium: comparison and longtime
  behavior.
\newblock {\em Stochastic Process. Appl.}, 98(1):23--41, 2002.

\bibitem[Gl{\"o}12]{Gl12}
P.~Gl{\"o}de.
\newblock {\em Dynamics of genealogical trees for autocatalytic branching
  processes}.
\newblock PhD thesis, Department Mathematik, Erlangen, Germany, 2012.
\newblock http://www.opus.ub.uni-erlangen.de/opus/volltexte/2013/4545/.

\bibitem[GLW05]{GLW05}
A.~Greven, V.~Limic, and A.~Winter.
\newblock Representation theorems for interacting {Moran} models, interacting
  {F}isher--{W}right diffusions and applications.
\newblock {\em Electronic J. of Probab.}, 10(39):1286--1358, 2005.

\bibitem[GPW09]{GPW09}
A.~Greven, P.~Pfaffelhuber, and A.~Winter.
\newblock Convergence in distribution of random metric measure spaces
  ({$\Lambda$}-coalescent measure trees).
\newblock {\em Probab. Theory Related Fields}, 145(1-2):285--322, 2009.

\bibitem[GPW13]{GPWmp13}
A.~Greven, P.~Pfaffelhuber, and A.~Winter.
\newblock Tree-valued resampling dynamics martingale problems and applications.
\newblock {\em Probab. Theory Related Fields}, 155(3-4):789--838, 2013.

\bibitem[GSW16]{GSW}
A.~Greven, R.~Sun, and A.~Winter.
\newblock Continuum space limit of the genealogies of interacting
  {F}leming-{V}iot processes on $\mathbb{Z}^1$.
\newblock {\em ArXive 1508.07169, EJP Vol.21 (2016), paper no. 58,64 pp.},
  2016.

\bibitem[GSW20]{GSWfoss}
A.~Greven, R.~Sun, and A.~Winter.
\newblock The evolving genealogy of fossils: Unique characterization by
  martingale problems and applications.
\newblock In preparation 2020.

\bibitem[Har02]{Harris2002}
T.~E. Harris.
\newblock {\em {T}he theory of branching processes}.
\newblock Dover Phoenix Editions. Dover Publications, Inc., Mineola, NY, 2002.
\newblock Corrected reprint of the 1963 original [Springer, Berlin; MR0163361
  (29 \#664)].

\bibitem[Lam07]{Lamb07}
A.~Lambert.
\newblock Quasi-stationary distributions and the continuous-state branching
  process conditioned to be never extinct.
\newblock {\em Electron. J. Probab.}, 12:no. 14, 420--446, 2007.

\bibitem[LG89]{LG89}
J.-F. Le~Gall.
\newblock Marches al\'eatoires, mouvement brownien et processus de branchement.
\newblock In {\em S\'eminaire de {P}robabilit\'es, {XXIII}}, volume 1372 of
  {\em Lecture Notes in Math.}, pages 258--274. Springer, Berlin, 1989.

\bibitem[LG93]{AldlGall}
J.-F. Le~Gall.
\newblock The uniform random tree in a {B}rownian excursion.
\newblock {\em Probab. Theory Related Fields}, 96(3):369--383, 1993.

\bibitem[LG99]{LG99}
J.-F. Le~Gall.
\newblock {\em Spatial branching processes, random snakes, and partial
  differential equations}.
\newblock Springer Science \& Business Media, 1999.

\bibitem[LJ91]{LJ91}
Y.~Le~Jan.
\newblock Superprocesses and projective limits of branching {M}arkov process.
\newblock {\em Ann. Inst. H. Poincar\'e Probab. Statist.}, 27(1):91--106, 1991.

\bibitem[LS81]{LS81}
T.~M. Liggett and F.~Spitzer.
\newblock Ergodic theorems for coupled random walks and other systems with
  locally interacting components.
\newblock {\em Z. Wahrsch. Verw. Gebiete}, 56(4):443--468, 1981.

\bibitem[Nev86]{Neveu86}
J.~Neveu.
\newblock Erasing a branching tree.
\newblock {\em Adv. in Appl. Probab.}, (suppl.):101--108, 1986.

\bibitem[NP89]{NP89}
J.~Neveu and J.~W. Pitman.
\newblock The branching process in a {B}rownian excursion.
\newblock In {\em S\'eminaire de {P}robabilit\'es, {XXIII}}, volume 1372 of
  {\em Lecture Notes in Math.}, pages 248--257. Springer, Berlin, 1989.

\end{thebibliography}
\bibliographystyle{alpha}

\appendix

\end{document}